\documentclass[oneside]{amsart}
\usepackage{geometry}
\usepackage{amsmath}
\usepackage{amssymb}
\usepackage{amsthm}
\usepackage[hidelinks]{hyperref}
\usepackage{tikz}
\usepackage{tikz-cd}
\usepackage{quiver}
\usepackage{extarrows}
\usepackage{enumitem}
\usepackage{dynkin-diagrams}

\newtheorem{theorem}{Theorem}[section]
\newtheorem{lemma}[theorem]{Lemma}
\newtheorem{proposition}[theorem]{Proposition}

\newtheorem{showtheorem}{Theorem}

\theoremstyle{definition}
\newtheorem{definition}[theorem]{Definition}
\newtheorem{example}[theorem]{Example}
\newtheorem{remark}[theorem]{Remark}

\providecommand{\SL}{\operatorname{SL}}
\providecommand{\GL}{\operatorname{GL}}
\providecommand{\PGL}{\operatorname{PGL}}
\providecommand{\Stab}{\operatorname{Stab}}

\providecommand{\Spec}{\operatorname{Spec}}
\providecommand{\Proj}{\operatorname{Proj}}
\providecommand{\Hilb}{\operatorname{Hilb}}
\providecommand{\len}{\operatorname{len}}
\providecommand{\Bl}{\operatorname{Bl}}
\providecommand{\Coh}{\operatorname{Coh}}
\providecommand{\Gr}{\operatorname{Gr}}
\providecommand{\codim}{\operatorname{codim}}

\providecommand{\Hom}{\operatorname{Hom}}
\providecommand{\coker}{\operatorname{coker}}
\providecommand{\rk}{\operatorname{rk}}
\providecommand{\Span}{\operatorname{Span}}

\providecommand{\Sym}{\operatorname{Sym}}
\providecommand{\Aut}{\operatorname{Aut}}
\providecommand{\id}{\operatorname{id}}
\providecommand{\wt}{\operatorname{wt}}

\newlength{\Tspace}
\usepackage[backend=biber, style=alphabetic, sorting=nyt, maxbibnames=9]{biblatex}
\title{The GIT of $4\times3\times3$ tensors (determinantal cubic surfaces)}
\author{Charlotte Sherratt}
\address{Mathematical Sciences Institute, Australian National University,  Ngunnawal and Ngambri Country, Canberra, Australia}

\begin{document}

\begin{abstract}
    A general $4\times3\times3$ tensor defines three determinantal varieties in $\mathbb{P}^3$ and $\mathbb{P}^2$ and $\mathbb{P}^2$ where it drops rank: a cubic surface and two length $6$ subschemes. We study the Geometric Invariant Theory (GIT) quotient of $\mathbb{P}(\mathbb{C}^4\otimes\mathbb{C}^3\otimes\mathbb{C}^3)$ by $\SL(4)\times\SL(3)\times\SL(3)$, a moduli space of linear determinantal representations of cubic surfaces. We give a geometric description of the (semi)stable locus and some of the geometry of this GIT quotient. This GIT quotient admits a surjective rational map to the GIT moduli space of cubic surfaces, undefined at only one point.
\end{abstract}

\maketitle

\section{Introduction}

The determinant of a $d\times d$ matrix of linear forms in $n$ variables is a homogeneous degree $d$ polynomial. Going backwards, one can ask whether a given homogeneous degree $d$ polynomial in $n$ variables can be written as the determinant of a $d\times d$ matrix of linear forms (we call this matrix a \emph{determinantal representation}). A quick dimension count shows that for most values of $(n,d)$ a general polynomial has either zero or infinitely many determinantal representations (apart from the trivial case of $d\ne1$). This is except for when $(n,d)=(3,2),(4,2),(4,3)$, in which case a general polynomial has \emph{finitely many} determinantal representations (up to $\GL(d)\times\GL(d)$-equivalence). In the case that $(n,d)=(4,3)$, that is of cubic surfaces in $\mathbb{P}^3$, the equation defining a sufficiently general (smooth) cubic surface can be written as the determinant of a $3\times3$ matrix of linear forms \emph{in exactly $72$ ways} (up to $\GL(3)\times\GL(3)$-equivalence). A $3\times 3$ matrix of linear forms in $4$ variables is adjoint to a $4\times 3\times 3$ tensor. A $4\times3\times3$ tensor corresponds to a cubic surface and two length $6$ subschemes of $\mathbb{P}^2$. We briefly describe this correspondence. 

Let $U$ and $V$ and $W$ be complex vector spaces with $\dim U=4$ and ${\dim V=\dim W=3}$. We let $\varphi\in U^\vee\otimes V^\vee\otimes W^\vee$ be a sufficiently general tensor. The tensor $\varphi$ has two adjoint matrices $\varphi_U\in U^\vee\otimes\Hom(W,V^\vee)$ and $\varphi_V\in V^\vee\otimes\Hom(W,U^\vee)$. The determinant of the $3\times3$ matrix $\varphi_U$ is a cubic polynomial $\det(\varphi_U)$, which cuts out a smooth cubic surface $S_\varphi\subset\mathbb{P}U$. The $4\times3$ matrix $\varphi_V$ presents a free resolution
\begin{equation}
\label{eqn:intro-IX-exseq}
    0\longrightarrow W\otimes\mathcal{O}_{\mathbb{P}V}(-4)\xlongrightarrow{\varphi_V}U^\vee\otimes\mathcal{O}_{\mathbb{P}V}(-3)\longrightarrow\mathcal{I}_{X_\varphi}\longrightarrow0
\end{equation}
of an ideal sheaf $\mathcal{I}_{X_\varphi}\subset\mathcal{O}_{\mathbb{P}V}$, which cuts out a length $6$ subscheme $X_\varphi\subset\mathbb{P}V$. Twisting and taking cohomology of \eqref{eqn:intro-IX-exseq} gives an identification of $H^0(\mathbb{P}V,\mathcal{I}_{X_\varphi}(3))\cong U^\vee$; this linear system induces a rational map $f:\mathbb{P}V\dashrightarrow\mathbb{P}U=\vert\mathcal{I}_{X_\varphi}(3)\vert$. The image of $f$ is contained in $S_\varphi$ and the map $f:\mathbb{P}V\dashrightarrow S_\varphi$ is birational. Its inverse extends to a morphism $f^{-1}:S_\varphi\to\mathbb{P}V$ which identifies the surface $S_\varphi$ as the blowup of $\mathbb{P}V$ at $6$ points $X_\varphi\subset\mathbb{P}V$. The inverse map $f^{-1}$ is also induced from the tensor $\varphi$: we have an exact sequence
$$0\longrightarrow W\otimes\mathcal{O}_{S_\varphi}(-1)\xlongrightarrow{\varphi_U}V^\vee\otimes\mathcal{O}_{S_\varphi}\longrightarrow\coker(\varphi_U)\longrightarrow0.$$
Taking cohomology identifies $H^0(S_\varphi,\coker(\varphi_U))\cong V^\vee$ and $\coker(\varphi_U)=(f^{-1})^\ast\mathcal{O}_{\mathbb{P}V}(1)$. 

Conversely, let $S\subset\mathbb{P}U$ be a smooth cubic surface and let $g:S\to\mathbb{P}V$ a birational morphism. The map $g$ is the blowup of $\mathbb{P}V$ at $6$ points $X\subset\mathbb{P}V$. By the Hilbert-Burch theorem, the ideal sheaf $\mathcal{I}_X\subset\mathcal{O}_{\mathbb{P}V}$ cutting out $X\subset\mathbb{P}V$ has a free resolution of the form of \eqref{eqn:intro-IX-exseq} that is presented by a $4\times3$ matrix $\varphi_V\in V^\vee\otimes\Hom(W,U^\vee)$ of linear forms in $V^\vee$. The determinant of its adjoint $3\times3$ matrix $\varphi_U\in U^\vee\otimes \Hom(W,V^\vee)$ cuts out the surface $S\subset\mathbb{P}U$ (up to a $\GL(U)$-action). The $72$ different birational morphisms $g:S\to\mathbb{P}V$ are in bijection with the $72$ different ways to write the equation cutting out $S\subset\mathbb{P}U$ as the determinant of a $3\times3$ matrix valued in $U^\vee$.

The other adjoint matrix $\varphi_W\in W^\vee\otimes\Hom(V,U^\vee)$ presents a free resolution 
\begin{equation}
\label{eqn:intro-IY-exseq}
    0\longrightarrow V\otimes\mathcal{O}_{\mathbb{P}W}(-4)\xlongrightarrow{\varphi_W}U^\vee\otimes\mathcal{O}_{\mathbb{P}W}(-3)\longrightarrow\mathcal{I}_{Y_\varphi}\longrightarrow0
\end{equation}
of an ideal sheaf $\mathcal{I}_{Y_\varphi}\subset\mathcal{O}_{\mathbb{P}W}$ in the same way as \eqref{eqn:intro-IY-exseq}, which cuts out $6$ points $Y_\varphi\subset\mathbb{P}W$. The surface $S_\varphi$ is also isomorphic to the blowup $\Bl_{Y_\varphi}\mathbb{P}W$. The pair of point configurations $X_\varphi\subset\mathbb{P}V$ and $Y_\varphi\subset\mathbb{P}W$ are Gale-dual, in the sense of \cite{eisenbud-popescu00}. This correspondence induces an isomorphism between the moduli spaces of:
\begin{enumerate}[label=(\alph*)]
    \item triples $([S],[X],[Y])\in\Hilb_{\textnormal{cubic surfaces}}(\mathbb{P}U)\times\Hilb_6(\mathbb{P}V)\times\Hilb_6(\mathbb{P}W)$ where $X\subset\mathbb{P}V$ and $Y\subset\mathbb{P}W$ are a pair of Gale-dual point configurations and $S$ is a smooth cubic surface isomorphic to both $\Bl_X\mathbb{P}V$ and $\Bl_Y\mathbb{P}W$; and
    \item tensors $\varphi\in\mathbb{P}(U^\vee\otimes V^\vee\otimes W^\vee)$ where $\det(\varphi_U)$ cuts out a smooth cubic surface $S_\varphi\subset\mathbb{P}U$.
\end{enumerate}
This correspondence is extremely well studied. In particular, it extends to triples $(S,X,Y)$ consisting of a cubic surface $S\subset\mathbb{P}U$ with at worst canonical singularities and an associated pair of Gale-dual length $6$ curvilinear subschemes $X\subset\mathbb{P}V$ and $Y\subset\mathbb{P}W$ to the surface $S$, provided that the subschemes $X$ and $Y$ are also not contained in a conic.

Quotienting by $\SL(V)\times\SL(W)$ gives a compactification ${\mathbb{P}(U^\vee\otimes V^\vee\otimes W^\vee)//(\SL(V)\times\SL(W))}$ of the moduli space of determinantal representations of cubic surfaces. This GIT quotient is closely related to the Hilbert schemes of twisted cubics inside cubic surfaces and has been extensively studied in \cite{twisted-cubics_lehn-lehn-sorger-straten}. 

Quotienting by $\SL(U)\times\SL(W)$ gives a compactification ${\mathbb{P}(U^\vee\otimes V^\vee\otimes W^\vee)//(\SL(U)\times\SL(W))}$ of the moduli space of $6$ unordered points in $\mathbb{P}V$. By \cite{abch}, this GIT quotient is the final minimal model of the Hilbert scheme $\Hilb_6(\mathbb{P}V)$ and the birational contraction 
$$\Hilb_6(\mathbb{P}V)\dashrightarrow\mathbb{P}(U^\vee\otimes V^\vee\otimes W^\vee)//(\SL(U)\times\SL(W))$$
contracts the divisor parametrising length $6$ subschemes $X\subset\mathbb{P}V$ contained in a conic.

This paper is concerned with the triple quotient of $\mathbb{P}(U^\vee\otimes V^\vee\otimes W^\vee)$ by the group \linebreak $G:=\SL(U)\times\SL(V)\times\SL(W)$; the GIT compactification of the moduli problems of (a) and (b) after forgetting the bases of all three vector spaces $U$ and $V$ and $W$. The rational map 
$$\det:\mathbb{P}(U^\vee\otimes V^\vee\otimes W^\vee)\dashrightarrow\vert\mathcal{O}_{\mathbb{P}U}(3)\vert$$
sending a $4\times3\times3$ tensor to its determinant cubic descends to a rational map of GIT quotients
$${\det}_{UVW}:\mathbb{P}(U^\vee\otimes V^\vee\otimes W^\vee)//G\dashrightarrow\vert\mathcal{O}_{\mathbb{P}U}(3)\vert//\SL(U).$$

The main results of this paper are the following:
\begin{showtheorem}
\label{thm:showtheorem}
    A point $\varphi\in\mathbb{P}(U^\vee\otimes V^\vee\otimes W^\vee)$ is $G$-stable if and only if the associated cubic surface $\det(\varphi)\in\vert\mathcal{O}_{\mathbb{P}U}(3)\vert$ is $\SL(U)$-stable. The point $\varphi$ is $G$-semistable if and only if either:
    \begin{itemize}
        \item the cubic surface $\det(\varphi)$ is $\SL(U)$-semistable; or
        \item the cubic surface $\det(\varphi)$ has either one $A_3$ singularity or both one $A_1$ and one $A_3$ singularity (and is otherwise smooth) and the associated subscheme $X_\varphi\subset\mathbb{P}V$ is supported at only two points; or
        \item the cubic surface $\det(\varphi)$ is reducible and the associated subscheme $X_\varphi\subset\mathbb{P}V$ is a smooth conic.
    \end{itemize}
    The quotient $\mathbb{P}(U^\vee\otimes V^\vee\otimes W^\vee)^{G-sss}//G$ of the strictly $G$-semistable locus is the disjoint union of a singleton and a subscheme whose normalisation is $\mathbb{P}^1$. The rational map $\det_{UVW}$ is undefined only at one point.
\end{showtheorem}

\subsection{A brief history}

Much of the theory of determinantal representations of normal cubic surfaces was known by classical geometers such as Segre, who proved in 1906 that every normal cubic surface without an $E_6$ singularity has a determinantal representation \cite{segre}. Thrall then published a classification of $4\times3\times3$ tensors in 1941 \cite{thrall}. This was supplemented by Ng in 2002 \cite{ng}, who wrote down an explicit matrix presentation for almost all $4\times3\times3$ tensors.

In 1989, Gimigliano introduced a sheaf-theoretic approach to smooth determinantal surfaces \cite{gimigliano} by viewing the various adjoint matrices as maps of vector bundles on $\mathbb{P}^2$ and $\mathbb{P}^3$. This was extended by Dolgachev and Kapranov in \cite{dolgachev-kapranov}, who described the geometry of smooth determinantal cubic surfaces. Dolgachev gives a detailed description of an extension of this sheaf-theoretic approach to cubic surfaces with at worst canonical singularities in \cite[\S9.3]{dolgachev}. Eisenbud and Popescu's two papers on the Gale transform in the late 90s \cite{eisenbud-popescu99,eisenbud-popescu00} give a modern treatment of the connections between the Gale transform, free resolutions of ideals of points and determinantal varieties. The special case of $4\times3\times3$ tensors is a very well-behaved example of this connection.

The connection between determinantal representations of cubic surfaces and twisted cubics was studied by Ellingsrud, Piene and Str{\o}mme in 1987 \cite{ellingsrud-piene-stromme} and then more recently by Lehn, Lehn, Sorger and Straten in \cite{twisted-cubics_lehn-lehn-sorger-straten}. The latter paper uses the geometry of the GIT quotient ${\mathbb{P}(U^\vee\otimes V^\vee\otimes W^\vee)//(\SL(V)\times\SL(W))}$ to describe the moduli space of twisted cubics on a smooth cubic fivefold. 

\subsection{Structure of the paper}

Section~\ref{sec:section-2} introduces important notation and recounts some known facts about determinantal representations. Most of these results, with more details and proofs, can be found in either \cite[\S9]{dolgachev} or \cite{ng} or \cite{twisted-cubics_lehn-lehn-sorger-straten}. 

Section~\ref{sec:partial-quotients} uses Gale-duality to construct the birational contraction
$$\Hilb_6(\mathbb{P}^2)\dashrightarrow\mathbb{P}(U^\vee\otimes V^\vee\otimes W^\vee)//(\SL(V)\times\SL(W)),$$
and we describe where this map is defined.

Section~\ref{sec:full-quotient} is devoted to understanding the full GIT quotient by $\SL(U)\times\SL(V)\times\SL(W)$ and proving Theorem A. We give a stratification of $\mathbb{P}(U^\vee\otimes V^\vee\otimes W^\vee)$, which allows us to describe the $G$-semistable and $G$-stable loci, and the identification of $G$-semistable orbits in the GIT quotient. 

\subsection{Conventions and organisation}

We are working over the complex numbers. We use the subspace convention for the projective space $\mathbb{P}A$ associated to a vector space $A$, i.e. $\mathbb{P}A$ is the projective space of $1$-dimensional subspaces of $A$; and $\mathbb{P}A=\Proj(\mathbb{C}[A^\vee]_\bullet)$ and ${H^0(\mathbb{P}A,\mathcal{O}_{\mathbb{P}A}(1))=A^\vee}$. We use the same convention for the projectivisation of a vector bundle and do not distinguish between a vector bundle and its associated sheaf of sections. When $A$ is a vector space we will always let $\SL(A)$ act on $A$ with the standard left action $a\mapsto g\cdot a$ and on $A^\vee$ by the left action $\alpha\mapsto\alpha\circ g^{-1}$ of precomposition. For convenience of notation, we will not distinguish between non-zero vectors $a\in A$ and their projective equivalence class $\overline{a}\in\mathbb{P}A$.

All schemes are of finite type over $\mathbb{C}$. Unless otherwise specified, all tensor products are taken either over $\mathbb{C}$ or over the structure sheaf of the ambient scheme. All \emph{points} are $\mathbb{C}$-points unless otherwise specified and by a \emph{finite scheme} we mean a scheme with finite structure morphism to $\Spec\mathbb{C}$. When $X$ is a finite scheme, the \emph{length} of a point $x\in X$ is the length of the non-reduced subscheme supported at $x$. When $X\subset Y$ is a closed subscheme we will write $\mathcal{I}_X\subset\mathcal O_Y$ for the ideal sheaf that cuts out $X\subset Y$. When $\mathcal{F}$ is a coherent sheaf on a scheme $X$ we will write $\vert\mathcal{F}\vert=\mathbb{P}(H^0(X,\mathcal{F}))$ for the complete linear system. 

\subsection{Acknowledgements}

This paper is an extension of the author's undergraduate thesis at the Australian National University. I would like to thank my advisor Anand Deopurkar for introducing me to this topic, for countless meetings and helpful discussions, and for teaching me most of what I know about algebraic geometry. This research was partially supported by the Australian Research Council Discovery Project DP240101084. 

\tableofcontents

\section{Background}
\label{sec:section-2}

Recall that $U$ and $V$ and $W$ are vector spaces with $\dim U=4$ and $\dim V=\dim W=3$. The vector spaces $V$ and $W$ play completely symmetric roles in this paper. Whenever we fix bases of $U^\vee$ and $V^\vee$ and $W^\vee$ we will write the basis vectors as $u_0,u_1,u_2,u_3$ and $v_0,v_1,v_2$ and $w_0,w_1,w_2$. 

As outlined in the introduction, a general tensor $\varphi\in\mathbb{P}(U^\vee\otimes V^\vee\otimes W^\vee)$ is adjoint to (a) a matrix $\varphi_U\in U^\vee\otimes\Hom(W,V^\vee)$ whose determinant $\det(\varphi_U)$ cuts out a smooth cubic surface $S_\varphi\subset\mathbb{P}U$ and (b) the presentation matrix $\varphi_V$ of a free resolution of an ideal sheaf $\mathcal{I}_{X_\varphi}\subset\mathcal{O}_{\mathbb{P}V}$ that cuts out a length $6$ subscheme $X_\varphi\subset\mathbb{P}V$. The cubic surface $S_\varphi$ is isomorphic to the blowup of $\mathbb{P}V$ at $X_\varphi$, the morphism $S_\varphi\to\mathbb{P}V$ is induced by a natural identification of the linear system $\vert\coker(\varphi_U)\vert$ with $\mathbb{P}V$. This correspondence is extremely well-studied. For brevity of exposition we refer the reader to \cite{dolgachev,ng,twisted-cubics_lehn-lehn-sorger-straten} for further details on Subsections~\ref{subsec:cubics-background} and \ref{subsec:points-background}. 

The goal of Section~\ref{sec:section-2} is to introduce necessary notation and sketch part of this correspondence, including its generalisation to cubic surfaces with at worst canonical singularities. We will only discuss the connection between length $6$ subschemes $X\subset\mathbb{P}V$ and determinantal representations of cubic surfaces $S\subset\mathbb{P}U$, we refer the reader to the papers \cite{ellingsrud-piene-stromme,twisted-cubics_lehn-lehn-sorger-straten} for the connection between (generalised) twisted cubics and determinantal representations of cubic surfaces.

\subsection{Cubic surfaces}
\label{subsec:cubics-background}

We start with some background on cubic surfaces. A result of Schläfli \cite{schlafli-classification} is the following, a modern proof can be found in Section~2 of \cite{bruce-wall}.

\begin{proposition}
\label{prop:types-of-cubics}
    Every cubic surface $S$ has one of the following configurations of singularities:
    \begin{enumerate}[label=(\roman*)]
        \item $S$ is smooth; or
        \item $S$ is singular with only canonical singularities (also known as du Val singularities, rational double points), equivalently $S$ is a singular, normal cubic surface with finitely many lines;
        \item $S$ is reducible; or
        \item $S$ is a cone over a cubic curve; or
        \item $S$ is non-normal, not a cone, and irreducible.
    \end{enumerate}
\end{proposition}

\begin{definition}
    A cubic surface $S\subset\mathbb{P}U$ is said to be \emph{canonical} if $S$ is smooth or if all of the singularities of $S$ are canonical singularities. 
\end{definition}

We are predominantly interested in canonical and reducible cubic surfaces in this paper; we will see that a necessary condition for a tensor $\varphi\in\mathbb{P}(U^\vee\otimes V^\vee\otimes W^\vee)$ to be $\SL(U)\times\SL(V)\times\SL(W)$-semistable is that its associated cubic surface $S_\varphi$ is either canonical or reducible. 

Every cubic surface $S\subset\mathbb{P}U$ is anticanonically embedded. As such, we won't distinguish between an abstract cubic surface $S$ and a cubic surface $S$ equipped with an embedding $S\subset\vert\omega_S^{-1}\vert\cong\mathbb{P}U$ unless the choice of identification $\vert\omega_S^{-1}\vert\cong\mathbb{P}U$ is important. Any isomorphism $\sigma:S\to S'$ of two cubic surfaces $S,S'\subset\mathbb{P}U$ induces an isomorphism of anticanonical divisors, which extends to an isomorphism of the linear systems $\vert\omega_S^{-1}\vert$ and $\vert\omega_{S'}^{-1}\vert$, inducing an isomorphism between their ambient projective spaces. As a result, two cubic surfaces $S,S'\subset\mathbb{P}U$ are isomorphic as schemes if and only if they are isomorphic \emph{as subschemes} of $\mathbb{P}U$.

Cubic surfaces $S\subset\mathbb{P}U$ are parametrised by the complete linear system $\vert\mathcal{O}_{\mathbb{P}U}(3)\vert=\mathbb{P}(\Sym^3(U^\vee))$. This linear system is also the Hilbert scheme $\Hilb_{\tfrac{3}{2}t^2+\tfrac{9}{2}t+1}(\mathbb{P}U)$ of cubic surfaces. We will write $[S]\in\vert\mathcal{O}_{\mathbb{P}U}(3)\vert$ for the point of the Hilbert scheme associated to the surface $S\subset\mathbb{P}U$. 

Let $S$ be a canonical cubic surface and let $r:\widetilde{S}\to S$ be its minimal resolution of singularities. The intersection graph $F_S$ of the $(-2)$-curves on $\widetilde{S}$ determines the singularity configuration on $S$. The graph $F_S$ is a strict subdiagram of the extended Dynkin diagram $$\widetilde{E}_6=\dynkin[extended]{E}{6}.$$

To be precise, a Dynkin subdiagram $F'$ of a Dynkin diagram $F$ consists of a subset $F'_v\subset F_v$ of the vertices of $F$ along with all the edges that were originally present between these vertices inside $F$. As an example, the diagram $\dynkin{A}{1}\;\dynkin{A}{1}$ is not a subdiagram of $\dynkin{A}{2}$ even though it is a subgraph. All possible strict Dynkin subdiagrams $F\subsetneq\widetilde{E}_6$ arise as a configuration of singularities on a canonical cubic surface; these are $A_1$, $2A_1$, $A_2$, $3A_1$, $A_1A_2$, $A_3$, $4A_1$, $2A_1A_2$, $A_1A_3$, $2A_2$, $A_4$, $D_4$, $2A_1A_3$, $A_12A_2$, $A_1A_4$, $A_5$, $D_5$, $A_1A_5$, $3A_2$, $E_6$. 

\begin{definition}
    We call $F_S$ the \emph{associated Dynkin diagram} to the surface $S$. We say a cubic surface \emph{has $F$ singularities} if its associated Dynkin diagram is isomorphic to $F$ (as a graph). We say a cubic surface $S$ has \emph{at worst $F$ singularities} if its associated Dynkin diagram is isomorphic (as a graph) to a Dynkin subdiagram of $F$. 

    For each strict Dynkin subdiagram $F\subsetneq\widetilde{E}_6$ we let $\mathcal{S}(F)\subset\vert\mathcal{O}_{\mathbb{P}U}(3)\vert$ denote the set of cubic surfaces $S\subset\mathbb{P}U$ with $F$ singularities.
\end{definition}

The subsets $\mathcal{S}(F)\subset\vert\mathcal{O}_{\mathbb{P}U}(3)\vert$ are connected, locally closed subvarieties and $\mathcal{S}(F')\subset\overline{\mathcal{S}(F)}$ if and only if $F$ is a Dynkin subdiagram of $F'$. 

\begin{proposition}
\label{prop:S(F)-dimensions}
    The subvariety $\mathcal{S}(\emptyset)\subset\vert\mathcal{O}_{\mathbb{P}U}(3)\vert$ parametrising smooth cubic surfaces is dense in $\vert\mathcal{O}_{\mathbb{P}U}(3)\vert$ and every smooth cubic surface has a finite automorphism group ($\PGL(U)$-stabiliser). More generally:
    \begin{itemize}
        \item Let $F\subsetneq\widetilde{E}_6$ be a Dyknin subdiagram where $F\ne 2A_2$ and $F\ne D_4$. Then $\mathcal{S}(F)\subset\vert\mathcal{O}_{\mathbb{P}U}(3)\vert$ has codimension $\vert F\vert$ and every surface $[S]\in\mathcal{S}(F)$ has a $\min\{0,\vert F\vert-4\}$-dimensional automorphism group. If additionally $\vert F\vert\geq 4$ then all surfaces with $F$ singularities are isomorphic and $\mathcal{S}(F)$ is a $\PGL(U)$-orbit.
        \item Let $F=2A_2$. Then the subvariety $\mathcal{S}(2A_2)\subset\vert\mathcal{O}_{\mathbb{P}U}(3)\vert$ has codimension $3$ and every surface $[S]\in\mathcal{S}(2A_2)$ has a $1$-dimensional automorphism group.
        \item Let $F=D_4$. Then the subvariety $\mathcal{S}(D_4)\subset\vert\mathcal{O}_{\mathbb{P}U}(3)\vert$ has codimension $4$. There are two types of $D_4$ cubic surfaces up to isomorphism, we call them $S_{D_4(a)}$ and $S_{D_4(b)}$. We have $\Aut(S_{D_4(a)})=\mathfrak{S}_3$ and $\Aut(S_{D_4(b)})=\mathbb{G}_m\ltimes\mathfrak{S}_3$. The orbit $\PGL(U)\cdot[S_{D_4(a)}]\subset\mathcal{S}(D_4)$ is dense. 
    \end{itemize}
\end{proposition}
The automorphism groups of cubic surfaces with $F$ singularities satisfying both $\vert F\vert\geq 4$ and $F\ne 2A_2$ can be found in \cite{sakamaki}. Section~8 of \cite{chelstov-prokhorov} lists all positive-dimensional automorphism groups of canonical cubic surfaces. 

\subsection{Point configurations in $\mathbb{P}^2$}
\label{subsec:points-background}

Let $X\subset\mathbb{P}V\cong\mathbb{P}^2$ be a configuration of $6$ points in general position with respect to plane curves (no $3$ points in $X$ are collinear and $X$ is not contained in a conic). The blowup $b:S=\Bl_X\mathbb{P}V\to\mathbb{P}V$ is a smooth cubic surface with anticanonical divisor $\omega_S^{-1}=b^\ast\mathcal{I}_{\mathbb{P}V}(3)$. Conversely, let $S'$ be a smooth cubic surface. Every birational morphism $S'\to\mathbb{P}V$ is the blowup of $\mathbb{P}V$ at $6$ points $X\subset\mathbb{P}V$ which are in general position with respect to plane curves. There are exactly $72$ birational morphisms $S'\to\mathbb{P}V$, up to $\PGL(V)$-equivalence on the target. 

Canonical cubic surfaces are similarly described by certain curvilinear length $6$ subschemes of $\mathbb{P}V$. We introduce some notation, closely following Subsection~2.5 of \cite{counting-cubics_deopurkar-patel-tseng} (note that our definition of canonical is their definition of admissible). We are also interested in length $6$ subschemes ${Y\subset\mathbb{P}W\cong\mathbb{P}^2}$. We will mainly write definitions and theorems in terms of $X\subset\mathbb{P}V$, in both the following definition and the rest of the paper the reader should replace all $V$s with $W$s and all $X$s with $Y$s to apply the appropriate statements to subschemes of $\mathbb{P}W$; the two $3$-dimensional vector spaces $V$ and $W$ play completely symmetric roles. 
\begin{definition}
\label{defn:canonical-subschemes}
    A length $6$ subscheme $X\subset\mathbb{P}V$ is said to be \emph{canonical} the following conditions are all satisfied:
    \begin{enumerate}[label=(\roman*)]
        \item $X$ is curvilinear; and
        \item $h^0(\mathbb{P}V, \mathcal{I}_X(3))=4$; and
        \item the twisted ideal sheaf $\mathcal{I}_X(3)$ is generated by its global sections.
    \end{enumerate}
    A canonical subscheme $X\subset\mathbb{P}V$ is said to be \emph{admissible} if additionally:
    \begin{enumerate}[resume*]
        \item $X$ is not contained in a conic.
    \end{enumerate}
    An admissible subscheme $X\subset\mathbb{P}V$ is said to be a \emph{smooth point configuration} if additionally:
    \begin{enumerate}[resume*]
        \item $X$ is reduced; and
        \item $X$ contains no length $3$ collinear subscheme.
    \end{enumerate}
\end{definition}

Smooth point configurations $X\subset\mathbb{P}V$ are simply sets of $6$ points in general position with respect to plane curves, and are associated to smooth cubic surfaces. For canonical subschemes we have the following result:

\begin{proposition}
\label{prop:facts-about-blowing-up}
    Let $X\subset\mathbb{P}V$ be a canonical length $6$ subscheme. Let $S=\Bl_X\mathbb{P}V$ and let ${b:S=\Bl_X\mathbb{P}V\to\mathbb{P}V}$ be the blowup. The following facts are true:
    \begin{enumerate}[label=(\roman*)]
        \item $S$ has at worst $A_n$ singularities.
        \item The linear system $H^0(S,\omega_S^{-1})$ is $4$-dimensional and basepoint-free.
        \item The image of the map $f:S\to\mathbb{P}U\cong\vert\omega_S^{-1}\vert$ induced by $\omega_S^{-1}=b^\ast\mathcal{I}_X(3)$ is a cubic surface $S_X\subset\mathbb{P}U$, the morphism $S\to S_X$ is birational.
        \item The morphism $S\to S_X$ is an isomorphism away from the following curves, which it contracts:
        \begin{itemize}
            \item the strict transform $b^\ast(C)$ of any conic $C\subset\mathbb{P}V$ containing $X$, and
            \item the strict transform $b^\ast(L)$ of any line $L\subset\mathbb{P}V$ such that $\len(X\cap L)=3$. 
        \end{itemize}
        \item $S_X$ is normal and has only canonical singularities. 
    \end{enumerate}
    Conversely, let $S'$ be a canonical cubic surface and let $\widetilde{S'}\to S'$ be the minimal resolution of singularities. Any birational morphism $\widetilde{S'}\to\mathbb{P}V$ factors through the blowup $\Bl_X\mathbb{P}V\to\mathbb{P}V$ of $\mathbb{P}V$ at a canonical subscheme $X\subset\mathbb{P}V$. The surface $S'$ is isomorphic to $S_X$ as above. 
\end{proposition}
\begin{proof}
    The details of (i-iii) and (v) can be found in Proposition~2.5.2 of \cite{counting-cubics_deopurkar-patel-tseng}. We have added (iv), which is easy to verify using basic facts about linear systems. The case when $S_X$ is smooth can be found in detail in Proposition~IV.9 of \cite{beauville}. 
    
    The details of the converse statement can be found in Chapter~9 of \cite{dolgachev}, which uses the language of bubble cycles instead of length $6$ subschemes. 
\end{proof}

The relationship between a canonical subscheme $X\subset\mathbb{P}V$, the blowup $b:\Bl_X\mathbb{P}V\to\mathbb{P}V$, the associated cubic surface $S_X$, and the minimal resolution of singularities $r:\widetilde{S}\to S_X$ can be summarised in the following commutative diagram where all morphisms are birational:
\begin{equation}
    \label{eqn:associated-surface-diagram}
    \begin{tikzcd}
        & {\widetilde{S}} & \\
    	{\Bl_X\mathbb{P}V=S} & {S_X} & {\vert\omega^{-1}_S\vert\cong\mathbb{P}^3} \\
    	& {\mathbb{P}V}
    	\arrow[from=1-2, to=2-1]
    	\arrow["r", from=1-2, to=2-2]
    	\arrow["f", from=2-1, to=2-2]
    	\arrow["b"', from=2-1, to=3-2]
    	\arrow[hook, from=2-2, to=2-3]
    	\arrow[dashed, from=2-2, to=3-2]
    \end{tikzcd}
\end{equation}
The birational map $f^{-1}\circ r:\widetilde{S}\dashrightarrow\Bl_X\mathbb{P}V$ extends to a morphism; this morphism is the minimal resolution of singularities for the blowup $\Bl_X\mathbb{P}V$. The composition $b\circ f^{-1}\circ r:\widetilde{S}\to\mathbb{P}V$ is the composition of $6$ blowups of $\mathbb{P}V$ (blowing up the bubble cycle associated to the subscheme $X\subset\mathbb{P}V$). The birational map $f\circ b^{-1}:\mathbb{P}V\dashrightarrow S_X$ is given by the linear system $\vert\mathcal{I}_X(3)\vert$ of cubics through $X$. If $X\subset\mathbb{P}V$ is also admissible (i.e. not contained in a conic) then the pullback $(b\circ f^{-1})^\ast\mathcal{O}_{\mathbb{P}V}(1)$ under the birational map $b\circ f^{-1}:S_X\dashrightarrow\mathbb{P}V$ is a line bundle on $S_X$ associated to a twisted cubic.

\begin{definition}
    Let $X\subset\mathbb{P}V$ be a canonical subscheme. We define \emph{the cubic surface associated to} $X$, denoted by $S_X$, as the image of the map $\Bl_X\mathbb{P}V\to\mathbb{P}U$ induced by the linear system $\vert\omega^{-1}_{\Bl_X\mathbb{P}V}\vert$. 

    Conversely, if $S$ is a cubic surface and $X\subset\mathbb{P}V$ is a canonical subscheme where $S\cong S_X$ then $X$ is said to be \emph{an associated subscheme} to the surface $S$.  
\end{definition}

The associated cubic surface $S_X$ to a canonical subscheme $X\subset\mathbb{P}V$ is always unique up to isomorphism, however in general there are multiple (at most $72$) non-isomorphic canonical subschemes associated to a given canonical cubic surface. 

\begin{remark}
\label{remk:6tocubic-families}
    The construction sending a canonical subscheme $X$ to its associated cubic surface $S_X$ can also be done in families. For ease of notation we only write it over affine schemes. Let $\Spec A$ be an affine scheme and let $X_A\subset\mathbb{P}V_A:=\mathbb{P}V\times\Spec A$ be a closed subscheme, flat over $\Spec A$, whose fibres are canonical length $6$ subschemes. The blowup $\Bl_X\mathbb{P}V_A$ is $A$-flat and the associated family of cubic surfaces
    \[\mathcal{S}_{X_A}=\Proj_A\left(\bigoplus_{n\geq0}H^0(\Bl_X\mathbb{P}V_A,\omega^{-n}_{\Bl_X{\mathbb{P}V_A}})\right)\]
    is $A$-flat (see Proposition~2.5.4 of \cite{counting-cubics_deopurkar-patel-tseng}). The family $S_{X_A}$ is embedded in the projectivisation of the rank $4$, locally-free $A$-module $H^0(\Bl_{X_A}\mathbb{P}V_A,\omega^{-1}_{\Bl_{X_A} \mathbb{P}V_A})$. Note that under the blowup ${f:\Bl_{X_A}\mathbb{P}V_A\to\mathbb{P}V_A}$ we have $f^\ast\mathcal{I}_{X_A}(3)=\omega^{-1}_{\Bl_{X_A} \mathbb{P}V_A}$. This induces an isomorphism of locally-free $A$-modules
    \[f^\ast:H^0(\mathbb{P}V_A,\mathcal{I}_X(3))\longrightarrow H^0(\Bl_X\mathbb{P}V_A,\omega^{-1}_{\Bl_{X_A} \mathbb{P}V_A})\]
    because the map $f^\ast\otimes\id_{A/\mathfrak{a}}:H^0(\mathbb{P}V,\mathcal{I}_{X_\mathfrak{a}}(3))\to H^0((\Bl_X\mathbb{P}V_A)_\mathfrak{a},\omega^{-1}_{(\Bl_{X_A}\mathbb{P}V_A)_\mathfrak{a}})$ is an isomorphism for every maximal ideal $\mathfrak{a}\subset A$. 
    
    The flat family $\mathcal{S}_{X_A}\to\Spec A$ induces a rational map
    \[\gamma:\Hilb_6(\mathbb{P}V)\dashrightarrow\vert\mathcal{O}_{\mathbb{P}U}(3)\vert//\SL(U),\]
    which is defined on the locus of canonical subschemes $X$ whose associated cubic surface $S_X$ is $\SL(U)$-semistable. A result of Mumford \cite[Subsection~1.14]{mumford-proj-var} states that a surface $S$ is \mbox{$\SL(U)$-semistable} if and only if $S$ has at worst $4A_1$ or $3A_2$ singularities (equivalently, if all of the singularities on $S$ are either of type $A_1$ or $A_2$).
\end{remark}

The set of singular cubic surfaces $\Delta\subset\vert\mathcal{O}_{\mathbb{P}U}(3)\vert$ forms an irreducible divisor of degree $32$ \cite{salmon-invariants,clebsch-invariants1} and the subvariety $\mathcal{S}(A_1)\subset\Delta$ is dense. However, the set of length $6$ subschemes associated to a singular cubic surface does not form (an open subset of) an irreducible divisor inside $\Hilb_6(\mathbb{P}V)$. Concretely, a general (automorphism-free) $A_1$ cubic surface $S$ has $51$ associated subschemes ${X\subset\mathbb{P}V}$ (up to $\PGL(3)$-equivalence). These subschemes are split into the following three types: 
\begin{enumerate}[label=(\roman*)]
    \item $30$ subschemes which contain a length $2$ point; and
    \item $20$ subschemes which contain a length $3$ collinear subscheme; and
    \item $1$ subscheme contained in a conic.
\end{enumerate}
Other than satisfying exactly one of the above three conditions, a subscheme $X\subset\mathbb{P}V$ associated to an $A_1$ cubic surface is otherwise reduced and in general position with respect to plane curves. 

A consequence of the Hilbert-Burch theorem (see Subchapter~20.4 of \cite{eisenbud-commutative_algebra}) is that the ideal sheaves $\mathcal{I}_X\subset\mathcal{O}_{\mathbb{P}V}$ cutting out a subscheme $X\subset\mathbb{P}V$ of types (i) and (ii) have a free resolution of the same shape, but the ideal sheaf cutting out a subscheme of type (iii) does not. More precisely,

\begin{proposition}
\label{prop:hilbert-burch-consequence}
    The ideal sheaf $\mathcal{I}_X$ cutting out a canonical length $6$ subscheme $X\subset\mathbb{P}V$ has a free resolution of the form
    \begin{equation}
    \label{eqn:hilbert-burch-exseq}
        0\longrightarrow \mathcal{O}_{\mathbb{P}V}^{\oplus 3}(-4)\longrightarrow\mathcal{O}_{\mathbb{P}V}^{\oplus 4}(-3)\longrightarrow\mathcal{I}_X\longrightarrow0
    \end{equation}
    if and only if $X$ is admissible (not contained in a conic). When such a free resolution exists, it is unique up to the $\GL(4)\times\GL(3)$ action on $\mathcal{H}om_{\mathcal{O}_{\mathbb{P}V}}(\mathcal{O}_{\mathbb{P}V}^{\oplus 3}(-4),\mathcal{O}_{\mathbb{P}V}^{\oplus 4}(-3))\cong\Hom(\mathbb{C}^3,\mathbb{C}^4)$.
\end{proposition}
\begin{proof}
    A proof can be found in Subchapter~9.3.2 of \cite{dolgachev}, a proof for the case when $X$ is reduced is also found in Subchapter~20.4.1 of \cite{eisenbud-commutative_algebra}. We only give a sketch of Dolgachev's proof.
    
    Let $X\subset\mathbb{P}V$ be an admissible subscheme cut out by an ideal $I_X\subset\mathbb{C}[V^\vee]_\bullet$. The ideal $I_X$ is generated by $4$ linearly independent cubic polynomials and has $\operatorname{depth}(I_X)=\codim(I_X)=2$, so $I_X$ satisfies the conditions of the Hilbert-Burch theorem. Then there is a $4\times3$ matrix $\varphi_1$ valued in $\mathbb{C}[V^\vee]_\bullet$ and a polynomial $a\in\mathbb{C}[V^\vee]_\bullet$ such that $I_X=(am_0,\ldots,am_3)$, where $m_0,\ldots,m_3$ are the maximal $3\times3$ minors of $\varphi_1$. Then $a$ is a unit and $\varphi_1\in V^\vee\otimes\Hom(\mathbb{C}^3,\mathbb{C}^4)$ is valued in linear forms in $V$, and $I_X=(m_0,\ldots,m_3)$. The matrix $\varphi_1$ presents a free resolution
    $$0\longrightarrow \mathbb{C}[V^\vee]_\bullet^{\oplus3}(-4)\xlongrightarrow{\varphi_1}\mathbb{C}[V^\vee]_\bullet^{\oplus 4}(-3)\longrightarrow I_X\longrightarrow0$$
    which, after taking the projective spectrum, gives the exact sequence \eqref{eqn:hilbert-burch-exseq}. 
    
    The matrix presenting any free resolution of the form of \eqref{eqn:hilbert-burch-exseq} is uniquely determined by the \mbox{$V^\vee$-linear} relations between the cubic generators of $I_X$. These relations are the rows of the matrix $\varphi_1$, unique up to the $\GL(4)$-freedom of the choice of basis for the set of generators of $I_X$ and a $\GL(3)$-freedom of the choice of basis for the set of relations. 

    Conversely, if $X$ is contained in a conic $\{f_2=0\}$ then $X$ is the complete intersection of the conic $\{f_2=0\}$ with a cubic $\{f_3=0\}$. The ideal $I_X\subset\mathbb{C}[V^\vee]_\bullet$ is generated by the polynomials $f_2$ and $f_3$. The Hilbert-Burch theorem still gives a free resolution of $I_X$ as
    $$0\longrightarrow\mathbb{C}[V^\vee]_\bullet(-5)\xlongrightarrow{\varphi_1}\mathbb{C}[V^\vee]_\bullet(-2)\oplus \mathbb{C}[V^\vee]_\bullet(-3)\longrightarrow I_X\longrightarrow0,$$
    presented by the matrix $\varphi_1=\begin{pmatrix}
        -f_3 & f_2
    \end{pmatrix}$, and $I_X$ is generated by the two maximal minors of $\varphi_1$. But this is not a free resolution given by a $4\times3$ matrix of linear forms and any attempt to construct such a resolution fails: the matrix of $V^\vee$-linear relations between the cubic generators $v_0f_2,v_1f_2,v_2f_2,f_3$ is
    $$\begin{pmatrix}
        0&v_2&-v_1\\
        -v_2&0&v_0\\
        v_1&-v_0&0\\
        0&0&0
    \end{pmatrix},$$
    whose $3\times3$ minors all vanish. 
\end{proof}

After choosing identifications $U^\vee\cong\mathbb{C}^4$ and $W\cong\mathbb{C}^3$, the matrix $\varphi_1\in V^\vee\otimes\Hom(\mathbb{C}^3,\mathbb{C}^4)$ is adjoint to a tensor $\varphi\in U^\vee\otimes V^\vee\otimes W^\vee$. The GIT moduli space of tensors $\varphi\in U^\vee\otimes V^\vee\otimes W^\vee$ is the focus of this paper. 

\subsection{Determinantal representations}
\label{subsec:determinantal-representations}

\begin{definition}
    \label{defn:433-varphi}
    Let $\varphi\in U^\vee\otimes V^\vee\otimes W^\vee$ be a $4\times3\times3$ tensor. The tensor $\varphi$ has several adjoints we are interested in; for each decomposition of $U\otimes V\otimes W$ as $A\otimes B$ we will write $\varphi_A$ for the adjoint linear map
    \[\varphi_A:A\longrightarrow B^\vee.\]
    For example, $\varphi_U\in\Hom(U,V^\vee\otimes W^\vee)$ and $\varphi_{U\otimes W}\in\Hom(U\otimes W,V^\vee)$. When $A,B,C\in\{U,V,W\}$ are three distinct vector spaces we will often view $\varphi_A:A\to B^\vee\otimes C^\vee=\Hom(B,C^\vee)$ as an element of $A^\vee\otimes\Hom(B,C^\vee)$ which, after choosing bases for $B$ and $C$, can be identified with a $\dim(C)\times\dim(B)$ matrix of linear forms in $A$. As a convention, we will view $\varphi_U$ as an element of $U^\vee\otimes\Hom(W,V^\vee)$; the vector space $W^\vee$ indexes the rows and $V^\vee$ indexes the columns. 
\end{definition}

\begin{definition}
\label{defn:associated-subschemes}
    Let $\varphi_A\in \Hom(A,B^\vee\otimes C^\vee)$ be a tensor. For each $r\leq\dim A$ we let the \emph{rank at most $r$} subscheme $\rk_{\leq r}(\varphi_A)\subset\mathbb{P}A$ be the subscheme of points $a\in A$ where $\varphi_A(a)$ has rank at most $r$. We say a point $a\in\rk_{\leq r}(\varphi_A)$ is a \emph{rank at most $r$} point of $\varphi_A$.

    The rank at most $2$ subschemes of $\varphi$ will be given special names.
    \begin{itemize}
        \item We let $S_\varphi:=\rk_{\leq2}(\varphi_U)\subset\mathbb{P}U$, which is cut out by the determinant cubic $\det(\varphi_U)$. We call the scheme $S_\varphi$ the \emph{associated cubic surface} to $\varphi$.
        \item We let $X_\varphi:=\rk_{\leq2}(\varphi_V)\subset\mathbb{P}V$.
        \item We let $Y_\varphi:=\rk_{\leq2}(\varphi_W)\subset\mathbb{P}W$. The schemes $X_\varphi$ and $Y_\varphi$ are cut out by the maximal $3\times3$ minors of $\varphi_V$ and $\varphi_W$, respectively. We call $X_\varphi$ and $Y_\varphi$ the \emph{associated point configurations} to $\varphi$.
    \end{itemize}
\end{definition}

In general $S_\varphi$ is a smooth cubic surface, the subschemes $X_\varphi$ and $Y_\varphi$ are smooth point configurations, and the rank at most $1$ subschemes $\rk_{\leq1}(\varphi_U)$ and $\rk_{\leq1}(\varphi_V)$ and $\rk_{\leq1}(\varphi_W)$ are empty. The derivative of the determinant $\det(\varphi_U)$ vanishes if all the $2\times2$ minors vanish, so every point of $\rk_{\leq1}(\varphi_U)$ is a singularity of $S_\varphi$. The converse is not true. 

\begin{definition}
    Let $S\subset\mathbb{P}U$ be a cubic surface (resp. $X\subset\mathbb{P}V$, resp. $Y\subset\mathbb{P}W$ be subschemes). A \emph{determinantal representation} of $S$ (resp. $X$, resp. $Y$) is a tensor $\varphi\in U^\vee\otimes V^\vee\otimes W^\vee$ where $S=S_\varphi$ (resp. $X=X_\varphi$, resp. $Y=Y_\varphi$). 
\end{definition}

Every cubic surface $S\subset\mathbb{P}U$ has a determinantal representation except for the $E_6$ cubic surface (whose defining polynomial is projectively equivalent to $u_0^2u_3+u_0u_2^2+u_1^3$). This was known to Segre in \cite{segre}, a proof can be found in Theorem~7.6 of \cite{thrall}. When $S$ is canonical the structure (and existence) of these determinantal representations is completely characterised by its set of associated subschemes, following the construction outlined in Proposition~\ref{prop:hilbert-burch-consequence}. 

\begin{proposition}
\label{prop:canonical-detreps}
    Fix a tensor $\varphi\in\mathbb{P}(U^\vee\otimes V^\vee\otimes W^\vee)$. The following are equivalent:
    \begin{enumerate}[label=(\roman*)]
        \item $S_\varphi$ is canonical.
        \item $X_\varphi$ is admissible.
        \item $Y_\varphi$ is admissible.
    \end{enumerate}
    If $S_\varphi$ is canonical then $S_\varphi$ is the associated cubic surface to $X_\varphi$ and $Y_\varphi$. 

    Every admissible length $6$ subscheme of $\mathbb{P}V$ has a unique determinantal representation, up to $\GL(U)\times\GL(W)$-equivalence. The set of determinantal representations of a canonical cubic surface $S$ is in bijection (up to $\GL(U)\times\GL(V)\times\GL(W)$-equivalence) with the set of admissible length $6$ subschemes $X\subset\mathbb{P}V$ whose associated cubic surface is $S$.
\end{proposition}
\begin{proof}
    All details can be found in \cite[\S9.3.2]{dolgachev}.
\end{proof}

\subsection{Low rank linear subspaces}

A useful $\GL(U)\times\GL(V)\times\GL(W)$-invariant of a tensor ${\varphi\in U^\vee\otimes V^\vee\otimes W^\vee}$ is the number and type of linear subspaces on which $\varphi$ drops rank (linear subspaces of $\rk_{\leq r}(\varphi_A)$). We will use these linear subspaces for our GIT semistability and stability classification. We introduce some notation and facts about these subspaces, all proofs and details can be found in the papers \cite{atkinson-lloyd80} and \cite{atkinson-lloyd81} by Atkinson and Lloyd.

\begin{definition}
    Let $A$ and $B$ be vector spaces. A subspace $C\subset A^\vee\otimes B^\vee$ is said to be \emph{of block type} $(a,b)$ if there exist subspaces $A'\subset A$ and $B'\subset B$ with $\dim A'=a$ and $\dim B'=b$ such that $A'\otimes B'\subset\ker(c:A\otimes B\to \mathbb{C})$ for every pairing $c\in C$.

    Let $r$ be an integer. A subspace $C\subset A^\vee\otimes B^\vee$ is said to be a \emph{rank at most $r$ subspace} if every pairing $c\in C$ has rank at most $r$. 
\end{definition}
Up to a choice of basis for $A$ and $B$, each subspace of block type $(a,b)$ is of the form 
$$\left[\begin{array}{c|c}
     \ast & \ast \\
     \hline
     \ast & 0
\end{array}\right]\subset A^\vee\otimes B^\vee$$
with a $a\times b$ block of zeros in the corner. 

It is easy to verify that every subspace of block type $(a,b)$ is a rank at most $r$ subspace where $r=\dim A+\dim B-a-b$, but the converse is not true. The classification of rank at most $r$ subspaces is in general quite difficult, but we only need the cases of $r=1$ and $r=2$.

\begin{proposition}[{\cite[Lemma~2]{atkinson-lloyd81}}]
\label{prop:rank1-subspaces}
    Every rank $1$ subspace satisfies a rank $1$ block condition.
\end{proposition}

For rank at most $2$ subspaces we introduce some notation to describe the only exception, which is the vector space of skew-symmetric $3\times3$ matrices. 

\begin{definition}
\label{defn:skew-sym-subspace-defn}
    Let $A$ and $B$ be vector spaces and let $C\subset A^\vee\otimes B^\vee$ be a $3$-dimensional subspace. If there exist injections $\iota_A:\mathbb{C}^3\to A^\vee$ and $\iota_B:\mathbb{C}^3\to B^\vee$ such that
    \[C=(\iota_A\otimes\iota_B)(\mathbb{C}^3\wedge\mathbb{C}^3)\subset A^\vee\otimes B^\vee\]
    then $C$ is said to be a \emph{skew-symmetric subspace} of $A^\vee\otimes B^\vee$. 
\end{definition}

An interesting property of skew-symmetric subspaces is the following:

\begin{proposition}
\label{prop:skew-sym-adjoints}
    Let $A$ and $B$ and $C$ be $3$-dimensional vector spaces. Let $\varphi\in A^\vee\otimes B^\vee\otimes C^\vee$ be a $3\times3\times3$ tensor. All three of the subspaces $\varphi_A(A)\subset B^\vee\otimes C^\vee$ and $\varphi_B(B)\subset A^\vee\otimes C^\vee$ and $\varphi_C(C)\subset A^\vee\otimes B^\vee$ are skew-symmetric subspaces if and only if just one of them is a skew-symmetric subspace.
\end{proposition}
\begin{proof}
    Suppose $\varphi_C(C)\subset A^\vee\otimes B^\vee$ is a skew-symmetric subspace. Then there is a choice of basis for $A$ and $B$ and $C$ that identifies $\varphi_C:C\to A^\vee\otimes B^\vee$ with the Hodge star map $\star:\mathbb{C}^3\to\wedge^2(\mathbb{C}^3)^\vee$. The Hodge star map $\star$ is adjoint to a non-zero element of $\wedge^3(\mathbb{C}^3)^\vee\cong\mathbb{C}$, the proposition follows by the symmetry of the three adjoints of $\star$. 
\end{proof}

\begin{proposition}[{\cite[Lemma~7]{atkinson-lloyd81}}]
\label{prop:rank2-subspaces}
    Every rank at most $2$ subspace either satisfies a rank $2$ block condition or is a ($3$-dimensional) skew-symmetric subspace.
\end{proposition}

The class of tensors $\varphi\in U^\vee\otimes V^\vee\otimes W^\vee$ whose associated cubic surface $S_\varphi$ contains a skew-symmetric subspace will appear repeatedly throughout this paper. 

\begin{definition}
\label{defn:skew-sym-tensor}
    A tensor $\varphi\in U^\vee\otimes V^\vee\otimes W^\vee$ is said to be \emph{skew-symmetric} if there is a \mbox{$3$-dimensional} linear subspace $U'\subset U$ such that $\varphi_U(U')\subset V^\vee\otimes W^\vee$ is a skew-symmetric subspace. 

    Up to $\GL(U)\times\GL(V)\times\GL(W)$-equivalence, a skew-symmetric tensor $\varphi$ can be written as 
    $$\varphi_U = u_0Q + \begin{pmatrix}
        0&u_3&-u_2\\-u_3&0&u_1\\
        u_2&-u_1&0
    \end{pmatrix}\in U^\vee\otimes\Hom(W,V^\vee),$$
    where $U'=\ker(u_0)\subset U$. We define the \emph{rank} of the skew-symmetric tensor $\varphi$ to be the rank of the matrix $Q$. We let $\Sigma\subset \mathbb{P}(U^\vee\otimes V^\vee\otimes W^\vee)$ be the subvariety of skew-symmetric tensors. We let $\Sigma_r\subset\Sigma$ be the subvariety of skew-symmetric tensors $\varphi\in\Sigma$ of rank $r$. 
\end{definition}

When $\varphi\in U^\vee\otimes V^\vee\otimes W^\vee$ is a skew-symmetric tensor, the associated $3$-dimensional subspace $U'\subset U$ is unique. The fact that $\varphi_U(U')\subset V^\vee\otimes W^\vee$ is a skew-symmetric subspace is equivalent to the following commutative diagram being induced by three isomorphisms $U'\cong V\cong W\cong\mathbb{C}^3$:
\[\begin{tikzcd}
	{U'} & {V^\vee\otimes W^\vee} & \\
	{\mathbb{C}^3} & {\wedge^2(\mathbb{C}^3)^\vee} & {(\mathbb{C}^3)^\vee\otimes(\mathbb{C}^3)^\vee.}
	\arrow["{\varphi_U\vert_{U'}}", from=1-1, to=1-2]
	\arrow["\sim"', from=1-1, to=2-1]
	\arrow[from=1-2, to=2-2]
	\arrow["\sim", from=1-2, to=2-3]
	\arrow["\star", from=2-1, to=2-2]
	\arrow[hook, from=2-2, to=2-3]
\end{tikzcd}\]
As a result, a skew-symmetric tensor $\varphi_U$ fixes an identification $U'\cong V\cong W$.

Up to $\GL(U)$-equivalence, the matrix $Q\in\Hom(W,V^\vee)$ can be chosen to be symmetric with respect to this identification $V\cong W$. There is a symmetric matrix $B\in\GL(V)\cong\GL(W)$ that diagonalises $A$ and preserves the skew-symmetric component of $\varphi_U$ (the matrix $B$ also acts on $U'$, preserving the isomorphism $U'\cong V\cong W$). The eigenvalues of $B^TQB$ can be normalised to either $0$ or $1$ (depending on the rank of $Q$) and can be permuted by permuting the bases of $U$ and $V$ and $W$. As a result,

\begin{proposition}
\label{prop:skew-sym-equivalence}
    Every skew-symmetric tensor with the same rank is $\GL(U)\times\GL(V)\times\GL(W)$-equivalent.
\end{proposition}

The rank at most $2$ subschemes associated to a skew-symmetric tensor $\varphi\in\Sigma_r$ are easy to compute from the $3\times3$ minors of the matrices $\varphi_U$ and $\varphi_V$ and $\varphi_W$. These are:
\begin{itemize}
    \item Suppose $\varphi\in\Sigma_3$. The surface $S_\varphi\subset\mathbb{P}U$ is the transverse union of a (skew-symmetric) plane and a smooth quadric surface. The subschemes $X_\varphi\subset\mathbb{P}V$ and $Y_\varphi\subset\mathbb{P}W$ are smooth conics. 
    \item Suppose $\varphi\in\Sigma_2$. The surface $S_\varphi\subset\mathbb{P}U$ is the transverse union of three planes (one skew-symmetric and two of type $(2,2)$). The subschemes $X_\varphi\subset\mathbb{P}V$ and $Y_\varphi\subset\mathbb{P}W$ are the transverse unions of two lines.
    \item Suppose $\varphi\in\Sigma_1$. The surface $S_\varphi\subset\mathbb{P}U$ is the transverse union of a (skew-symmetric) plane and a doubled plane (of type $(2,2)$). The subschemes $X_\varphi\subset\mathbb{P}V$ and $Y_\varphi\subset\mathbb{P}W$ are doubled lines.
    \item Suppose $\varphi\in\Sigma_0$. Then $S_\varphi=\mathbb{P}U$ and $X_\varphi=\mathbb{P}V$ and $Y_\varphi=\mathbb{P}W$.
\end{itemize}

\section{Partial GIT quotients}

\label{sec:partial-quotients}

Before looking at the full Geometric Invariant Theory (GIT) quotient of $\mathbb{P}(U^\vee\otimes V^\vee\otimes W^\vee)$ by $\SL(U)\times\SL(V)\times\SL(W)$ we will review the partial quotients by the subgroups $\SL(U)$ and $\SL(V)$ and $\SL(W)$, and $\SL(U)\times\SL(V)$ and $\SL(U)\times\SL(W)$ and $\SL(V)\times\SL(W)$. These are all special cases of quiver GIT. We refer the reader to either \cite{hoskins-notes} or \cite{mumford-git} for details on GIT.

Let $H$ be a reductive group acting linearly on a vector space $T$. A closed point $t\in\mathbb{P}T$ is said to be $H$-\emph{semistable} if there exists a non-constant $H$-invariant function $f\in\mathbb{C}[T^\vee]_\bullet$ where $f(t)\ne0$. The point $t$ is said to be $H$-\emph{stable} if additionally the stabiliser $\Stab_H(t)$ is finite and the $H$-orbit $H\cdot t$ is closed inside the set of $H$-semistable points. 

Let $Z\subset\mathbb{P}T$ be a $H$-invariant subscheme. We will write 
$$Z^{H-s}:=\left\{z\in Z\;\vert\; \textnormal{$z$ is $H$-stable}\right\}$$
for the $H$-\emph{stable} locus,
$$Z^{H-ss}:=\left\{z\in Z\;\vert\;\textnormal{$z$ is $H$-semistable}\right\}$$
for the $H$-\emph{semistable} locus, and we will write $Z^{H-sss}:=Z^{H-ss}\setminus Z^{H-s}$ for the \emph{strictly} $H$-semistable locus. Note that $Z^{H-s}\subset Z^{H-ss}$ and $Z^{H-ss}\subset Z$ are (possibly empty) open subschemes. A point $t\in\mathbb{P}T$ is said to be $T$-\emph{unstable} if $t$ is not $T$-\emph{semi}stable. We will write $\mathbb{C}[T^\vee]_\bullet^H$ for the ring of $H$-invariant functions on $\mathbb{P}T$. The \emph{GIT quotient} of $\mathbb{P}T$ with respect to $H$ is the scheme
$$\mathbb{P}T//H:=\Proj\left(\mathbb{C}[T^\vee]_\bullet^H\right).$$
By construction, the inclusion of rings $\mathbb{C}[T^\vee]_\bullet^H\subset\mathbb{C}[T^\vee]_\bullet$ induces a rational map of projective spectra
$$\pi:\mathbb{P}T\dashrightarrow\mathbb{P}T//H$$
which is defined at a point $t\in\mathbb{P}T$ if and only if $t$ is $H$-semistable. We introduce some notation:

\begin{definition}
    Let $G=\SL(U)\times\SL(V)\times\SL(W)$.
\end{definition}

\begin{definition}
\label{defn:GIT-maps}
    We will write 
    $$p_U:\vert\mathcal{O}_{\mathbb{P}U}(3)\vert\dashrightarrow\vert\mathcal{O}_{\mathbb{P}U}(3)\vert//\SL(U)$$
    for the quotient map.

    For each labelling of $\{A,B\}=\{U,V,W\}$ we let $\pi_A$ or $\pi_{AB}$ or $\pi_{UVW}$ be the resulting quotient map from $\mathbb{P}(U^\vee\otimes V^\vee\otimes W^\vee)$ to its GIT quotient by $\SL(A)$ or $\SL(A)\times\SL(B)$ or $G$. As examples:
    \begin{align*}
        \pi_U&:\mathbb{P}(U^\vee\otimes V^\vee\otimes W^\vee)\dashrightarrow\mathbb{P}(U^\vee\otimes V^\vee\otimes W^\vee)//\SL(U),\\
        \pi_{VW}&:\mathbb{P}(U^\vee\otimes V^\vee\otimes W^\vee)\dashrightarrow \mathbb{P}(U^\vee\otimes V^\vee\otimes W^\vee)//(\SL(V)\times\SL(W)).
    \end{align*}
\end{definition}

Fix identifications of $V\cong\mathbb{C}^3$ and $W\cong\mathbb{C}^3$. Taking the determinant $\det(\varphi_U)$ of a $3\times3$ matrix $\varphi\in U^\vee\otimes\Hom(W,V^\vee)\cong U^\vee\otimes\operatorname{Mat}_{3\times3}(\mathbb{C})$ gives a $G$-equivariant function
$$\det: U^\vee\otimes V^\vee\otimes W^\vee\longrightarrow\Sym^3(U^\vee).$$
Other choices of bases for $V$ and $W$ only rescale the map $\det$ by a scalar, this choice will not matter after projectivisation. Pullback by $\det$ induces a $G$-equivariant map of graded rings 
$${\det}^\ast:\mathbb{C}[\Sym^3(U)]_\bullet\longrightarrow\mathbb{C}[U\otimes V\otimes W]_\bullet.$$

\begin{definition}
    By abuse of notation, we let
    $$\det:\mathbb{P}(U^\vee\otimes V^\vee\otimes W^\vee)\dashrightarrow\vert\mathcal{O}_{\mathbb{P}U}(3)\vert=\mathbb{P}\Sym^3(U^\vee)$$
    be the resulting $G$-equivariant map of projective spectra induced by $\det^\ast$. Similarly to Definition~\ref{defn:GIT-maps}, for each labelling of $\{A,B\}\in\{U,V,W\}$ we let $\det_A$ and $\det_{AB}$ and $\det_{UVW}$ be the resulting map after quotienting both sides by $\SL(A)$ or $\SL(A)\times\SL(B)$ or $G$. For example:
    \begin{align*}
        {\det}_{V}&:\mathbb{P}(U^\vee\otimes V^\vee\otimes W^\vee)//\SL(V)\dashrightarrow\vert\mathcal{O}_{\mathbb{P}U}(3)\vert, \\
        {\det}_{UVW}&:\mathbb{P}(U^\vee\otimes V^\vee\otimes W^\vee)//G\dashrightarrow\vert\mathcal{O}_{\mathbb{P}U}(3)\vert//\SL(U).
    \end{align*}

    Note that for each choice of $\Box=V,W,VW$, or omitting $\Box$, we have $p_U\circ\det_\Box=\det_{U\Box}\circ\pi_{U\Box}$. 
\end{definition}

Similarly, fixing a choice of bases $U\cong\mathbb{C}^4$ and $W\cong\mathbb{C}^3$ gives four $G$-equivariant functions
$$m_i: U^\vee\otimes V^\vee\otimes W^\vee\longrightarrow\Sym^3(V^\vee),$$
where $m_i$ outputs the $i^\textnormal{th}$ minor of the matrix $\varphi_U$. The four functions $m_0,m_1,m_2,m_3$ generate an ideal sheaf $\mathcal{I}\subset\mathcal{O}_{\mathbb{P}V\times\mathbb{P}(U^\vee\otimes V^\vee\otimes W^\vee)}$, independent of the choice of bases, which cuts out a family of subschemes $\mathcal{X}\subset\mathbb{P}V\times\mathbb{P}(U^\vee\otimes V^\vee\otimes W^\vee)$. We write $\mathcal{X}^\textnormal{len $6$}$ for the restriction of $\mathcal{X}$ to the open subset $\mathbb{P}V\times\mathbb{P}(U^\vee\otimes V^\vee\otimes W^\vee)^\textnormal{len $6$}$ of points $(v,\varphi)$ where the associated subscheme $X_\varphi$ is a length $6$ subscheme of $\mathbb{P}V$. The scheme $\mathcal{X}^\textnormal{len $6$}$ is flat over $\mathbb{P}(U^\vee\otimes V^\vee\otimes W^\vee)^\textnormal{len $6$}$.

\begin{definition}
    We let 
    $$\chi:\mathbb{P}(U^\vee\otimes V^\vee\otimes W^\vee)\dashrightarrow\Hilb_6(\mathbb{P}V)$$
    be the morphism associated to the flat family $\mathcal{X}^\textnormal{len $6$}\to\mathbb{P}(U^\vee\otimes V^\vee\otimes W^\vee)^\textnormal{len $6$}$. We let 
    $$\chi_{UW}:\mathbb{P}(U^\vee\otimes V^\vee\otimes W^\vee)//(\SL(U)\times\SL(W))\dashrightarrow\Hilb_6(\mathbb{P}V)$$
    be the induced map from the GIT quotient.
\end{definition}

By construction, we have $\chi(\varphi)=X_\varphi$ and $\chi=\chi_{UW}\circ\pi_{UW}$ when the right hand side is defined, noting that $\chi_{UW}$ is well-defined because $\chi$ is $\SL(U)\times\SL(W)$-invariant. 

The quotients of $\mathbb{P}(U^\vee\otimes V^\vee\otimes W^\vee)$ by $\SL(U)$ or $\SL(V)$ or $\SL(W)$ are the respective Grassmannians $\Gr(4,V^\vee\otimes W^\vee)\cong\Gr(4,9)$ or $\Gr(3,U^\vee\otimes W^\vee)$ or $\Gr(3,U^\vee\otimes V^\vee)$, the latter two are isomorphic to $\Gr(3,12)$. All $\SL(A)$-semistable points are $\SL(A)$-stable and a tensor $\varphi$ is \mbox{$\SL(A)$-stable} if and only if the corresponding map $\varphi_A:A\to B^\vee\otimes C^\vee$ is injective (where we relabel $\{A,B,C\}=\{U,V,W\}$). 

The moduli space $\Gr(3,U^\vee\otimes V^\vee)$ is a compactification of the moduli space of pairs $(S,X)$ where $S\subset\mathbb{P}U$ is a canonical cubic surface and $X\subset\mathbb{P}V$ is an admissible subscheme associated to $S$. The moduli space $\Gr(3,V^\vee\otimes W^\vee)$ is a compactification of the moduli space of pairs $(X,Y)$ where $X\subset\mathbb{P}V$ and $Y\subset\mathbb{P}W$ are Gale-dual admissible length $6$ subschemes. 

The quotient $\mathbb{P}(U^\vee\otimes V^\vee\otimes W^\vee)//(\SL(V)\times\SL(W))$ is the GIT moduli space of determinantal representations of cubic surfaces $S\subset\mathbb{P}U$. This quotient has been studied in great detail by Lehn, Lehn, Storger and van Straten, we refer the reader to \cite{twisted-cubics_lehn-lehn-sorger-straten} for all details as well as many closely related moduli spaces. In particular, they prove

\begin{proposition}
\label{prop:VW-semistability}
    Let $\varphi\in\mathbb{P}(U^\vee\otimes V^\vee\otimes W^\vee)$ be a tensor. 
    \begin{itemize}
        \item $\varphi$ is $\SL(V)\times\SL(W)$-stable if and only if $\varphi_U(U)\subset V^\vee\otimes W^\vee$ is not a linear subspace of block type $(1,2)$ or $(2,1)$. This occurs if and only if either:
        \begin{itemize}
            \item the surface $S_\varphi\subsetneq\mathbb{P}U$ is irreducible; or
            \item the tensor $\varphi$ is skew-symmetric.
        \end{itemize}
        \item $\varphi$ is $\SL(V)\times\SL(W)$-semistable if and only if $\varphi_U(U)\subset V^\vee\otimes W^\vee$ is not a linear subspace of block type $(1,3)$ or $(2,2)$ or $(3,1)$. The tensor $\varphi$ is strictly $\SL(V)\times\SL(W)$-semistable if and only if the surface $S_\varphi\subset\mathbb{P}U$ is reducible but $\varphi$ is not skew-symmetric.
        \item $\varphi$ is $\SL(V)\times\SL(W)$-unstable if and only if both $\det(\varphi_U)=0$ and $\varphi$ is not skew-symmetric.
    \end{itemize}
\end{proposition}
\begin{proof}
    Everything follows from Corollary~3.7 of \cite{twisted-cubics_lehn-lehn-sorger-straten}, which we have restated in slightly more detail using our notation. 
\end{proof}

By Subsection~10.5 of \cite{abch}, the quotients of $\mathbb{P}(U^\vee\otimes V^\vee\otimes W^\vee)$ by $\SL(U)\times\SL(V)$ and $\SL(U)\times\SL(W)$ are the final minimal models of the respective Hilbert schemes $\Hilb_6(\mathbb PW)$ and $\Hilb_6(\mathbb{P}V)$. The quotient $\mathbb{P}(U^\vee\otimes V^\vee\otimes W^\vee)//(\SL(U)\times\SL(W))$ is an orbit space. The birational contraction 
$$\gamma:\Hilb_6(\mathbb{P}V)\dashrightarrow\mathbb{P}(U^\vee\otimes V^\vee\otimes W^\vee)//(\SL(U)\times\SL(W))$$
sends a length $6$ subscheme $X$ to the $\SL(U)\times\SL(W)$-orbit of $\SL(U)\times\SL(W)$-stable tensors $\varphi\in\mathbb{P}(U^\vee\otimes V^\vee\otimes W^\vee)$ such that $X\subset X_\varphi$ (which is defined whenever this orbit exists and is unique). 

In Subsection~\ref{subsec:gale-duality} we give some background on Gale-duality and use it to construct the map $\gamma$. We give a new proof that $\gamma$ is a birational contraction in Subsection~\ref{subsec:hilb-contraction} and show that $\gamma$ is an isomorphism when restricted to the locus of admissible length $6$ subscheme; we use this isomorphism to prove our classification of $\SL(U)\times\SL(V)\times\SL(W)$-semistability and stability in Section~\ref{sec:full-quotient}.

\subsection{Gale-duality}
\label{subsec:gale-duality}

There is an involution, known as the \emph{Gale transform}, that sends a sufficiently general ordered set of $r+s+2$ points $X\subset\mathbb{P}^r$ to another ordered set of $r+s+2$ points $Y\subset\mathbb{P}^s$, uniquely defined up to $\PGL(s)$. It is an involution in the sense that the Gale transform of $Y$ is projectively equivalent to the original point configuration $X$. This was first noticed by Coble in 1922 \cite{coble-22} and has a simple linear-algebraic description: two ordered sets of $r+s+2$ points $\{x_1,\ldots,x_{r+s+2}\}\subset\mathbb{P}^r$ and $\{y_1,\ldots,y_{r+s+2}\}\subset\mathbb{P}^s$ are said to be \emph{Gale-dual} if the set of tensor products $\{x_i\otimes y_i\}_{i=1}^{r+s+2}\subset\mathbb{P}(\mathbb{C}^{r+1}\otimes\mathbb{C}^{s+1})$ is linearly dependent. 

Eisenbud and Popescu extended the Gale transform to Gorenstein subschemes in \cite{eisenbud-popescu99,eisenbud-popescu00}. We give a sketch of their construction, all technical details and related results can be found in Sections~5 and 6 of \cite{eisenbud-popescu00}. In most cases, the Gale transform of a subscheme $X\subset\mathbb{P}^r$ is determined by the structure of the free resolution of the ideal sheaf $\mathcal{I}_X\subset\mathcal{O}_{\mathbb{P}^r}$. In the case that $r=s=2$ and $X\subset\mathbb{P}^2$ is sufficiently general (e.g. admissible), the pair of the subscheme $X\subset\mathbb{P}^2$ and its Gale transform $Y\subset\mathbb{P}^2$ also reconstruct the free resolution of $\mathcal{I}_X$ given in Proposition~\ref{prop:hilbert-burch-consequence}.

Let $X\subset\mathbb{P}V$ be a Gorenstein length $6$ subscheme not contained in a line, with dualising sheaf $\omega_X$. The composition of the trace map $\tau:H^0(X,\omega_X)\to\mathbb{C}$ with the cup product induces the perfect pairing
\[H^0(X,\mathcal{O}_X(1))\otimes H^0(X,\omega_X(-1))\longrightarrow H^0(X,\omega_X)\xlongrightarrow{\tau}\mathbb{C}\]
of Serre duality. We let $V'^\vee=(V^\vee)^\perp\subset H^0(X,\omega_X(-1))$ be the annihilator of $V^\vee\subset H^0(X,\mathcal{O}_X(1))$ with respect to this pairing; we say that the pair of linear series $V^\vee\subset H^0(X,\mathcal{O}_X(1))$ and \linebreak ${V'^\vee\subset H^0(X,\omega_X(-1))}$ are \emph{Gale-dual}. 

\begin{definition}
\label{defn:gale-duality}
    Let $X\subset\mathbb{P}V$ be a Gorenstein length $6$ subscheme that is not contained in a line. The image $X'\subset\mathbb{P}V'$ of $X$ under the linear series $V'^\vee\subset H^0(X,\omega_X(-1))$, as above, is called the \emph{Gale transform} of the scheme $X\subset\mathbb{P}V$.

    If a projective isomorphism $\mathbb{P}V'\to\mathbb{P}W$ induces an isomorphism from $X'\subset\mathbb{P}V'$ to a subscheme $Y\subset\mathbb{P}W$ then we say the subschemes $X\subset\mathbb{P}V$ and $Y\subset\mathbb{P}W$ are \emph{Gale-dual}. 
\end{definition}

One can check that Definition \ref{defn:gale-duality} agrees with the previous linear-algebraic definition. The reason we require $X$ to not be contained in a line is to ensure the map $V^\vee\to H^0(X,\mathcal{O}_X(1))$ is injective. The Gale transform of a length $6$ subscheme contained in a line can still be defined, but it lives on a rational normal curve inside $\mathbb{P}^3$. 

To a Gorenstein length $6$ subscheme $X\subset\mathbb{P}V$ we associate the linear map 
$$\phi:V^\vee\otimes V'^\vee\longrightarrow \ker(\tau)\subset H^0(X,\omega_X).$$
The vector space $\ker(\tau)$ is $5$-dimensional and, under the additional assumption that $\phi$ is surjective, we consider the inclusion of its $4$-dimensional kernel
\[\psi_X:\ker(\phi)\to V^\vee\otimes V'^\vee.\]
A sufficient condition for surjectivity of $\phi$ is that $X$ contains a length $4$ subscheme $X'\subset X$ in linearly general position (the intersection $L\cap X'$ of $X'$ with any line $L\subset\mathbb{P}V$ has length at most $2$) \cite[Proposition~5.8]{eisenbud-popescu00}.
\begin{definition}
\label{defn:gale-construction}
    Let $X\subset\mathbb{P}V$ be a Gorenstein length $6$ subscheme. Following the above construction, when the map $\phi:V^\vee\otimes V'\to\ker(\tau)$ is surjective we call the tensor ${\psi_X\in\Hom(\ker(\phi), V^\vee\otimes V'^\vee)}$ the \emph{associated tensor} to $X$.
\end{definition}

We have $\dim\ker(\phi)=4$ and $\dim V'=3$ so, after choosing identifications $\ker(\phi)\cong U$ and $V'\cong W$, the tensor $\psi$ can be identified with a $4\times3\times3$ tensor ${\psi\in U^\vee\otimes V^\vee\otimes W^\vee}$. In certain cases, the associated tensor $\psi$ to a Gorenstein length $6$ subscheme $X\subset\mathbb{P}V$ recovers the original scheme $X$ as $X_\psi$. 

For a tensor $\varphi\in U^\vee\otimes V^\vee\otimes W^\vee$ we let $\widetilde{\varphi}_V^\vee: U\otimes\mathcal{O}_{\mathbb{P}V}(1)\to W^\vee\otimes\mathcal{O}_{\mathbb{P}V}(2)$ be the map of sheaves presented by the matrix $\varphi_V\in V^\vee\otimes\Hom(U,W^\vee)$. We have

\begin{proposition}
\label{prop:EPthm61-facts}
    Let $\varphi\in U^\vee\otimes V^\vee\otimes W^\vee$ be a tensor. If the associated subscheme $X_\varphi\subset\mathbb{P}V$ is finite and the map $\varphi_V:V\to U^\vee\otimes W^\vee$ is injective then $X_\varphi$ is length $6$ (the same applies for $Y_\varphi$ and $\varphi_W$). The following are equivalent:
    \begin{enumerate}[label=(\roman*)]
        \item $X_\varphi$ and $Y_\varphi$ are finite.
        \item $X_\varphi$ is finite and Gorenstein.
        \item $Y_\varphi$ is finite and Gorenstein.
    \end{enumerate}
    If one of the above three statements holds then the schemes $X_\varphi$ and $Y_\varphi$ are Gale-dual. 
\end{proposition}
\begin{proof}
    Everything follows from Theorem~6.1 and Proposition~6.2 of \cite{eisenbud-popescu00}, except the claim that a finite subscheme is length $6$. If $\varphi_V$ is injective then $\varphi_V$ induces a linear embedding of ${\mathbb{P}V\xrightarrow{\sim}\mathbb{P}(\varphi_V(V))\subset\mathbb{P}(U^\vee\otimes W^\vee)}$. The subscheme $X_\varphi$ is the scheme-theoretic intersection of $\mathbb{P}V$ with the rank at most $2$ subvariety of $\mathbb{P}(U^\vee\otimes W^\vee)$, which is degree $6$ and codimension $2$. This intersection is then length $6$ whenever it is finite.
\end{proof}

The case for Gorenstein subschemes contained in a conic is slightly different. While we cannot use the Hilbert-Burch theorem to construct a determinantal representation $\varphi$ for a length $6$ Gorenstein subscheme $X\subset\mathbb{P}V$ contained in a conic, the construction of Definition~\ref{defn:gale-construction} still produces an associated tensor well-defined whenever $X$ is the complete intersection of a conic and a cubic. In this case, the associated subscheme $X_\varphi\subset\mathbb{P}V$ to $\varphi$ is the conic containing $X$. The following construction is essentially a reformulation of the proof of Theorem~7.1 of \cite{eisenbud-popescu00}. In particular, Eisenbud and Popescu prove that when $X\subset\mathbb{P}V$ is a Gorenstein length $6$ subscheme contained in a conic, its Gale transform $X'\subset\mathbb{P}^2$ is isomorphic to $X$ as a subscheme of $\mathbb{P}^2$. 

\begin{remark}
\label{remk:gale-conic}
    Concretely, let $X\subset\mathbb{P}V$ be the complete intersection of a conic $Q$ and a cubic curve. The scheme $X$ satisfies condition (b) of Theorem~7.1 of \cite{eisenbud-popescu00}, in particular there exists an isomorphism $\alpha':\mathcal{O}_X(1)\to\mathcal{O}_X(1)^\vee=\omega_X(-1)$ such that the composition
    $$V^\vee\longrightarrow H^0(X,\mathcal{O}_X(1))\xlongrightarrow{\alpha'} H^0(X,\mathcal{O}_X(1)^\vee)\longrightarrow V$$
    is zero. The isomorphism $\alpha'$ is a twist of the isomorphism $\alpha:\mathcal   O_X(2)\to\omega_X$ given by the adjunction formula. Now, consider the exact sequence of cohomology
    \[0\longrightarrow H^0(\mathbb{P}V,\mathcal{I}_X(2))\xlongrightarrow{f}H^0(\mathbb{P}V,\mathcal{O}_{\mathbb{P}V}(2))\xlongrightarrow{g}H^0(X,\mathcal{O}_X(2))\xlongrightarrow{h}H^1(\mathbb{P}V,\mathcal{I}_X(2))\longrightarrow0\]
    where $g$ sends $Q\mapsto0$ (viewing $Q\in\Sym^2(V^\vee)$ as a quadratic form by abuse of notation). We identify $\coker(f)=\Sym^2(V^\vee)/\Span\{Q\}$ with $\ker h\subset H^0(X,\mathcal{O}_X(2))$. The map $h$ generates $\Hom_\mathbb{C}\left(H^0(X,\mathcal{O}_X(2)),\mathbb{C}\right)$ as an $\mathcal{O}_X$-algebra. There is an isomorphism $H^1(\mathbb{P}V,\mathcal{I}_X(2))\cong\mathbb{C}$ so that the following diagram commutes
    \[\begin{tikzcd}
    	{H^0(X,\mathcal{O}_X(1))\times H^0(X,\mathcal{O}_X(1))} & {H^0(X,\mathcal{O}_X(2))} & {H^1(\mathbb{P}V,\mathcal{I}_X(2))} \\
    	{H^0(X,\mathcal{O}_X(1))\times H^0(X,\omega_X(-1))} & {H^0(X,\omega_X)} & {\mathbb{C}}
    	\arrow["\smile", from=1-1, to=1-2]
    	\arrow["{\id\times\alpha'}"', from=1-1, to=2-1]
    	\arrow["h", from=1-2, to=1-3]
    	\arrow["\alpha"', from=1-2, to=2-2]
    	\arrow["\sim", from=1-3, to=2-3]
    	\arrow["\smile", from=2-1, to=2-2]
    	\arrow["\tau", from=2-2, to=2-3]
    \end{tikzcd}\]
    and, as a result, we identify $h$ with the trace map $\tau$. In particular, $g(H^0(\mathbb{P}V,\mathcal{O}_{\mathbb{P}V}(2)))=\ker h$, i.e. $h$ annihilates the image of $V^\vee\otimes V^\vee$ inside $H^0(X,\mathcal{O}_X(2))$. The linear map $V^\vee\otimes V^\vee\to\ker h$ can then be identified with the natural map
    $$\phi:V^\vee\otimes V^\vee\longrightarrow\Sym^2(V^\vee)/\Span\{Q\}$$
    and $\ker\phi=\Span\{Q\}\oplus\wedge^2V^\vee$. We then consider the inclusion
    \begin{equation}
    \label{eqn:gale-conic-psi}
        \psi:\Span\{Q\}\oplus\wedge^2V^\vee\longrightarrow V^\vee\otimes V^\vee.
    \end{equation}
    We choose isomorphisms $U\cong\Span\{Q\}\oplus\wedge^2V^\vee$ and $W^\vee\cong V^\vee$ (only onto the second copy of $V^\vee$), which identifies $\psi$ with a rank $\rk(Q)$ skew-symmetric tensor $\varphi\in\Sigma_{\rk(Q)}$, unique up to ${\GL(U)\times\GL(W)}$-equivalence. 
    
    The tensor $\varphi$ is the tensor associated to the subscheme $X$, in the sense of Definition~\ref{defn:gale-construction}. In coordinates, $\varphi$ is $\GL(U)$-equivalent to the skew-symmetric tensor
    $$\varphi_U=u_0Q+\begin{pmatrix}
        0&u_3&-u_2\\
        -u_3&0&u_1\\
        u_2&-u_1&0
    \end{pmatrix}.$$
    The adjoint $4\times3$ tensor $\varphi_V\in V^\vee\otimes \Hom(W,U^\vee)$ to $\varphi_U$ is 
    $$\varphi_V=\begin{pmatrix}
        q_0&q_1&q_2\\
        0&-v_2&v_1\\
        v_2&0&-v_0\\
        -v_1&v_0&0\\
    \end{pmatrix},$$
    for some elements $q_0,q_1,q_2\in V^\vee$. There is a unique isomorphism $W^\vee\cong V^\vee$ that identifies the matrix $\begin{pmatrix}
        q_0 & q_1 & q_2
    \end{pmatrix}\in V^\vee\otimes W^\vee$ with the symmetric matrix $Q\in\Sym^2(V^\vee)\subset V^\vee\otimes V^\vee$. The $3\times3$ minors of $\varphi_V$ are 
    $$(0,v_0(q_0v_0+q_1v_1+q_2v_2),v_1(q_0v_0+q_1v_1+q_2v_2),v_2(q_0v_0+q_1v_1+q_2v_2)).$$
    These minors generate the ideal sheaf $\mathcal{I}_Q$ in degree $3$, so $X_\varphi=Q$ is the conic containing the original scheme $X$.

    When $h^0(\mathbb{P}V,\mathcal{I}_X(2))\geq2$ this procedure fails. Concretely, we only have  ${g(H^0(\mathbb{P}V,\mathcal{O}_{\mathbb{P}V}(2)))\subset\ker h}$. The above procedure produces a tensor
    $$\psi:H^0(\mathbb{P}V,\mathcal{I}_X(2))\oplus\wedge^2V^\vee\longrightarrow V^\vee\otimes V^\vee$$
    and there is a family of associated tensors to the length $6$ subscheme $X$ parametrised by the possible choices of injections $U\to H^0(\mathbb{P}V,\mathcal{I}_X(2))\oplus\wedge^2V^\vee$.
\end{remark}

\begin{remark}
\label{remk:gale-families}
    The map sending a Gorenstein subscheme $X\subset\mathbb{P}V$ to its associated tensor can also be done in families. For ease of notation we only write the construction over affine schemes. Let $\Spec A$ be an affine scheme and fix a subscheme $X_A\subset\mathbb{P}V\times\Spec A$, flat and Gorenstein over $A$, whose fibres $X_\mathfrak{a}\subset\mathbb{P}V\times\Spec(A/\mathfrak{a})$ are all length $6$ and span the projective space $\mathbb{P}V\times\Spec(A/\mathfrak{a})$. In particular, $X_A$ is $A$-flat. We let $\tau:H^0(X_A,\omega_{X_A})\to A$ be the trace map. 

    Because each $X_\mathfrak{a}$ spans $\mathbb{P}V\times\Spec(A/\mathfrak{a})$, the map $f:V^\vee\otimes A\to H^0(X_A,\mathcal{O}_{X_A}(1))$ is injective. We then identify $V^\vee\otimes A$ with its image $f(V^\vee\otimes A)\subset H^0(X_A,\mathcal{O}_{X_A}(1))$. The annihilator of $V^\vee\otimes A$ with respect to the trace pairing 
    $$H^0(X_A,\mathcal{O}_{X_A}(1))\otimes_AH^0(X_A,\omega_{X_A}(-1))\longrightarrow A$$
    is then a locally free submodule $\mathcal{W}^\vee\subset H^0(X_A,\omega_{X_A}(-1))$ of rank $3$. We let $Y_A\subset\mathbb{P}\mathcal{W}$ be the image of $X_A$ under the linear series $(\mathcal{W}^\vee,\omega_{X_A}(-1))$. For each maximal ideal $\mathfrak{a}\in\Spec A$ the pair of subschemes $X_\mathfrak{a}\subset\mathbb{P}V\times\Spec(A/\mathfrak{a})$ and $Y_\mathfrak{a}\subset\mathbb{P}\mathcal{W}\times\Spec(A/\mathfrak{a})$ are Gale-dual. 

    As in the construction preceding Definition \ref{defn:gale-construction}, Serre duality gives a linear map $$\phi:V^\vee\otimes\mathcal{W}^\vee\longrightarrow\ker(\tau)\subset H^0(X_A,\omega_{X_A}).$$
    We let $\mathcal{U}:=\ker(\phi)$ and consider its inclusion
    $$\psi:\mathcal{U}\longrightarrow V^\vee\otimes\mathcal{W}^\vee.$$
    Under the additional assumption that each fibre $X_\mathfrak{a}$ contains a length $4$ subscheme $X'_\mathfrak{a}\subset X_\mathfrak{a}$ in linearly general position, the map $\phi\otimes\id_{A/\mathfrak{a}}:(V^\vee\otimes \mathcal{W}^\vee)\otimes_A A/\mathfrak{a}\to \ker(\tau)\otimes_A A/\mathfrak{a}$ of locally free $A$-modules is surjective for each maximal ideal $\mathfrak{a}\in\Spec A$; so $\phi$ is surjective. Then $\mathcal{U}$ is a rank $4$ locally free $A$-module. The map $\psi$ can be identified with a section 
    $$\psi':\Spec A\longrightarrow\mathbb{P}(\mathcal{U}^\vee\otimes V^\vee\otimes \mathcal{W}^\vee)$$
    into a projectivised vector bundle of $4\times3\times3$ tensors. 
\end{remark}

\subsection{The final minimal model of the Hilbert scheme of $6$ points in $\mathbb{P}^2$}
\label{subsec:hilb-contraction}

The construction of Remark \ref{remk:gale-families} induces a rational map
$$\gamma:\Hilb_6(\mathbb{P}V)\dashrightarrow\mathbb{P}(U^\vee\otimes V^\vee\otimes W^\vee)//(\SL(U)\times\SL(W)).$$
By \cite{abch}, GIT semistability of $\mathbb{P}(U^\vee\otimes V^\vee\otimes W^\vee)$ under the $\SL(U)\times\SL(W)$-action coincides with a Bridgeland stability condition on $D^b\Coh(\mathbb{P}V)$. The map $\gamma$ is a birational contraction contracting the divisor $C_6\subset\Hilb_6(\mathbb{P}V)$ that parametrises length $6$ subschemes $X\subset\mathbb{P}V$ contained in a conic. Details can be found in Section~8 and Subsection~10.5 of their paper, however they do not construct this contraction. The purpose of Subsection~\ref{subsec:hilb-contraction} is to explicitly construct the contraction $\gamma$, we will use the map $\gamma$ to compute orbit closures in Section~\ref{sec:full-quotient}. 

The map $\gamma$ is defined when the following three conditions hold:
\begin{enumerate}[label=(\roman*)]
    \item the scheme $X\subset\mathbb{P}V$ is Gorenstein; and
    \item the map $V^\vee\otimes W^\vee\to\ker(\tau)\subset H^0(X,\omega_X)$ is surjective; and
    \item the associated tensor $\varphi$ is $\SL(U)\times\SL(W)$-semistable. 
\end{enumerate}
The $\SL(U)\times\SL(W)$-semistability of $\mathbb{P}(U^\vee\otimes V^\vee\otimes W^\vee)$ is known, a result of Drézet \cite{drezet}:
\begin{proposition}
\label{prop:UW-semistability}
    A tensor $\varphi\in\mathbb{P}(U^\vee\otimes V^\vee\otimes W^\vee)$ is $\SL(U)\times\SL(W)$-semistable if and only if $\varphi_V$ is not $\SL(U)\times\SL(W)$-equivalent to a matrix of one of the following forms:

    \[\begin{pmatrix}
        \ast&\ast&\ast\\
        \ast&\ast&\ast\\
        \ast&\ast&\ast\\
        0&0&0        
    \end{pmatrix},\qquad\begin{pmatrix}
        \ast&\ast&\ast\\
        \ast&\ast&\ast\\
        \ast&0&0\\
        \ast&0&0
    \end{pmatrix},\qquad\begin{pmatrix}
        \ast&\ast&\ast\\
        \ast&\ast&0\\
        \ast&\ast&0\\
        \ast&\ast&0\\
    \end{pmatrix}.\]

    All $\SL(U)\times\SL(W)$-semistable points are $\SL(U)\times\SL(W)$-stable.
\end{proposition}

A quick computation of the $3\times3$ minors shows that when $\varphi$ is $\SL(U)\times\SL(W)$-unstable (${\SL(U)\times\SL(W)}$-equivalent to one of the above three matrices) then the ideal cutting out the subscheme $X_\varphi\subset\mathbb{P}V$ is generated by either:
\begin{enumerate}[label=(\roman*)]
    \item a cubic polynomials; or
    \item two terms $l_1q$ and $l_2q$, where $l_1$ and $l_2$ are linear and $q$ is a quadratic; or 
    \item three terms $lq_1,lq_2,lq_3$, where $l$ is linear and $q_1,q_2,q_3$ are quadratics.
\end{enumerate}
Some (or all) of these generators may be zero. In particular, if $\varphi\in\mathbb{P}(U^\vee\otimes V^\vee\otimes W^\vee)$ is ${\SL(U)\times\SL(W)}$-unstable then the scheme $X_\varphi$ contains a positive-dimensional component, however the converse is not true. Some sufficient conditions for $\SL(U)\times\SL(W)$-semistability include

\begin{proposition}
    A tensor $\varphi\in\mathbb{P}(U^\vee\otimes V^\vee\otimes W^\vee)$ is $\SL(U)\times\SL(W)$-semistable if:
    \begin{enumerate}[label=(\roman*)]
        \item the scheme $X_\varphi$ is finite; or
        \item the tensor $\varphi$ is a positive rank skew-symmetric tensor ($\varphi\in\Sigma\setminus\Sigma_0$).
    \end{enumerate}
\end{proposition}
\begin{proof}
    Only (ii) is unproven, this is immediate because a positive rank skew-symmetric tensor $\varphi\in\Sigma\setminus\Sigma_0$ does not satisfy the block conditions of Proposition~\ref{prop:UW-semistability}.
\end{proof}

\begin{remark}
    The GIT moduli space $\mathbb{P}(U^\vee\otimes V^\vee\otimes W^\vee)//(\SL(U)\times\SL(W))$ is a moduli space of sheaves. Concretely, $\mathbb{P}(U^\vee\otimes V^\vee\otimes W^\vee)$ parametrises complexes of the form
    \begin{equation}
    \label{eqn:DbCoh-UW-complex1}
        W\otimes\mathcal{O}_{\mathbb{P}V}(-4)\xlongrightarrow{\widetilde{\varphi}_V}U^\vee\otimes \mathcal{O}_{\mathbb{P}V}(-3)
    \end{equation}
    inside $D^b\operatorname{Coh}(\mathbb{P}V)$ (forgetting the basis for $U$ and $W$). A complex of the form of \eqref{eqn:DbCoh-UW-complex1} is quasi-isomorphic to the complex
    \begin{equation}
    \label{eqn:DbCoh-UW-complex2}
        \ker(\widetilde{\varphi}_V)\longrightarrow 0\longrightarrow0\longrightarrow\coker(\widetilde{\varphi}_V),
    \end{equation}
    which is the complex associated to the sheaf $\coker(\widetilde{\varphi}_V)$ whenever $\ker(\widetilde{\varphi}_V)=0$. The kernel $\ker(\widetilde{\varphi}_V)$ is non-zero if and only if $\varphi_V(V)\subset U^\vee\otimes W^\vee$ is a rank at most $2$ subspace. This is only possible if either $\varphi_V(V)$ is of block type $(2,3)$ or $(3,2)$ or $(4,1)$, and hence $\SL(U)\times\SL(W)$-unstable, or if $\varphi_V(V)$ is a skew-symmetric subspace. But then $\varphi_V(V)$ is of block type $(1,3)$ and is also $\SL(U)\times\SL(W)$-unstable. So all points in $\varphi\in\mathbb{P}(U^\vee\otimes V^\vee\otimes W^\vee)^{\SL(U)\times\SL(W)-ss}$ correspond to complexes that are quasi-isomorphic to sheaves. 

    Because all $\SL(U)\times\SL(W)$-semistable points are stable, ${\mathbb{P}(U^\vee\otimes V^\vee\otimes W^\vee)//(\SL(U)\times\SL(W))}$ is an orbit space and its points parametrise complexes quasi-isomorphic to a complex of the form \eqref{eqn:DbCoh-UW-complex2} with $\ker(\widetilde{\varphi}_V)=0$.
\end{remark}

Recall that $\Sigma$ is the $\SL(U)\times\SL(W)$-orbit of tensors of the form 
$$\varphi_U=u_0Q+\begin{pmatrix}
    0&u_3&-u_2&\\
    -u_3&0&u_1\\
    u_2&-u_1&0
\end{pmatrix},$$
where the skew-symmetric part of a tensor $\varphi\in\Sigma$ gives an identification $V\cong W$. The matrix $Q\in V^\vee\otimes W^\vee$ can be chosen to be symmetric with respect to this identification. The subscheme $X_\varphi$ is cut out by the quadratic form $v^TQv=0$, and as a result we have
\begin{proposition}
    The map $\varphi\mapsto X_\varphi$ induces an isomorphism $\Sigma//(\SL(U)\times\SL(W)\cong\vert\mathcal{O}_{\mathbb{P}V}(2)\vert$.
\end{proposition}
\begin{proof}
    The family $\mathcal{Q} \subset\mathbb{P}V\times(\Sigma\setminus\Sigma_0)$ of associated conics $X_\varphi\subset\mathbb{P}V$ to tensors $\varphi\in\Sigma\setminus\Sigma_0$ is flat over $\Sigma\setminus\Sigma_0$. So $\mathcal{Q}$ corresponds to an $\SL(U)\times\SL(W)$-invariant morphism $f:\Sigma\setminus\Sigma_0\to\vert\mathcal{O}_{\mathbb{P}V}(2)\vert$. The morphism $f$ is surjective. We have $\Sigma^{\SL(U)\times\SL(W)-ss}=\Sigma\setminus\Sigma_0$, so $f$ factors through the GIT quotient and induces a surjective morphism $f_{UV}:\Sigma//(\SL(U)\times\SL(W))\to\vert\mathcal{O}_{\mathbb{P}U}(2)\vert$. Two tensors $\varphi,\varphi'\in\Sigma$ are $\SL(U)\times\SL(W)$-equivalent if and only if their associated quadratic forms $A$ and $A'$ are equal, so $f$ is a bijective morphism.
    
    Finally, the $\PGL(U)\times\PGL(W)$-stabiliser of every tensor $\varphi\in\Sigma$ is trivial, so the GIT quotient $\Sigma//(\PGL(U)\times\PGL(W))$ is smooth. The two GIT quotients $\Sigma//(\SL(U)\times\SL(W))$ and ${\Sigma//(\PGL(U)\times\PGL(W))}$ are isomorphic as schemes (the centre of $\SL(U)\times\SL(W)$ acts trivially), so ${\Sigma//(\SL(U)\times\SL(W))}$ is smooth. Then the morphism $f$ is a bijective morphism of smooth varieties, so is an isomorphism.
\end{proof}

\begin{definition}
    Let $C_6\subset\Hilb_6(\mathbb{P}V)$ be the divisor parametrising length $6$ subschemes $X\subset\mathbb{P}V$ contained in a conic. 

    Let $\mathcal{H}\subset\Hilb_6(\mathbb{P}V)$ be the open subvariety parametrising determinantal, Gorenstein length $6$ subschemes $X\subset\mathbb{P}V$ that are not contained in a conic. 
\end{definition}

\begin{proposition}
\label{prop:gamma-facts}
    The map $\chi_{UW}$ is the birational inverse of the map $$\gamma:\Hilb_6(\mathbb{P}V)\dashrightarrow\mathbb{P}(U^\vee\otimes V^\vee\otimes W^\vee)//(\SL(U)\times\SL(W)).$$
    The map $\gamma$ contracts $C_6$ onto $\Sigma//(\SL(U)\times\SL(W))\cong\vert\mathcal{O}_{\mathbb{P}V}(2)\vert$. The restriction of $\gamma$ to $\mathcal{H}$ is an isomorphism onto its image.
\end{proposition}
\begin{proof}
    The fact that $\gamma$ is a birational contraction contracting $C_6$ is already known, a result of \cite{abch}. Explicitly, the contraction comes from the construction of Remarks \ref{remk:gale-conic} and \ref{remk:gale-families}.

    To show $\gamma$ is an isomorphism when restricted to $\mathcal{H}$, we need to show that $\chi_{UW}\circ\gamma=\id$ whenever $\gamma$ is defined. Because both $\Hilb_6(\mathbb{P}V)$ and $\mathbb{P}(U^\vee\otimes V^\vee\otimes W^\vee)//(\SL(U)\times\SL(W))$ are smooth then it suffices to show that $\gamma$ is bijective when restricted to $\mathcal{H}$. Let $[X]\in\mathcal{H}$ be a determinantal, Gorenstein length $6$ subscheme. Then $X=X_\varphi$ for some $\varphi\in\mathbb{P}(U^\vee\otimes V^\vee\otimes W^\vee)^{\SL(U)\times\SL(W)-ss}$. By \cite[Theorem~6.1]{eisenbud-popescu00} we have $(\chi_{UW}\circ\gamma)([X_\varphi])=\chi(\varphi)=[X_\varphi]$ as desired.
\end{proof}

\section{The full $4\times3\times3$ GIT quotient}
\label{sec:full-quotient}

Recall that $G=\SL(U)\times\SL(V)\times\SL(W)$. The group $G$ is a finite extension of the symmetry group $\PGL(U)\times\PGL(V)\times\PGL(W)$ of $\mathbb{P}(U^\vee\otimes V^\vee\otimes W^\vee)$. The primary goal of this paper is to describe the $G$-semistability and $G$-stability of the moduli space $\mathbb{P}(U^\vee\otimes V^\vee\otimes W^\vee)$ of $4\times3\times3$ tensors. To describe the $G$-semistable and $G$-stable loci of $\mathbb{P}(U^\vee\otimes V^\vee\otimes W^\vee)$ we will need to classify all of the points $\varphi\in\mathbb{P}(U^\vee\otimes V^\vee\otimes W^\vee)$. In particular, we need a detailed enumeration of all canonical tensors; this was done by Ng in \cite{ng}, we review their classification and describe the resulting stratification of $\mathbb{P}(U^\vee\otimes V^\vee\otimes W^\vee)$ in Subsection \ref{subsec:tensor-classification}.
 
Subsections~\ref{subsec:G-semistability}, \ref{subsec:G-stability} and \ref{subsec:G-polystability} are devoted to describing the $G$-semistable, $G$-stable, and \mbox{$G$-polystable} points of $\mathbb{P}(U^\vee\otimes V^\vee\otimes W^\vee)$. We will give sufficient geometric conditions for \mbox{$G$-semistability} and $G$-stability and then prove these conditions are necessary using the Hilbert-Mumford criterion. Along the way we describe the $G$-invariant ring $\mathbb{C}[U\otimes V\otimes W]_\bullet^G$ up to finite extension and give a detailed description of a stratification of the $G$-semistable locus $\mathbb{P}(U^\vee\otimes V^\vee\otimes W^\vee)^{G-ss}$, two necessary ingredients for our proofs of $G$-stability and $G$-polystability. 

\subsection{A classification of $4\times3\times3$ tensors}
\label{subsec:tensor-classification}

It is convenient to stratify the set of tensors \linebreak ${\varphi\in\mathbb{P}(U^\vee\otimes V^\vee\otimes W^\vee)}$ by subsets corresponding to the singularity type that the associated cubic surface $S_\varphi$ has. We have

\begin{proposition}
\label{prop:types-of-tensors}
    Every tensor $\varphi\in\mathbb{P}(U^\vee\otimes V^\vee\otimes W^\vee)$ is of exactly one of the following types:
    \begin{enumerate}[label=(\roman*)]
        \item $\varphi$ is canonical (the associated surface $S_\varphi$ is canonical); or
        \item $S_\varphi$ is reducible; or
        \item $S_\varphi$ is a cone over a cubic curve; or
        \item $S_\varphi$ is non-normal, non-conical, and irreducible. In this case, the polynomial defining $S$ is $\SL(U)$-equivalent to either $u_2u_0^2+u_3u_1^2$ or $u_2u_0^2+u_3u_0u_1+u_1^3$. Or
        \item $\det(\varphi)=0$ and $S_\varphi=\mathbb{P}U$.
    \end{enumerate}
\end{proposition}
\begin{proof}
    We apply the classification of Proposition~\ref{prop:types-of-cubics} to $\det(\varphi)$. The polynomial forms of (iv) can be found in \cite{bruce-wall}.
\end{proof}

Ng listed almost all $G$-equivalence classes of tensors $\varphi\in\mathbb{P}(U^\vee\otimes V^\vee\otimes W^\vee)$ in \cite{ng} with the goal of eventually describing the GIT quotient $\mathbb{P}(U^\vee\otimes V^\vee\otimes W^\vee)//G$, however the latter paper was never published. We draw heavily on their complete enumeration of canonical tensors for the remainder of the paper. We derive our own results for non-canonical tensors as their classification is unfortunately incomplete -- it omits rank $2$ and $3$ skew-symmetric tensors, which are $G$-equivalent to
$$\varphi_U=u_0Q + \begin{pmatrix}
    0&v_2&-v_1\\-v_2&0&v_0\\v_1&-v_0&0
\end{pmatrix}$$
where $Q$ is a rank $2$ or $3$ matrix. To enumerate all canonical tensors, \cite{ng} uses the correspondence between determinantal representations of canonical cubic surfaces $S\subset\mathbb{P}U$ and admissible length $6$ subschemes $X\subset\mathbb{P}V$ (recall Proposition~\ref{prop:canonical-detreps}). Explicitly, they use Maple to enumerate all $\SL(V)$-equivalence classes of admissible length $6$ subschemes $X\subset\mathbb{P}V$ and then compute the $4\times3$ presentation matrix for the minimal free resolution of the ideal sheaf $\mathcal{I}_X$. 

Recall the subvarieties $\mathcal{S}(F)\subset\vert\mathcal{O}_{\mathbb{P}U}(3)\vert$, where $\mathcal{S}(F)$ is the set of cubic surfaces with $F$-type singularities. The image of the rational map 
$$\det:\mathbb{P}(U^\vee\otimes V^\vee\otimes W^\vee)\dashrightarrow\vert\mathcal{O}_{\mathbb{P}U}(3)\vert$$
is the subvariety $\mathbb{P}(U^\vee\otimes V^\vee\otimes W^\vee)\setminus\mathcal{S}(E_6)$ (recall that the $E_6$ cubic surface has no determinantal representations). 

\begin{definition}
    We let $\mathbb{S}:=\det^{-1}(\mathcal{S}(\emptyset))$ be the set of determinantal representations of smooth cubic surfaces, it is open inside $\mathbb{P}(U^\vee\otimes V^\vee\otimes W^\vee)$. For every other strict Dynkin subdiagram $\emptyset\ne F\subsetneq\widetilde{E}_6$ the preimage $\det^{-1}(\mathcal{S}(F))$ of $\mathcal{S}(F)$ splits into the disjoint union of a small number of (at most $6$) connected components. We label these connected components $F(n)$ with a Latin letter $n\in\{a,b,c,d,e,f\}$. 
\end{definition}

This labelling is consistent with Ng's notation, we refer the reader to \cite{ng} for a complete description of the points of each subvariety $F(n)$.

We use slightly different notation for $F=D_4$ to be consistent with \cite{ng}. The preimage $\det^{-1}(\mathcal{S}(D_4))$ is connected, but Ng decomposes the preimage as the disjoint union of two components $D_4(a)\cup D_4(b):=\det^{-1}(\mathcal{S}(D_4))$. This decomposition is consistent with the notation of Proposition~\ref{prop:S(F)-dimensions}, in the sense that the associated cubic surface $S_\varphi$ to a tensor $\varphi\in D_4(n)$ is isomorphic to $S_{D_4(n)}$. The subvarieties $D_4(a)$ and $D_4(b)$ are both $G$-orbits and $D_4(b)\subset\overline{D_4(a)}$. Each of the two isomorphism classes of $D_4$ cubic surfaces have a unique determinantal representation (up to $G$-equivalence). The technicalities of the determinantal representations of the $D_4$ cubic surfaces are unimportant for our results, we will see that all determinantal representations of a $D_4$ cubic surface are $G$-unstable. 

\begin{example}
    The subvariety $A_1A_2(a)\subset\mathbb{P}(U^\vee\otimes V^\vee\otimes W^\vee)$ is the $G$-orbit of tensors of the form 
    \[\varphi_U=\begin{pmatrix}
        \lambda(\lambda+1)u_1+(\lambda+1)u_3&u_3&u_0\\\lambda u_1+u_3&-u_1&u_3\\u_2&0&u_3
    \end{pmatrix},\]
    parametrised by $\lambda\in\mathbb{C}\setminus\{0,1\}$. 
\end{example}

Recall that $\chi:\mathbb{P}(U^\vee\otimes V^\vee\otimes W^\vee)\dashrightarrow\Hilb_6(\mathbb{P}V)$ is the rational map that sends a tensor $\varphi$ to its associated length $6$ subscheme $X_\varphi\subset\mathbb{P}V$, defined when $X_\varphi$ is length $6$. Using Propositions~\ref{prop:canonical-detreps} and \ref{prop:gamma-facts}, the subvarieties $F(n)$ can be equivalently described in terms of their associated length $6$ subschemes. The image $\chi(F(n))\subset\Hilb_6(\mathbb{P}V)$ is a locally closed subvariety parametrising canonical length $6$ subschemes satisfying certain geometric conditions (in terms of the number and type of non-reduced points and the ways in which they intersect with lines).

\begin{example}
\label{ex:A1A2(a)-subschemes}
    The subvariety $\chi(A_1A_2(a))\subset\Hilb_6(\mathbb{P}V)$ parametrises length $6$ subschemes $X\subset\mathbb{P}V$ satisfying (a) the following closed conditions:
    \begin{itemize}
        \item $X$ contains a length $2$ point $X_1\subset X$; and
        \item $X$ contains two length $3$ subschemes $X_2,X_3\subset X$ such that $X_2$ and $X_3$ are each contained in a line; and
        \item the subschemes $X_2$ and $X_3$ satisfy $X_1\subset X_2$ and $(X_1)_\textnormal{red}\subset X_3$.
    \end{itemize}
    and (b) the following open conditions:
    \begin{itemize}
        \item $X$ is supported on at least five points; and
        \item $X$ contains at most two length $3$ collinear subschemes; and
        \item $X$ contains no length $4$ collinear subscheme; and
        \item $X$ is not contained in a conic.
    \end{itemize}
\end{example}

The conditions defining each $\chi(F(n))$ are similar to those of Example~\ref{ex:A1A2(a)-subschemes} and are straightforward to work out from the pictorial descriptions in \cite{ng}. Formally, we decompose the open subset $\Hilb_6(\mathbb{P}V)^\textnormal{adm}\subset\Hilb_6(\mathbb{P}V)$ parametrising admissible length $6$ subschemes into locally closed subvarieties $\mathcal{H}_F$ indexed by the associated Dynkin diagram $F$ to a subscheme $X\subset\mathbb{P}V$ (where $F$ is the intersection graph of the $(-2)$-curves on the minimal resolution of singularities of the blowup $\Bl_X\mathbb{P}V$). For each Dynkin diagram $F\subsetneq\widetilde{E}_6$ the subvariety $\mathcal{H}_F$ decomposes as the disjoint union of a small number of connected components (which are $\chi(F(n))$), and one can check that each of these connected components are cut out by conditions similar to those of Example~\ref{ex:A1A2(a)-subschemes} (open and closed conditions in terms of the number and type of non-reduced points and how the point configuration intersects with lines, as well as the open condition that $X$ is not contained in a conic). The conditions defining $D_4(a)$ and $D_4(b)$ are slightly different, but this technicality is unimportant for our results. When $F\ne D_4$ we have

\begin{proposition}
\label{prop:surj-detmap}
    For each stratum $F(n)\subset\mathbb{P}(U^\vee\otimes V^\vee\otimes W^\vee)$ where $F\ne D_4$ the induced map $$\det\vert_{F(n)}:F(n)\to\mathcal{S}(F)$$
    is surjective. 
\end{proposition}

The fact that Proposition~\ref{prop:surj-detmap} doesn't apply to $D_4$ cubic surfaces is an artefact of our notation; every $D_4$ cubic surface has a determinantal representation and the set $D_4(a)\cup D_4(b)$ of determinantal representations of $D_4$ cubic surfaces is both connected and maps surjectively onto $\mathcal{S}(D_4)$.

In the moduli space of cubic surfaces we had $\mathcal{S}(F')\subset\overline{\mathcal{S}(F)}$ if and only if $F'\subset F$ is a Dynkin subdiagram. For tensors we have
\begin{proposition}
    Let $F,F'\subsetneq\widetilde{E}_6$ be Dynkin subdiagrams and $n,n'$ be Latin letters. A necessary condition for $F'(n')\subset\overline{F(n)}$ is that $F\subset F'$ as Dynkin diagrams. 
\end{proposition}
\begin{proof}
    Suppose $F'(n')\subset\overline{F(n)}$. Then, because $\det:\mathbb{P}(U^\vee\otimes V^\vee\otimes W^\vee)\dashrightarrow\vert\mathcal{O}_{\mathbb{P}U}(3)\vert$ is continuous and surjectively maps $F(n)\to \mathcal{S}(F)$ and $F'(n')\to\mathcal{S}(F)$, we have $\mathcal{S}(F')\subset\overline{\mathcal{S}(F)}$. This implies that $F'\subset F$. 
\end{proof}

The closure relations between the canonical strata of $\mathbb{P}(U^\vee\otimes V^\vee\otimes W^\vee)$ can be most easily worked out by passing to the Hilbert scheme $\Hilb_6(\mathbb{P}V)$. Concretely, we have
\begin{proposition}
\label{prop:closure-rules}
    Let $F(n),F'(n')\subset\mathbb{P}(U^\vee\otimes V^\vee\otimes W^\vee)$ be strata. We have $F'(n')\subset\overline{F'(n')}$ if and only if $\chi(F'(n'))\subset\overline{\chi(F(n))}$.
\end{proposition}
\begin{proof}
    Recall the commutative diagram
    \[\begin{tikzcd}
    	{\mathbb{P}(U^\vee\otimes V^\vee\otimes W^\vee)} & \\
    	{\Hilb_6(\mathbb{P}V)} & {\mathbb{P}(U^\vee\otimes V^\vee\otimes W^\vee)//(\SL(U)\times\SL(W)).}
    	\arrow["{\chi}"', dashed, from=1-1, to=2-1]
    	\arrow["{\pi_{UW}}", dashed, from=1-1, to=2-2]
    	\arrow["\gamma"', dashed, from=2-1, to=2-2]
    \end{tikzcd}\]
    When restricted to the locus of canonical tensors/admissible subschemes, the maps $\chi$ and $\pi_{UW}$ are morphisms and the map $\gamma$ is an isomorphism onto its image. Because each stratum of type $F(n)$ is $\SL(U)\times\SL(W)$-invariant then $F'(n')\subset\overline{F'(n')}$ if and only if $\pi_{UW}(F'(n'))\subset\overline{\pi_{UW}(F'(n'))}$. Passing through the isomorphism $\gamma^{-1}$ and using the identity $\chi=\gamma^{-1}\circ\pi_{UW}$ (when the right hand side is defined) gives the desired result. 
\end{proof}

A corollary is that for any stratum $F'(n')$ and Dynkin diagram $F\subset F'$ there is a letter $m$ such that $F'(n')\subset F(m)$ (this is straightforward to check by hand, although tedious, we don't use this fact). 

By Lemma~3.4 of \cite{twisted-cubics_lehn-lehn-sorger-straten}, the $\SL(V)\times\SL(W)$-stabiliser of any canonical tensor is finite (because every canonical tensor is $\SL(V)\times\SL(W)$-stable). Similarly, the $\SL(U)\times\SL(W)$- and $\SL(U)\times\SL(V)$-stabilisers are finite by Proposition~\ref{prop:UW-semistability}. As a result, whenever ${\varphi\in\mathbb{P}(U^\vee\otimes V^\vee\otimes W^\vee)}$ is a canonical tensor we have
\begin{equation}
\label{eqn:stab-dim-equalities}
    \dim\Stab_G(\varphi)=\dim\Stab_{\SL(U)}(S_\varphi)=\dim\Stab_{\SL(V)}(X_\varphi)=\dim\Stab_{\SL(W)}(Y_\varphi).
\end{equation}
Applying orbit-stabiliser to \eqref{eqn:stab-dim-equalities} gives 
$$\codim\big(F(n)\subset\mathbb{P}(U^\vee\otimes V^\vee\otimes W^\vee)\big)=\codim\big(\mathcal{S}(F)\subset\vert\mathcal{O}_{\mathbb{P}U}(3)\vert\big).$$
The dimension of each stratum $F(n)$ follows from the dimensions of the stratum $\mathcal{S}(F)$ listed in Proposition~\ref{prop:S(F)-dimensions}. We have
\begin{proposition}
\label{prop:tensor-dims}
    The subvariety $\mathbb{S}\subset\mathbb{P}(U^\vee\otimes V^\vee\otimes W^\vee)$, parametrising tensors $\varphi$ whose associated cubic surface $S_\varphi$ is smooth, is dense inside $\mathbb{P}(U^\vee\otimes V^\vee\otimes W^\vee)$ and $\dim\Stab_G(\varphi)=0$ whenever $\varphi\in\mathbb{S}$. More generally:
    \begin{itemize}
        \item Let $F\subsetneq\widetilde{E}_6$ be a Dyknin subdiagram where $F\ne 2A_2$ and $F\ne D_4$ and let $n$ be a Latin letter. Then $\codim\left(F(n)\right)=\vert F\vert$ and $\dim\Stab_G(\varphi)=\min\{0,\vert F\vert-4\}$ whenever $\varphi\in F(n)$. If additionally $\vert F\vert\geq 4$ then $F(n)$ is a $G$-orbit.
        \item Let $F=2A_2$ and $n\in\{a,b\}$. Then $\codim\left(2A_2(n)\right)=3$ and $\dim\Stab_G(\varphi)=1$ whenever $\varphi\in 2A_2(n)$.
        \item Let $F=D_4$. Then $\codim\left(D_4(a)\right)=4$ and $\codim\left(D_4(b)\right)=5$. We have $\dim\Stab_G(\varphi)=0$ if $\varphi\in D_4(a)$ and $\dim\Stab_G(\varphi)=1$ if $\varphi\in D_4(b)$.
    \end{itemize}
    The subvariety $\Sigma_3\subset\mathbb{P}(U^\vee\otimes V^\vee\otimes W^\vee)$ has codimension $7$ and $\dim\Stab_G(\varphi)=3$ when $\varphi\in\Sigma_3$ is a rank $3$ skew-symmetric tensor.
\end{proposition}
\begin{proof}
    All that remains is to prove is that $\dim\Stab_G(\varphi)=3$ when $\varphi\in\Sigma_3$. There is a unique $3$-dimensional subspace $U'\subset U$ such that $\varphi_U(U')\subset V^\vee\otimes W^\vee$. Recall from Definition \ref{defn:skew-sym-tensor} that there is a unique set of isomorphisms $U'\cong V\cong W\cong\mathbb{C}^3$ under which the map $\varphi_U\vert_{U'}:U'\to V^\vee\otimes W^\vee$ is equivalent to the Hodge star map $\star:\mathbb{C}^3\to\wedge^2(\mathbb{C}^3)^\vee\subset(\mathbb{C}^3)^\vee\otimes(\mathbb{C}^3)^\vee$. It is then a straightforward calculation that $\Stab_G(\varphi)$ is a degree $9$ extension of the complex special orthogonal group $\operatorname{SO}(3,\mathbb{C})$, where $\operatorname{SO}(3,\mathbb{C})$ acts diagonally on $U'\otimes V\otimes W$ via the isomorphisms ${U'\cong V\cong W\cong\mathbb{C}^3}$. The degree $9$ extension comes from the extra freedom of two sets of third roots of unity. We have ${\dim\operatorname{SO}(3,\mathbb{C})=3}$, so $\dim\Stab_G(\varphi)=3$. The subvariety $\Sigma_3$ is a $G$-orbit, so ${\codim(\Sigma_3\subset\mathbb{P}(U^\vee\otimes V^\vee\otimes W^\vee))=7}$ by orbit-stabiliser. 
\end{proof}

\subsection{$G$-semistability}
\label{subsec:G-semistability}

The goal of Subsection~\ref{subsec:G-semistability} is to prove the following criterion for $G$-semi\-stability:

\begin{theorem}
\label{thm:G-semistability}
    A tensor $\varphi\in\mathbb{P}(U^\vee\otimes V^\vee\otimes W^\vee)$ is $G$-semistable if and only if at least one of the following holds:
    \begin{enumerate}[label=(\roman*)]
        \item the associated cubic surface is $\SL(U)$-semistable; or
        \item the surface $S_\varphi$ contains no rank at most $1$ points. 
    \end{enumerate}
\end{theorem}

Our strategy is the following:
\begin{enumerate}[label=(\alph*)]
    \item We start by showing the conditions of Theorem~\ref{thm:G-semistability} are sufficient.
    \item We then give some necessary conditions for $G$-semistability using the Hilbert-Mumford criterion. 
    \item We split the points $\varphi\in\mathbb{P}(U^\vee\otimes V^\vee\otimes W^\vee)$ into the five cases of Proposition~\ref{prop:types-of-tensors}. We then show every tensor $\varphi$ is either semistable by (a) or unstable by (b), proving the sufficient conditions of (a) are actually necessary. 
\end{enumerate}

The first sufficient condition for $G$-semistability comes from the $\SL(U)$-semistability of cubic surfaces. Recall the $\SL(V)\times\SL(W)$-invariant and $\SL(U)$-equivariant determinant map 
$$\det:U^\vee\otimes V^\vee\otimes W^\vee\longrightarrow\Sym^3(U^\vee),$$
sending a tensor $\varphi$ to its determinant cubic polynomial $\det(\varphi)$. We have

\begin{proposition}
\label{prop:cubic-ss-pullback}
    Let $\varphi\in\mathbb{P}(U^\vee\otimes V^\vee\otimes W^\vee)$ be a tensor such that $S_\varphi$ is an $\SL(U)$-semistable cubic surface. Then $\varphi$ is $G$-semistable. 
\end{proposition}
\begin{proof}
    If $[S_\varphi]\in\vert\mathcal{O}_{\mathbb{P}U}(3)\vert$ is $\SL(U)$-semistable then there is a non-constant, $\SL(U)$-invariant function $f\in\mathbb{C}[\Sym^3U]_\bullet$ where $f([S_\varphi])\ne0$. The pullback $f\circ\det\in\mathbb{C}[U\otimes V\otimes W]_\bullet$ is a non-constant \mbox{$G$-invariant} function and $(f\circ\det)(\varphi)=f([S_\varphi])\ne0$. So $\varphi$ is $\SL(U)$-semistable. 
\end{proof}

The $\SL(U)$-semistability of cubic surfaces is well known, for now all we need is the following result of Mumford \cite[Subsection~1.14]{mumford-proj-var}:
\begin{proposition}
\label{prop:cubic-semistability}
    A cubic surface $[S]\in\vert\mathcal{O}_{\mathbb{P}U}(3)\vert$ is:
    \begin{itemize}
        \item $\SL(U)$-stable if and only if $S$ is canonical with at worst $4A_1$ singularities (equivalently, every singularity on $S$ is an $A_1$ singularity);
        \item $\SL(U)$-semistable if and only if $S$ is canonical with at worst $4A_1$ or $3A_2$ singularities (equivalently, every singularity on $S$ is either an $A_1$ or an $A_2$ singularity). 
    \end{itemize}
\end{proposition}

The other sufficient condition comes from rank conditions on the tensor $\varphi$. 

\begin{proposition}
\label{prop:R24}
    There is an irreducible, homogeneous, $G$-invariant, degree $24$ polynomial \linebreak${R_{24}\in\mathbb{C}[U\otimes V\otimes W]_\bullet}$ where $R_{24}(\varphi)=0$ if and only if the associated cubic surface $S_\varphi$ contains a rank at most $1$ point. 
\end{proposition}
\begin{proof}
    Consider the incidence variety
    $$\widetilde{\mathcal{R}}:=\left\{(\varphi,u,V')\vert \varphi(u,V',-)=0\right\}\subset\mathbb{P}(U^\vee\otimes V^\vee\otimes W^\vee)\times\mathbb{P}U\times\Gr(2,V).$$
    The variety $\widetilde{\mathcal{R}}$ comes equipped with two projections \mbox{$\pi_1:\widetilde{\mathcal{R}}\to\mathbb{P}(U^\vee\otimes V^\vee\otimes W^\vee)$} and \linebreak $\pi_2:\widetilde{\mathcal{R}}\to\mathbb{P}U\times\Gr(2,V)$. 
    We let $\mathcal{R}\subset\mathbb{P}(U^\vee\otimes V^\vee\otimes W^\vee)$ be the image $\pi_1(\widetilde{\mathcal{R}})$ of $\pi_1$, the closed points $\varphi\in \mathcal{R}$ are the tensors that contain a rank at most $1$ point. It suffices to prove that $\mathcal{R}$ is a $G$-invariant, degree $24$ prime divisor; note that $\mathcal{R}$ is $G$-invariant by construction. 

    We first prove the map $\pi_1:\widetilde{\mathcal{R}}\to \mathcal{R}$ is generically finite. It suffices to exhibit one tensor with exactly one rank $1$ point, consider
    \[\varphi_U = \begin{pmatrix}
            u_0&u_3&-u_2\\-u_3&0&u_1\\u_2&-u_1&0
        \end{pmatrix}.\]
    One can check from the $2\times2$ minors that $\varphi_U(u)$ is rank at most $1$ if and only if $u=[1:0:0:0]\in\mathbb{P}U$. 
    
    Now, the projection $\pi_2$ is surjective and the fibre above a point $(u,V')\in\mathbb{P}U\times\Gr(2,V')$ is cut out by $6$ independent linear equations on $\mathbb{P}(U^\vee\otimes V^\vee\otimes W^\vee)\times\{(u,V')\}$. So $\widetilde{\mathcal{R}}$ is codimension $6$ inside $\mathbb{P}(U^\vee\otimes V^\vee\otimes W^\vee)\times\mathbb{P}U\times\Gr(2,V')$. Then we have
    \[\codim(\mathcal{R})=\codim(\widetilde{\mathcal{R}})-\dim(\mathbb{P}U\times\Gr(2,V'))=6-5=1.\]
    Because the projection $\pi_2$ is surjective onto an irreducible variety and the fibres of $\pi_2$ are all irreducible and equidimensional, then $\widetilde{\mathcal{R}}$ is irreducible. So $\pi_1(\widetilde{\mathcal{R}})=\mathcal{R}$ is an irreducible, closed, $G$-invariant subscheme of codimension $1$, i.e. a $G$-invariant prime divisor. 

    To compute the degree we consider a general line $L\subset\mathbb{P}(U^\vee\otimes V^\vee\otimes W^\vee)$. The line $L$ is parametrised by a morphism $l:\mathbb{P}^1\to\mathbb{P}(U^\vee\otimes V^\vee\otimes W^\vee)$ sending $[s:t]\mapsto s\varphi+t\varphi'$ for two sufficiently general points $\varphi,\varphi'\in\mathbb{P}(U^\vee\otimes V^\vee\otimes W^\vee)$. The degree of $\mathcal{R}$ is the intersection number $L\cdot\mathcal{R}$, which is the intersection number of the Segre variety $\mathbb{P}V^\vee\times\mathbb{P}W^\vee\subset\mathbb{P}(V^\vee\otimes W^\vee)$ with the image of the bilinear map $f:\mathbb{P}^1\times\mathbb{P}U\to\mathbb{P}(V^\vee\otimes W^\vee)$ that sends $([s:t],u)\mapsto s\varphi_U(u)+t\varphi'_U(u)$.

    We compute this intersection number in the Chow ring. Let $h_1,h_2,h_3$ be the respective generators of the Chow rings $A_\ast(\mathbb{P}^1)$ and $A_\ast(\mathbb{P}U)$ and $A_\ast(\mathbb{P}(V^\vee\otimes W^\vee))$. We have $f^\ast(h_3)=h_1+h_2$ and $f^\ast(\mathbb{P}V^\vee\times\mathbb{P}W^\vee)=f^\ast(6h_3^4)=6f^\ast(h_3)^4$. And $(h_1+h_2)^4=4h_1h_2^3$, so $\deg(\mathcal{R})=4\cdot6=24$. 
\end{proof}

\begin{definition}
\label{defn:R-R24}
    As above, we let $\mathcal{R}\subset\mathbb{P}(U^\vee\otimes V^\vee\otimes W^\vee)$ be the subvariety parametrising tensors $\varphi\in\mathbb{P}(U^\vee\otimes V^\vee\otimes W^\vee)$ whose associated cubic surface $S_\varphi$ contains a rank at most $1$ point. We let $R_{24}\in\mathbb{C}[U\otimes V\otimes W]_\bullet$ be a generator of the principal ideal cutting out $\mathcal{R}\subset\mathbb{P}(U^\vee\otimes V^\vee\otimes W^\vee)$.
\end{definition}

An immediate consequence of Proposition~\ref{prop:R24} is

\begin{proposition}
\label{prop:no-rank-one-semistable}
    Let $\varphi\in\mathbb{P}(U^\vee\otimes V^\vee\otimes W^\vee)\setminus\mathcal{R}$ be a tensor with no rank at most $1$ points. Then $\varphi$ is $G$-semistable.
\end{proposition}
\begin{proof}
    The non-constant, $G$-invariant polynomial $R_{24}\in\mathbb{C}[U\otimes V\otimes W]_\bullet$ does not vanish at $\varphi$.
\end{proof}

Our two sufficient conditions for $G$-semistability are Propositions~\ref{prop:cubic-ss-pullback} and \ref{prop:no-rank-one-semistable}. Before applying the Hilbert-Mumford criterion to prove these two conditions are necessary, we introduce some notation. Let $B$ be a $4\times3\times3$ array of $0$s and $\ast$s, in the sense that $B=[B_0\vert B_1\vert B_2\vert B_3]$ and each $B_i$ is a $3\times3$ matrix of $0$s and $\ast$s. A tensor $\varphi\in\mathbb{P}(U^\vee\otimes V^\vee\otimes W^\vee)$ is said to be \emph{of block form $B$} if there is a basis for $U$ and $V$ and $W$ where $\varphi_U=\sum_{i=0}^3u_i\varphi_i$ and each $\varphi_i\in\Hom(V,W^\vee)\cong \Hom(\mathbb{C}^3,\mathbb{C}^3)$ has a $0$ in every entry that $B_i$ has a $0$ (the matrix $\varphi_i$ may also have $0$s where there is an $\ast$ in $B_i$). Note that we choose the ordering so that $\SL(V)$ acts on the columns of $B_i$ and $\SL(W)$ acts on the rows of $B_i$.

\begin{proposition}
\label{prop:G-ss-hilbertmumford}
    A tensor $\varphi\in\mathbb{P}(U^\vee\otimes V^\vee\otimes W^\vee)$ is $G$-unstable if $\varphi$ is of one of the following block forms:
    \begin{enumerate}[label=$(u1)\hspace{\Tspace}$]
        \item \hfill$\left[\begin{array}{ccc|ccc|ccc|ccc}
             \ast&\ast&\ast&\ast&\ast&\ast&\ast&\ast&\ast&0&0&0  \\
             \ast&\ast&\ast&\ast&\ast&\ast&\ast&\ast&\ast&0&0&0  \\
             \ast&\ast&\ast&\ast&\ast&\ast&\ast&\ast&\ast&0&0&0  \\
        \end{array}\right]$\hfill\bigskip\label{u1}
    \end{enumerate}
    \begin{enumerate}[label=$(u2)\hspace{\Tspace}$]
        \item \hfill$\left[\begin{array}{ccc|ccc|ccc|ccc}\ast&\ast&\ast&\ast&\ast&\ast&\ast&\ast&\ast&\ast&\ast&\ast  \\
         \ast&\ast&\ast&\ast&\ast&\ast&\ast&\ast&\ast&\ast&\ast&\ast  \\
         \ast&\ast&\ast&0&0&0&0&0&0&0&0&0\\
    \end{array}\right]$\hfill\bigskip\label{u2}
    \end{enumerate}
    \begin{enumerate}[label=$(u2)^T$]
        \item \hfill$\left[\begin{array}{ccc|ccc|ccc|ccc}         \ast&\ast&\ast&\ast&\ast&0&\ast&\ast&0&\ast&\ast&0\\
         \ast&\ast&\ast&\ast&\ast&0&\ast&\ast&0&\ast&\ast&0\\
         \ast&\ast&\ast&\ast&\ast&0&\ast&\ast&0&\ast&\ast&0\\
    \end{array}\right]$\hfill\bigskip\label{u2T}
    \end{enumerate}
    \begin{enumerate}[label=$(u3)\hspace{\Tspace}$]
        \item \hfill$\left[\begin{array}{ccc|ccc|ccc|ccc}\ast&\ast&\ast&\ast&\ast&\ast&\ast&\ast&\ast&\ast&\ast&\ast  \\
         \ast&\ast&\ast&\ast&0&0&\ast&0&0&\ast&0&0  \\
         \ast&\ast&\ast&\ast&0&0&\ast&0&0&\ast&0&0  \\
    \end{array}\right]$\hfill\bigskip\label{u3}
    \end{enumerate}
    \begin{enumerate}[label=$(u4)\hspace{\Tspace}$]
        \item \hfill$\left[\begin{array}{ccc|ccc|ccc|ccc}
             \ast&\ast&\ast&\ast&\ast&\ast&\ast&\ast&\ast&\ast&0&0  \\
             \ast&\ast&\ast&\ast&\ast&\ast&\ast&0&0&0&0&0\\
             \ast&\ast&\ast&\ast&\ast&\ast&\ast&0&0&0&0&0
        \end{array}\right]$\hfill\bigskip\label{u4}
    \end{enumerate}
    \begin{enumerate}[label=$(u5)\hspace{\Tspace}$]
        \item \hfill$\left[\begin{array}{ccc|ccc|ccc|ccc}
             \ast&\ast&\ast&\ast&\ast&\ast&\ast&\ast&0&\ast&0&0  \\
             \ast&\ast&\ast&\ast&\ast&\ast&\ast&\ast&0&\ast&0&0  \\
             \ast&\ast&0&\ast&\ast&0&\ast&0&0&0&0&0\\
            \end{array}\right]$\hfill\bigskip\label{u5}
    \end{enumerate} 
    \begin{enumerate}[label=$(u5)^T$]
        \item \hfill$\left[\begin{array}{ccc|ccc|ccc|ccc}
         \ast&\ast&\ast&\ast&\ast&\ast&\ast&\ast&\ast&\ast&\ast&0  \\
         \ast&\ast&\ast&\ast&\ast&\ast&\ast&\ast&0&0&0&0  \\
         \ast&\ast&0&\ast&\ast&0&0&0&0&0&0&0\\
        \end{array}\right]$\hfill\label{u5T}
    \end{enumerate}
\end{proposition}
\begin{proof}
    We apply the Hilbert-Mumford criterion (see, for example, Subsection~3.1 of \cite{hoskins-git-perspectives}). Let $\lambda:\mathbb{G}_m\to G$ be a one-parameter subgroup that acts diagonally in a basis for $U$ and $V$ and $W$. The subgroup $\lambda$ has \emph{weight vector} $\wt(\lambda)=((a_0,a_1,a_2,a_3),(b_0,b_1,b_2),(c_0,c_1,c_2))\in\mathbb{Z}^4\times\mathbb{Z}^3\times\mathbb{Z}^3$, in the sense that $\lambda$ acts as
    $$\lambda(t)\cdot\sum_{i,j,k}\varphi_{ijk}(u_i\otimes v_j\otimes w_k)=t^{a_i+b_j+c_k}\varphi_{ijk}(u_i\otimes v_j\otimes w_k).$$
    The exponent $a_i+b_j+c_k$ is said to be the \emph{weight} with which $\lambda$ acts on the vector $u_i\otimes v_j\otimes w_k$. The Hilbert-Mumford criterion states that a point $\varphi\in\mathbb{P}(U^\vee\otimes V^\vee\otimes W^\vee)$ is $G$-unstable if and only if there exists a one-parameter subgroup $\lambda:\mathbb{G}_m\to G$ that acts with only \emph{positive} weights on $\varphi$, in the sense that the coefficient $\varphi_{ijk}=0$ whenever $a_i+b_j+c_k\leq0$. While not all one-parameter subgroups $\lambda:\mathbb{G}_m\to G$ act diagonally, it is enough to only check the action of diagonal subgroups. 

    To each weight vector $w=((a_0,a_1,a_2,a_3),(b_0,b_1,b_2),(c_0,c_1,c_2))\in\mathbb{Z}^4\times\mathbb{Z}^3\times\mathbb{Z}^3$ we associate a block form $B$, where $B$ has a $0$ in the $ijk^\textnormal{th}$ entry whenever $a_i+b_j+c_k\leq0$. If a tensor $\varphi$ is of block form $B$ then there is a basis for $U$ and $V$ and $W$ where $\varphi_{ijk}=0$ whenever $a_i+b_j+c_k\leq0$. The one-parameter subgroup $\lambda_w:\mathbb{G}_m\to G$ that is diagonal in this basis with weight vector $\operatorname{wt}(\lambda_w)=w$ acts with only positive weights on $\varphi$. So $\varphi$ is $G$-unstable. 

    To complete the proof we need to exhibit a weight vector associated to each of the above block forms. The pairs of the weight vector $w$ and its associated $G$-unstable block form $B$ are the following:
    \[\begin{array}{c @{\hspace{0.5cm}}|@{\hspace{0.5cm+\Tspace}} l}
        \hspace{\Tspace}B\hspace{\Tspace} & \multicolumn{1}{c}{w} \\
        \hline
        \hspace{\Tspace}\text{\ref{u1}} & ((1,1,1,-3),(0,0,0),(0,0,0))\\
        \hspace{\Tspace}\text{\ref{u2}} & ((9, -3, -3, -3),(4, 4, -8), (0, 0,0))\\
        \hspace{\Tspace}\text{\ref{u2T}} & ((9, -3, -3, -3),(0, 0, 0), (4, 4, -8))\\
        \hspace{\Tspace}\text{\ref{u3}} & ((9, -3, -3, -3),(8, -4, -4), (8, -4, -4))\\
        \hspace{\Tspace}\text{\ref{u4}} & ((9, 9, -3, -15),(8, -4, -4), (8, -4, -4))\\
        \hspace{\Tspace}\text{\ref{u5}} & ((9, 9, -3, -15),(12, 0, -12), (4, 4, -8))\\
        \hspace{\Tspace}\text{\ref{u5T}} & ((9, 9, -3, -15),(4, 4, -8),(12, 0, -12)).
    \end{array}\]
\end{proof}

The conditions for instability in Proposition~\ref{prop:G-ss-hilbertmumford} were found with an algorithmic implementation of the Hilbert-Mumford criterion using convex geometry, see Subsection~3.1 of \cite{hoskins-git-perspectives} for details on the convex criterion. This algorithm outputs $18$ maximal $G$-unstable block forms, reducing to $12$ after taking transposes, which further reduces to the $5$ forms above after suitable changes of basis. We provide these $12$ forms in Appendix \ref{appendix:block-forms} for completeness. For the purpose of our proofs we also list block forms $(u10)$ and $(u10)^T$ which are subforms of \ref{u4}\hspace{-\Tspace}, and block form $(u12)$ which is a subform of both \ref{u5}\hspace{-\Tspace} and \ref{u5T}.
\begin{enumerate}[label=$(u10)\hspace{\Tspace}$]
    \item associated to the weight vector $((3, 3, -3, -3),(0, 0, 0), (4, -2, -2))$:
    \[\left[\begin{array}{ccc|ccc|ccc|ccc}
         \ast&\ast&\ast&\ast&\ast&\ast&\ast&\ast&\ast&\ast&\ast&\ast  \\
         \ast&\ast&\ast&\ast&\ast&\ast&0&0&0&0&0&0\\
         \ast&\ast&\ast&\ast&\ast&\ast&0&0&0&0&0&0\\
    \end{array}\right]\]\label{u10}
\end{enumerate}
\begin{enumerate}[label=$(u10)^T$]
    \item associated to the weight vector $((3, 3, -3, -3),(4, -2, -2),(0, 0, 0))$:
    \[\left[\begin{array}{ccc|ccc|ccc|ccc}
         \ast&\ast&\ast&\ast&\ast&\ast&\ast&0&0&\ast&0&0  \\
         \ast&\ast&\ast&\ast&\ast&\ast&\ast&0&0&\ast&0&0\\
         \ast&\ast&\ast&\ast&\ast&\ast&\ast&0&0&\ast&0&0\\
    \end{array}\right]\]\label{u10T}
\end{enumerate}
\begin{enumerate}[label=$(u12)\hspace{\Tspace}$]
    \item associated to the weight vector $(3, 3, -3, -3),(2, 2, -4), (2, 2, -4))$:
    \[\left[\begin{array}{ccc|ccc|ccc|ccc}
         \ast&\ast&\ast&\ast&\ast&\ast&\ast&\ast&0&\ast&\ast&0  \\
         \ast&\ast&\ast&\ast&\ast&\ast&\ast&\ast&0&\ast&\ast&0  \\
         \ast&\ast&0&\ast&\ast&0&0&0&0&0&0&0\\
    \end{array}\right]\]\label{u12}
\end{enumerate}
In principle, the Hilbert-Mumford criterion also gives a sufficient condition for $G$-semistability, as long as one can prove they have generated all possible $G$-unstable block forms. We instead use the sufficient conditions of Propositions~\ref{prop:cubic-ss-pullback} and \ref{prop:no-rank-one-semistable}.

The first three $G$-unstable block forms in Proposition~\ref{prop:G-ss-hilbertmumford} have straightforward geometric interpretations. It is easy to check that
\begin{proposition}
\label{prop:123-geometric-unstable}
    A tensor $\varphi\in\mathbb{P}(U^\vee\otimes V^\vee\otimes W^\vee)$ is $G$-unstable of block form:
    \begin{itemize}
        \item \ref{u1} if and only if the map $\varphi_U:U\to V^\vee\otimes W^\vee$ is not injective;
        \item \ref{u2} if and only if the surface $S_\varphi\subset\mathbb{P}U$ contains a plane of type $(3,1)$;
        \item \ref{u2T} if and only if the surface $S_\varphi\subset\mathbb{P}U$ contains a plane of type $(1,3)$;
        \item \ref{u3} if and only if the surface $S_\varphi\subset\mathbb{P}U$ contains a plane of type $(2,2)$.
    \end{itemize}
\end{proposition}
In particular, if a tensor $\varphi\in\mathbb{P}(U^\vee\otimes V^\vee\otimes W^\vee)$ is of one of the block forms \ref{u1}\hspace{-\Tspace} or \ref{u2}\hspace{-\Tspace} or \ref{u2T} or \ref{u3}\hspace{-\Tspace} then the associated surface $S_\varphi$ is either a cone or reducible, and conversely every $G$-unstable determinantal representation of a canonical cubic surface is of block form \ref{u4}\hspace{-\Tspace} or \ref{u5}\hspace{-\Tspace} or \ref{u5T}. 

We will now prove that every tensor $\varphi\in\mathbb{P}(U^\vee\otimes V^\vee\otimes W^\vee)$ either satisfies one of the conditions for $G$-semistability in Propositions~\ref{prop:cubic-ss-pullback} and \ref{prop:no-rank-one-semistable}, or is $G$-unstable by Proposition~\ref{prop:G-ss-hilbertmumford}. We will tackle each of the five cases of Proposition~\ref{prop:types-of-tensors} in order, starting with the case that $\varphi$ is a canonical tensor.

The set of canonical tensors $\varphi\in\mathbb{P}(U^\vee\otimes V^\vee\otimes W^\vee)$ whose associated cubic surface $S_\varphi$ is \mbox{$\SL(U)$-unstable} are all of type $F(n)$, where $A_3\subset F\subsetneq\widetilde{E}_6$ is a Dynkin diagram such that $F\ne E_6$. We have the following closure relations between these strata:

\begin{proposition}
\label{prop:unstable-canonical-closures}
    The sets $\overline{A_3(a)},\overline{A_3(b)},\overline{A_3(c)}\subset\mathbb{P}(U^\vee\otimes V^\vee\otimes W^\vee)$ contain the following subsets:
    \[\begin{array}{c|c}
        F(n) & \textnormal{$F'(n')$ appears if and only if $F'(n')\subset\overline{F(n)}$} \\ \hline
        A_3(a) & A_1A_3(a), A_4(a), A_4(b), 2A_1A_3(a), A_1A_4(a), A_1A_4(b), A_1A_5(a), D_4(a), D_4(b), D_5(a), D_5(b) \\ \hline
        A_3(b) & A_1A_3(d), A_4(a), A_4(c), 2A_1A_3(b), A_1A_4(a), A_1A_4(d), A_5(a), A_1A_5(a), D_4(a), D_4(b), D_5(a), D_5(b) \\ \hline
        A_3(c) & A_1A_3(b), A_1A_3(e), A_4(b), A_4(d), 2A_1A_3(c), A_1A_4(b), A_1A_4(c), A_5(a), A_1A_5(a), D_4(a), D_4(b), D_5(a)
    \end{array}\]
    If $A_3\subset F\subset\widetilde{E}_6$ and $F\ne E_6$ then either $F(n)\subset \overline{A_3(a)}\cup\overline{A_3(b)}\cup\overline{A_3(c)}$ or $F(n)=A_3(d)$ or $F(n)=A_1A_3(f)$.
\end{proposition}
\begin{proof}
    The proof amounts to carefully checking all of the closure relations between subschemes of the form $\chi(F(n))\subset\Hilb_6(\mathbb{P}V)$ using the lists of associated subschemes in \cite{ng}, and then passing to subschemes of the form $F(n)\subset\mathbb{P}(U^\vee\otimes V^\vee\otimes W^\vee)$ using Proposition~\ref{prop:closure-rules}.
\end{proof}

We are now ready to prove 
\begin{lemma}
\label{lem:canonical-semistability}
    Let $\varphi\in\mathbb{P}(U^\vee\otimes V^\vee\otimes W^\vee)$ be a tensor whose associated cubic surface $S_\varphi$ is canonical. Then $\varphi$ is $G$-semistable if and only if either $S_\varphi$ is $\SL(U)$-semistable or $S_\varphi$ contains no rank at most $1$ points. A $G$-semistable tensor $\varphi$ satisfies exactly one of the following conditions:
    \begin{enumerate}[label=(\roman*)]
        \item $S_\varphi$ has at worst $4A_1$ or $3A_2$ singularities; or
        \item $\varphi\in A_3(d)$ and $\varphi_U$ is $G$-equivalent to the matrix
        \[\begin{pmatrix}
            (\alpha-1)u_0+u_1&-u_1& u_0\\
            u_2&0&u_3\\
            u_0+u_3&\alpha u_0+u_1+u_2&0
        \end{pmatrix},\]
        where $\alpha\in\mathbb{C}\setminus\{0\}$ is a parameter that determines the isomorphism class of $S_\varphi$; or 
        \item $\varphi\in A_1A_3(f)$ and $\varphi_U$ is $G$-equivalent to the matrix
        \[\begin{pmatrix}
            u_1&u_0&0\\u_2&0&u_3-u_0-u_1\\
            u_3&u_2&u_1
        \end{pmatrix}.\]
    \end{enumerate}
\end{lemma}
\begin{proof}
    Note that by Proposition~\ref{prop:cubic-semistability} (a result of Mumford), a cubic surface $S$ is $\SL(U)$-semistable if and only if $S$ has at worst $4A_1$ or $3A_2$ singularities. The complement of this condition is that a canonical surface $S$ is $\SL(U)$-unstable if and only if $S$ has an $A_3$ (or worse) singularity, in particular $\vert\mathcal{O}_{\mathbb{P}U}(3)\vert^{\SL(U)-us}=\overline{\mathcal{S}(A_3)}$. 

    Fix canonical tensor $\varphi$. By Proposition~\ref{prop:cubic-ss-pullback}, if the associated cubic surface $S_\varphi$ is \mbox{$\SL(U)$-semistable} then $\varphi$ is $G$-semistable. And if $\varphi\in A_3(d)$ or $\varphi\in A_1A_3(f)$ then $S_\varphi$ has no rank $1$ points (this is easy to check from non-vanishing of the $2\times2$ minors), so $\varphi$ is $G$-semistable by Proposition~\ref{prop:no-rank-one-semistable}. By Proposition~\ref{prop:unstable-canonical-closures}, the only remaining possibility is that $\varphi\in\overline{A_3(a)}\cup\overline{A_3(b)}\cup\overline{A_3(c)}$. 
    
    So it suffices to prove that every tensor of type $A_3(a)$ or $A_3(b)$ or $A_3(c)$ is $G$-unstable. Using the classification in \cite{ng}, this is immediate:
    \begin{itemize}
        \item Every tensor $\varphi\in A_3(a)$ is of the form
        \[\varphi=\left[\begin{array}{ccc|ccc|ccc|ccc}
             0&0&0 & 0&0&1 & 0&1&0 & 0&0&0  \\
             0&0&0 & 0&0&1 & 0&0&0 & 1&0&0 \\
             0&0&1 & 1&-1&0 & -\alpha &0&0 & 0&1&0
        \end{array}\right]\]
        for some $\alpha\in\mathbb{C}\setminus\{0\}$. After permuting some rows and columns, $\varphi$ is of block form \ref{u4}\hspace{-\Tspace} (the order of the rows, columns, and $3\times3$ blocks all need to be reversed, the reader can easily apply this permutation by rotating the page by $180$ degrees).
        \item Every tensor $\varphi\in A_3(b)$ is of the form
        \[\varphi=\left[\begin{array}{ccc|ccc|ccc|ccc}
             0&0&0 & 0&0&1 & 0&0&0 & 0&1&0 \\
             0&0&0 & 0&1&0 & 1&1&0 & 0&1&0 \\
             0&0&1 & 0&-\beta&0 & 0&-\beta&0 & 1&-\beta&0
        \end{array}\right]\]
        for some $\beta\in\mathbb{C}\setminus\{0\}$. After the same $180$ degree rotation, $\varphi$ is of block form \ref{u5}\hspace{-\Tspace}.
        \item Every tensor $\varphi\in A_3(c)$ is of the form
        \[\varphi=\left[\begin{array}{ccc|ccc|ccc|ccc}
             0&0&0 & 0&0&0 & 0&1&1 & 0&-1&0  \\
             0&0&0 & 0&0&1 & 1&0&0 & 0&0&0 \\
             0&1&-1 & 1&0&2\beta-1 & 0&0&0 & \beta&0&0
        \end{array}\right]\]
        for some $\beta\in\mathbb{C}\setminus\{0\}$. After the same $180$ degree rotation, $\varphi$ is of block form \ref{u5T}.
    \end{itemize}

    The matrix presentations of types $A_3(d)$ and $A_1A_3(f)$ also come from \cite{ng}. 
\end{proof}

A general tensor of block form \ref{u4}\hspace{-\Tspace} or \ref{u5}\hspace{-\Tspace} or \ref{u5T} is of type $A_3(a)$ or $A_3(b)$ or $A_3(c)$, respectively. As a result,
\begin{proposition}
\label{prop:45-geometric-unstable}
    A tensor $\varphi\in\mathbb{P}(U^\vee\otimes V^\vee\otimes W^\vee)$ is $G$-unstable of block form:
    \begin{itemize}
        \item \ref{u4} if and only if $\varphi\in\overline{A_3(a)}$;
        \item \ref{u5} if and only if $\varphi\in\overline{A_3(b)}$;
        \item \ref{u5T} if and only if $\varphi\in\overline{A_3(c)}$.
    \end{itemize}
\end{proposition}

Propositions~\ref{prop:123-geometric-unstable} and \ref{prop:45-geometric-unstable} combine to give a geometric interpretation of the maximal $G$-unstable block forms that arose from the Hilbert-Mumford criterion, which are listed in Proposition~\ref{prop:G-ss-hilbertmumford}. We now tackle the case where the surface $S_\varphi$ is reducible.

\begin{lemma}
\label{lem:reducible-semistable}
    Let $\varphi\in\mathbb{P}(U^\vee\otimes V^\vee\otimes W^\vee)$ be a tensor whose associated cubic surface $S_\varphi$ is reducible. Then $\varphi$ is $G$-semistable if and only if $S_\varphi$ contains no rank at most $1$ points. In this case, $\varphi$ is a rank $3$ skew-symmetric tensor, which is $G$-equivalent to
    \[\sigma = \begin{pmatrix}
        u_0&u_3&-u_2\\-u_3&u_0&u_1\\u_2&-u_1&u_0
    \end{pmatrix}.\]
\end{lemma}
\begin{proof}
    Because $S_\varphi$ is reducible then $S_\varphi$ contains a hyperplane $\mathbb{P}U'\subset\mathbb{P}U$. If $\mathbb{P}U'$ is a plane of type $(3,1)$ or $(1,3)$ or $(2,2)$ then, by Proposition~\ref{prop:G-ss-hilbertmumford}, the tensor $\varphi$ is $G$-unstable of block form \ref{u2}\hspace{-\Tspace} or \ref{u2T} or \ref{u3}\hspace{-\Tspace}, respectively. By Proposition~\ref{prop:rank2-subspaces} the only remaining possibility is that $\varphi_U(U')\subset V^\vee\otimes W^\vee$ is a skew-symmetric subspace. 
    
    If $\varphi$ is a skew-symmetric tensor then, by Proposition~\ref{prop:skew-sym-equivalence}, $\varphi$ is $G$-equivalent to the tensor
    \[\sigma' = \begin{pmatrix}
        \lambda_1u_0&u_3&-u_2\\-u_3&\lambda_2u_0&u_1\\u_2&-u_1&\lambda_3u_0
    \end{pmatrix},\]
    where each $\lambda_i$ is either $0$ or $1$ and $\lambda_1\geq\lambda_2\geq\lambda_3$. If $\lambda_3=0$ then $\sigma'$ (and hence $\varphi$) is of block form $(u12)$ and is $G$-unstable. Otherwise, $\varphi$ is $G$-equivalent to $\sigma$. It is easy to check from the non-vanishing of the $2\times2$ minors that $S_\sigma$ has no rank at most $1$ points, so every tensor $\varphi\in\Sigma_3=G\cdot\sigma$ is $G$-semistable by Proposition~\ref{prop:no-rank-one-semistable}. 
\end{proof}

We now tackle the case where the surface $S_\varphi$ is a cone over a cubic curve.

\begin{lemma}
\label{lem:conical-unstable}
    Let $\varphi\in\mathbb{P}(U^\vee\otimes V^\vee\otimes W^\vee)$ be a tensor whose associated cubic surface $S_\varphi$ is a cone over a cubic curve. Then $\varphi$ is $G$-unstable. 
\end{lemma}
\begin{proof}
    We choose coordinates so that the polynomial $\det(\varphi_U)$ is a polynomial only in the variables $u_0,u_1,u_2$. We will split this proof into the four possible ranks of the matrix $\varphi_U(0,0,0,1)\in V^\vee\otimes W^\vee$. 
    \begin{itemize}
        \item Suppose $\varphi_U(0,0,0,1)$ is rank $0$. Then $\varphi_U$ is not injective, so $\varphi$ is $G$-unstable of block form \ref{u1}\hspace{-\Tspace} by Proposition~\ref{prop:G-ss-hilbertmumford}.
        \item Otherwise, suppose $\varphi_U(0,0,0,1)$ is rank $1$. Then, up to $\SL(V)\times\SL(W)$-equivalence, we have
        \[\varphi_U=\begin{pmatrix}
            a_{11}&a_{12}&a_{13}\\
            a_{21}&a_{22}&a_{23}\\
            a_{31}&a_{32}&a_{33}
        \end{pmatrix}+\begin{pmatrix}
            u_3&0&0\\0&0&0\\0&0&0
        \end{pmatrix},\]
        where the $a_{ij}$ are linear forms in $u_0,u_1,u_2$. Computing the determinant, we have
        \[\det(\varphi_U)=(a_{22}a_{33}-a_{23}a_{32})u_3+(\textnormal{terms in $u_0,u_1,u_2$}).\]
        By assumption, the polynomial $\det(\varphi_U)$ has no $u_3$ term so $a_{22}a_{33}-a_{23}a_{32}=0$ as polynomials. We can then choose coordinates so that $a_{22},a_{23},a_{32},a_{33}$ are each scalar multiples of either $u_0$ or $u_1$. Then $\varphi$ is of block form \ref{u4}\hspace{-\Tspace} and is $G$-unstable. 
    
        \item Otherwise, suppose $\varphi_U(0,0,0,1)$ is rank $2$. Then, up to $\SL(V)\times\SL(W)$-equivalence, we have 
        \[\varphi_U=\begin{pmatrix}
            a_{11}&a_{12}&a_{13}\\
            a_{21}&a_{22}&a_{23}\\
            a_{31}&a_{32}&a_{33}
        \end{pmatrix}+\begin{pmatrix}
            u_3&0&0\\0&u_3&0\\0&0&0
        \end{pmatrix},\]
        where the $a_{ij}$ are linear forms in $u_0,u_1,u_2$. We have
        \[\det(\varphi_U)=a_{33}u_3^2+(a_{11}a_{33}+a_{22}a_{33}-a_{13}a_{31}-a_{23}a_{32})u_3+(\textnormal{terms in $u_0,u_1,u_2$}).\]
        We have $a_{33}=0$ because the polynomial $\det(\varphi_U)$ has no $u_3^2$ term. And $a_{13}a_{31}+a_{23}a_{32}=0$ as polynomials because $\det(\varphi_U)$ has no $u_3$ term. We can then choose coordinates so that $a_{13},a_{31},a_{23},a_{32}$ are each scalar multiples of either $u_0$ or $u_1$. Then $\varphi$ is of block form $(u12)$ and is $G$-unstable. 
        \item It is not possible for $\varphi_U(0,0,0,1)$ to be rank $3$: in this case $\det(\varphi_U)$ would contain a non-zero $u_3^3$ term (whose coefficient is the non-zero determinant of $\varphi_U(0,0,0,1)$). 
    \end{itemize}
\end{proof}

We now tackle the case where the surface $S_\varphi$ is non-normal, not a cone, and irreducible. 

\begin{lemma}
\label{lem:non-normal-unstable}
    Let $\varphi\in\mathbb{P}(U^\vee\otimes V^\vee\otimes W^\vee)$ be a tensor whose associated cubic surface $S_\varphi$ is non-normal, non-conical, and irreducible. Then $\varphi$ is $G$-unstable.
\end{lemma}
\begin{proof}
    By Proposition~\ref{prop:types-of-tensors}, the tensor $\varphi$ is $\SL(U)$-equivalent to a tensor where either \linebreak${\det(\varphi_U)=u_2u_0^2+u_3u_1^2}$ or $\det(\varphi_U)=u_2u_0^2+u_3u_0u_1+u_1^3$. Either way, we can choose coordinates so that $\det(\varphi_U)$ is at most linear in the variables $u_2$ and $u_3$. We let $L\subset S_\varphi\subset\mathbb{P}U$ be the line parametrised by $[0:0:u_2:u_3]$.

    We first consider the case where the line $L$ is a rank at most $1$ linear subspace of $\varphi_U$, i.e. that $\varphi_U(0,0,u_2,u_3)$ is rank $1$ for all $[u_2:u_3]$. By Proposition~\ref{prop:rank1-subspaces}, $L$ satisfies a block rank condition of type $(3,2)$ or $(2,3)$. But then $\varphi$ is $G$-unstable of block form \ref{u10}\hspace{-\Tspace} or \ref{u10T}, respectively.

    Otherwise, a general point of $L$ is rank $2$. We choose two rank $2$ points $l_1,l_2\in L$ and change coordinates so that $l_1=[0:0:1:0]$ and $l_2=[0:0:0:1]$. We let
    \begin{align*}
        V_2=\coker(\varphi_U(0,0,1,0))\subset V,&\qquad V_3=\coker(\varphi_U(0,0,0,1))\subset V,\\ W_2=\ker(\varphi_U(0,0,1,0))\subset W,&\qquad W_3=\ker(\varphi_U(0,0,0,1))\subset W
    \end{align*}
    be their left and right kernels, which are all $1$-dimensional by assumption. We can choose coordinates so that
    \begin{equation}
    \label{eqn:lem-nnus1}
        \varphi_U=\begin{pmatrix}
            a_{11}&a_{12}&a_{13}\\
            a_{21}&a_{22}&a_{23}\\
            a_{31}&a_{32}&a_{33}
        \end{pmatrix}+\begin{pmatrix}
            u_2&0&0\\0&u_2&0\\0&0&0
        \end{pmatrix},
    \end{equation}
    where the $a_{ij}$ are linear forms in $u_0,u_1,u_3$. The coefficient of $u_2^2$ in $\det(\varphi_U)$ is $a_{33}$, which must be zero by assumption, so $\varphi_{V\otimes W}(V_2\otimes W_2)=0$. Similarly, $\varphi_{V\otimes W}(V_3\otimes W_3)=0$. Now, we split the rest of the proof into the following four cases:
    \begin{itemize}
        \item Suppose $V_2=V_3$ and $W_2=W_3$. Then, using the coordinates of \eqref{eqn:lem-nnus1}, the forms $a_{13},a_{23},\linebreak a_{31},a_{32}$ also contain no $u_3$ term and $a_{33}=0$. So $\varphi$ is $G$-unstable of form \ref{u12}\hspace{-\Tspace}.
        \item Otherwise, suppose that $V_2=V_3$ but $W_2\ne W_3$. We can choose coordinates so that 
        \[\varphi_U=\begin{pmatrix}
        b_{11}&b_{12}&b_{13}\\
        b_{21}&b_{22}&b_{23}\\
        b_{31}&0&0
        \end{pmatrix}+\begin{pmatrix}
            u_2&0&0\\0&u_2&0\\0&0&0
        \end{pmatrix}+\begin{pmatrix}
            c_{11}&0&c_{13}\\c_{21}&0&c_{23}\\0&0&0
        \end{pmatrix},\]
        where the $b_{ij}$ are linear forms in $u_0,u_1$ and the $c_{ij}$ are scalar multiples of $u_3$. But then the surface $S_\varphi$ contains the hyperplane cut out by $b_{31}=0$ and is reducible, a contradiction (and $\varphi$ would be $G$-unstable of block form \ref{u2}\hspace{-\Tspace}).
        \item Similarly, if $W_2=W_3$ but $V_2\ne V_3$ then the surface $S_\varphi$ is reducible.
        \item Finally, suppose that both $V_2\ne V_3$ and $W_2\ne W_3$. We can choose coordinates so that
        $$\varphi_U=\begin{pmatrix}
            a_{11}&a_{12}&a_{13}\\a_{21}&a_{22}&a_{23}\\a_{31}&a_{32}&a_{33}
        \end{pmatrix}+\begin{pmatrix}
            \beta_{11}&\beta_{12}&0\\\beta_{21}&\beta_{22}&0\\0&0&0
        \end{pmatrix}u_2 + \begin{pmatrix}
            \gamma_{11}&0&\gamma_{13}\\0&0&0\\\gamma_{31}&0&\gamma_{33}
        \end{pmatrix}u_3,$$
        where the $a_{ij}$ are linear forms in $u_0,u_1$ and the $\beta_{ij}$ and $\gamma_{ij}$ are scalars. We have
        \begin{align*}
            \det(\varphi_U)=&\;(\beta_{11}\beta_{22}-\beta_{12}\beta_{21})(\gamma_{33}u_3+a_{33})u_2^2+(\gamma_{11}\gamma_{33}-\gamma_{13}\gamma_{31})(\beta_{22}u_2+a_{33})u_3^2\\
            &\;+(\beta_{12}\gamma_{31}a_{23}+\beta_{21}\gamma_{13}a_{32})u_2u_3+(\textnormal{terms at most linear in $u_2$ and $u_3$})
        \end{align*}
        By assumption, the determinant $\det(\varphi_U)$ is at most linear in $u_2$ and $u_3$. Additionally, we assumed the matrices $\varphi_U(0,0,1,0)$ and $\varphi_U(0,0,0,1)$ were rank $2$, so $\beta_{11}\beta_{22}-\beta_{12}\beta_{21}\ne0$ and $\gamma_{11}\gamma_{33}-\gamma_{13}\gamma_{31}\ne0$. Then $a_{33}=0$ and $\gamma_{33}=0$ and $\beta_{22}=0$. 

        Because $\det(\varphi_U)$ contains no $u_2u_3$ term then we also have $\beta_{12}\gamma_{31}a_{23}+\beta_{21}\gamma_{13}a_{32}=0$ as polynomials. Because the matrices $\varphi_U(0,0,1,0)$ and $\varphi_U(0,0,0,1)$ are each rank $2$ then $\beta_{12}\beta_{21}\ne0$ and $\gamma_{13}\gamma_{31}\ne0$. So $a_{23}$ and $a_{32}$ are scalar multiples of eachother. Choosing coordinates so that $a_{23}=\lambda u_0$ and $a_{32}=\lambda'u_0$ (where $\lambda$ and $\lambda'$ may be zero), we have
        $$\varphi_U=\begin{pmatrix}
            a_{11}&a_{12}&a_{13}\\a_{21}&0&\lambda u_0\\a_{31}&\lambda' u_0&0
        \end{pmatrix}+\begin{pmatrix}
            \beta_{11}&\beta_{12}&0\\\beta_{21}&0&0\\0&0&0
        \end{pmatrix}u_2 + \begin{pmatrix}
            \gamma_{11}&0&\gamma_{13}\\0&0&0\\\gamma_{31}&0&0
        \end{pmatrix}u_3$$
        and $S_\varphi$ is reducible, a contradiction (and $\varphi$ would also be $G$-unstable of block form \ref{u3}\hspace{-\Tspace}). 
    \end{itemize}
\end{proof}

Finally, we tackle the case where the determinant cubic vanishes everywhere, i.e. when $S_\varphi=\mathbb{P}U$.

\begin{lemma}
\label{lem:zero-unstable}
    Let $\varphi\in\mathbb{P}(U^\vee\otimes V^\vee\otimes W^\vee)$ be a tensor such that $S_\varphi=\mathbb{P}U$. Then $\varphi$ is $G$-unstable.
\end{lemma}
\begin{proof}
    If $S_\varphi=\mathbb{P}U$ then $\varphi_U(U)\subset V^\vee\otimes W^\vee$ is a rank at most $2$ linear subspace. Either the map $\varphi_U$ is not injective, in which case $\varphi$ is $G$-unstable is of block form \ref{u1}\hspace{-\Tspace}. Or $\varphi_U(U)$ is a $4$-dimensional rank at most $2$ linear subspace which, by Proposition~\ref{prop:rank2-subspaces}, satisfies a block rank condition of type $(1,3)$ or $(3,1)$ or $(2,2)$. Then $\varphi$ is of block form \ref{u2}\hspace{-\Tspace} or \ref{u2T} or \ref{u3}\hspace{-\Tspace}, respectively, and is $G$-unstable.
\end{proof}

We can now complete the proof of Theorem~\ref{thm:G-semistability}: that a tensor $\varphi\in\mathbb{P}(U^\vee\otimes V^\vee\otimes W^\vee)$ is \mbox{$G$-semistable} if and only if the surface $S_\varphi$ is $\SL(U)$-semistable or $S_\varphi$ contains no rank at most $1$ points. Equivalently, a tensor $\varphi\in\mathbb{P}(U^\vee\otimes V^\vee\otimes W^\vee)$ is $G$-semistable if and only if either the associated cubic surface $S_\varphi$ has at worst $4A_1$ or $3A_2$ singularities or $\varphi\in A_3(d)\cup A_1A_3(f)\cup\Sigma_3$. 
\begin{proof}[Proof of Theorem~\ref{thm:G-semistability}]
    By Proposition~\ref{prop:types-of-tensors}, each closed point $\varphi\in\mathbb{P}(U^\vee\otimes V^\vee\otimes W^\vee)$ satisfies (at least) one of five conditions. These five cases are covered in order in Lemmas~\ref{lem:canonical-semistability}, \ref{lem:reducible-semistable}, \ref{lem:conical-unstable}, \ref{lem:non-normal-unstable} and \ref{lem:zero-unstable}. Combining these five lemmas completes the proof. 
\end{proof}

The criteria of Theorem~\ref{thm:G-semistability} allow us to describe the $G$-invariant subring of $\mathbb{C}[U\otimes V\otimes W]_\bullet$ up to finite extension. We will use this description to understand some of the topology of the quotient $\mathbb{P}(U^\vee\otimes V^\vee\otimes W^\vee)//G$ and determine the $G$-stable and $G$-polystable loci in Subsections~\ref{subsec:G-stability} and \ref{subsec:G-polystability}. The $\SL(U)$-invariant ring $\mathbb{C}[\Sym^3(U)]_\bullet^{\SL(U)}$ of cubic surfaces was known to Clebsch and Salmon in 1860. 

\begin{proposition}[{\cite{salmon-invariants,clebsch-invariants1,clebsch-invariants2}}]
    The subring of $\SL(U)$-invariants of $\mathbb{C}[\Sym^3(U)]_\bullet$ is generated by six invariants $I_8,I_{16},I_{24},I_{32}.I_{40},I_{100}$ in degrees $8,16,24,32,40,100$. The first five invariants are algebraically independent, and $I^2_{100}$ is a polynomial in the first five invariants. 
\end{proposition}

As a result, the GIT quotient $\vert\mathcal{O}_{\mathbb{P}U}(3)\vert//\SL(U)$ is isomorphic to $\mathbb{P}(1,2,3,4,5)$. We also note that the discriminant $\Delta\in\mathbb{C}[U\otimes V\otimes W]_\bullet^G$ is a prime divisor of degree $32$, where $\Delta([S])=0$ if and only if $S\subset\mathbb{P}U$ is a singular surface. We have
$$\Delta=I_{32}+(\textnormal{a polynomial in $I_8,I_{16},I_{24}$}).$$
Recall Proposition~\ref{prop:R24}, that there is a degree $24$ polynomial $R_{24}\in\mathbb{C}[U\otimes V\otimes W]_\bullet$ such that $R_{24}(\varphi)=0$ if and only if the associated surface $S_\varphi$ contains a rank at most $1$ point. 

\begin{definition}
\label{defn:almost-inv-ring}
    For $d=8,16,24,32,40,100$ we let $J_{3d}:=I_d\circ\det$. We define the ring
    \begin{equation}
    \label{eqn:Abullet}
        A_\bullet:=\mathbb{C}[R_{24},J_{24},J_{48},J_{72},J_{96},J_{120},J_{300}]_\bullet\subset\mathbb{C}[U\otimes V\otimes W]^G_\bullet.
    \end{equation}
    We let
    $$\eta:\Proj(A_\bullet)\dashrightarrow\mathbb{P}(U^\vee\otimes V^\vee\otimes W^\vee)//G$$
    be the candidate rational map of projective spectra induced by the inclusion \eqref{eqn:Abullet} of graded rings. 

    The inclusion of rings $\mathbb{C}[J_{24},J_{48},J_{72},J_{96},J_{120},J_{300}]_\bullet\subset A_\bullet$, together with the isomorphism 
    $${\det}^\ast:\mathbb{C}[\Sym^3(U)]_\bullet^{\SL(U)}=\mathbb{C}[I_{8},I_{16},I_{24},I_{32},I_{40},I_{100}]_\bullet\xlongrightarrow{\sim}\mathbb{C}[J_{24},J_{48},J_{72},J_{96},J_{120},J_{300}]_\bullet,$$ induces a rational map $\xi:\Proj(A_\bullet)\dashrightarrow\vert\mathcal{O}_{\mathbb{P}U}(3)\vert//\SL(U)$ and the following diagram commutes:
    \[\begin{tikzcd}
    	{\mathbb{P}(U^\vee\otimes V^\vee\otimes W^\vee)//G} & {\Proj(A_\bullet)} \\
    	& {\vert\mathcal{O}_{\mathbb{P}U}(3)\vert//\SL(U).}
    	\arrow["\eta", dashed, from=1-1, to=1-2]
    	\arrow["{\det_{UVW}}"', dashed, from=1-1, to=2-2]
    	\arrow["\xi", dashed, from=1-2, to=2-2]
    \end{tikzcd}\]
\end{definition}

We will view $\Proj(A_\bullet)$ as a closed subscheme of the weighted projective space $\mathbb{P}(1,1,2,3,4,5)$ using the coordinates ${[R_{24}:J_{24}:J_{48}:J_{72}:J_{96}:J_{120}]}$. The subscheme $\Proj(A_\bullet)$ is cut out by the relation between $R_{24}$ and the $J_{3d}$s. In these coordinates, the map $\xi$ is (the restriction of) the forgetful map $\mathbb{P}(1,1,2,3,4,5)\dashrightarrow\mathbb{P}(1,2,3,4,5)$ where we forget the first coordinate $R_{24}$.

By Theorem~\ref{thm:G-semistability}, the ideal generated by the subring $A_\bullet\subset\mathbb{C}[U\otimes V\otimes W]_\bullet$ set-theoretically cuts out the $G$-unstable locus. As a result, $A_\bullet$ generates the ring of $G$-invariant functions up to radical and $\mathbb{C}[U\otimes V\otimes W]_\bullet^G$ is a finite, integral extension of its subring $A_\bullet$. In particular, the map $\eta$ is a finite morphism. 

\subsection{$G$-stability}
\label{subsec:G-stability}

The goal of Subsection~\ref{subsec:G-stability} is to prove a simple criterion for $G$-stability: that a tensor $\varphi\in\mathbb{P}(U^\vee\otimes V^\vee\otimes W^\vee)$ is $G$-stable if and only if the associated cubic surface $S_\varphi$ is \mbox{$\SL(U)$-stable}. The proof follows a similar structure to that of Theorem~\ref{thm:G-semistability}. We start by showing this condition is sufficient and then use the Hilbert-Mumford criterion to prove it is necessary. 

\begin{proposition}
\label{prop:sufficient-G-stability}
    Let $\varphi\in\mathbb{P}(U^\vee\otimes V^\vee\otimes W^\vee)$ be a tensor whose associated cubic surface $S_\varphi$ is $\SL(U)$-stable. Then $\varphi$ is $G$-stable. 
\end{proposition}
\begin{proof}
    Recall that a cubic surface $S\subset\mathbb{P}U$ is $\SL(U)$-stable if and only if $S$ has at worst $4A_1$ singularities. By Proposition~\ref{prop:tensor-dims}, if $\varphi\in\mathbb{P}(U^\vee\otimes V^\vee\otimes W^\vee)$ is a tensor whose associated cubic surface $S_\varphi$ has at worst $4A_1$ singularities then the $G$-stabiliser $\Stab_G(\varphi)$ is finite. So it suffices to prove that the orbit $G\cdot\varphi\subset\mathbb{P}(U^\vee\otimes V^\vee\otimes W^\vee)^{G-ss}$ is closed inside the $G$-semistable locus whenever the associated surface $S_\varphi$ is $\SL(U)$-stable. We recall the commutative diagram of GIT quotients

    \[\begin{tikzcd}
        {\mathbb{P}(U^\vee\otimes V^\vee\otimes W^\vee)} & {\vert\mathcal{O}_{\mathbb{P}U}(3)\vert} \\
        {\mathbb{P}(U^\vee\otimes V^\vee\otimes W^\vee)//G} & {\vert\mathcal{O}_{\mathbb{P}U}(3)\vert//SL(U).}
        \arrow["\det", dashed, from=1-1, to=1-2]
        \arrow["{\pi_{UVW}}"', dashed, from=1-1, to=2-1]
        \arrow["{p_U}", dashed, from=1-2, to=2-2]
        \arrow["{\det_{UVW}}"', dashed, from=2-1, to=2-2]
    \end{tikzcd}\]

    Let $\varphi\in\mathbb{P}(U^\vee\otimes V^\vee\otimes W^\vee)$ be a tensor whose associated cubic surface $S_\varphi$ is $\SL(U)$-stable and let $\varphi'\in\overline{G\cdot\varphi}^{G-ss}$ be any $G$-semistable tensor inside its orbit closure. We have ${\pi_{UVW}(\varphi)=\pi_{UVW}(\varphi')}$. The map $\det_{UVW}$ is defined at $\pi_{UVW}(\varphi)$ because $S_\varphi$ is $\SL(U)$-semistable, so we have \linebreak${(\det_{UVW}\circ\pi_{UVW})(\varphi)=(\det_{UVW}\circ\pi_{UVW})(\varphi')}$. Then $p_U(S_\varphi)=p_U(S_{\varphi'})$ and, because $S_\varphi$ is $\SL(U)$-stable, the surfaces $S_\varphi$ and $S_{\varphi'}$ are $\SL(U)$-equivalent (because $\vert\mathcal{O}_{\mathbb{P}U}(3)\vert^{\SL(U)-s}//\SL(U)$ is an orbit space).

    The $G$-orbits $G\cdot\varphi$ and $G\cdot\varphi'$ are then irreducible codimension $4$ subvarieties of ${\mathbb{P}(U^\vee\otimes V^\vee\otimes W^\vee)}$. Since $G\cdot\varphi'\subset\overline{G\cdot\varphi}$, we must have $G\cdot\varphi'=G\cdot\varphi$ and $\varphi'\in G\cdot\varphi$. So the $G$-orbit $G\cdot\varphi$ is closed inside the $G$-semistable locus and $\varphi$ is $G$-stable whenever its associated cubic surface $S_\varphi$ is $\SL(U)$-stable.
\end{proof}

Conversely, by the Hilbert-Mumford criterion we have the necessary condition of

\begin{proposition}
\label{prop:G-s-hilbertmumford}
    A tensor $\varphi\in\mathbb{P}(U^\vee\otimes V^\vee\otimes W^\vee)$ is strictly $G$-semistable if both $\varphi$ is $G$-semistable and $\varphi$ is of one of the following block forms:
    \begin{enumerate}[label=$(ns1)\hspace{\Tspace}$]
        \item \hfill$\left[\begin{array}{ccc|ccc|ccc|ccc}
             \ast&\ast&\ast & \ast&\ast&\ast & \ast&\ast&\ast & \ast&0&0 \\
             \ast&\ast&\ast & \ast&\ast&\ast & \ast&\ast&\ast & \ast&0&0 \\
             \ast&\ast&\ast & \ast&0&0 & 0&0&0 & 0&0&0
        \end{array}\right]$\hfill\bigskip \label{ns1}
    \end{enumerate}
    \begin{enumerate}[label=$(ns1)^T$]
        \item \hfill$\left[\begin{array}{ccc|ccc|ccc|ccc}
             \ast&\ast&\ast & \ast&\ast&\ast & \ast&\ast&0 & \ast&\ast&0 \\
             \ast&\ast&\ast & \ast&\ast&0 & \ast&\ast&0 & 0&0&0 \\
             \ast&\ast&\ast & \ast&\ast&0 & \ast&\ast&0 & 0&0&0
        \end{array}\right]$\hfill\bigskip \label{ns1T}
    \end{enumerate}
    \begin{enumerate}[label=$(ns2)\hspace{\Tspace}$]
        \item \hfill$\left[\begin{array}{ccc|ccc|ccc|ccc}
             \ast&\ast&\ast & \ast&\ast&\ast & \ast&\ast&\ast & \ast&\ast&0  \\
             \ast&\ast&\ast & \ast&\ast&0 & \ast&\ast&0 & \ast&0&0 \\
             \ast&\ast&0 & \ast&0&0 & \ast&0&0 & 0&0&0
        \end{array}\right]$\hfill\bigskip \label{ns2}
    \end{enumerate}
\end{proposition}
\begin{proof}
    We apply the Hilbert-Mumford criterion for stability. A point is not $G$-stable if and only if there exists a one-parameter subgroup $\lambda$ that acts with \emph{non-negative} weights on a tensor ${\varphi\in\mathbb{P}(U^\vee\otimes V^\vee\otimes W^\vee)}$, i.e. that $\varphi_{ijk}=0$ whenever $a_i+b_j+c_k<0$ (where we write $\operatorname{wt}(\lambda)=((a_0,a_1,a_2,a_3),(b_0,b_1,b_2),(c_0,c_1,c_2))$ for the weight vector). This is almost identical to the Hilbert-Mumford criterion for semistability, the only difference is that we have replaced ``positive'' with ``non-negative'' and ``$\leq$'' with ``$<$''. We refer the reader to the proof of Proposition~\ref{prop:G-ss-hilbertmumford} for all details, the proof of these two propositions are identical after making those replacements. 

    To complete the proof, we need to exhibit a weight vector associated to each of the above block forms. The pairs of the weight vector $w$ and its associated not $G$-stable block form $B$ are the following:
    \[\begin{array}{c @{\hspace{0.5cm}}|@{\hspace{0.5cm+\Tspace}} l}
        \hspace{\Tspace}B\hspace{\Tspace} & \multicolumn{1}{c}{w} \\
        \hline
        \hspace{\Tspace}(ns1)\hspace{\Tspace} & ((3, 0, 0, -3), (2, -1, -1), (1, 1, -2))\\
        \hspace{\Tspace}(ns1)^T & ((3, 0, 0, -3), (1, 1, -2), (2, -1, -1))\\
        \hspace{\Tspace}(ns2)\hspace{\Tspace} & ((1, 0, 0, -1), (1, 0, -1), (1, 0, -1))\\
    \end{array}\]
\end{proof}

Every $G$-semistable tensor $\varphi\in\mathbb{P}(U^\vee\otimes V^\vee\otimes W^\vee)^{G-ss}$ that does not satisfy the conditions of Proposition~\ref{prop:sufficient-G-stability} is either of the form $F(n)$ for some Dynkin subdiagram $F\subset 3A_2$ that contains $A_2$, or $\varphi\in A_3(d)\cup A_1A_3(f)\cup\Sigma_3$. Either way, $\varphi$ is contained in the closure $\overline{A_2(a)}\cup\overline{A_2(b)}$. More precisely, we have

\begin{proposition}
\label{prop:sigma3-closures}
    Fix a stratum $F(n)\subset\mathbb{P}(U^\vee\otimes V^\vee\otimes W^\vee)$. Then $F(n)\not\subset\Sigma_3$. We have $\Sigma_3\subset\overline{F(n)}$ if and only if the associated cubic surface $S_\varphi$ to every tensor $\varphi\in F(n)$ contains no rank at most $1$ points.
\end{proposition}
\begin{proof}
    By Proposition~\ref{prop:tensor-dims} we have $\dim F(n)\geq 6$ and $\dim\Sigma_3=7$, so $\overline{F(n)}\not\subset\Sigma_3$. To prove the rest of this proposition we use a similar method to the proof of Proposition~\ref{prop:closure-rules}. 

    Fix a stratum $F(n)\subset\mathbb{P}(U^\vee\otimes V^\vee\otimes W^\vee)\setminus\mathcal{R}$ such that the associated cubic surface $S_\varphi$ to every tensor $\varphi\in F(n)$ contains no rank at most $1$ points. As illustrated in Example~\ref{ex:A1A2(a)-subschemes}, the subvariety $\chi(F(n))\subset\Hilb_6(\mathbb{P}V)$ parametrises schemes $[X]\in\Hilb_6(\mathbb{P}V)$ satisfying some number of closed conditions (a) and some number of conditions (b), including the condition that $X$ is not contained in a conic. We let $\mathcal{A}\subset\Hilb_6(\mathbb{P}V)$ be the subscheme cut out by the closed conditions (a) and the open conditions (b), with the modification that $[X]\in\mathcal{A}$ is allowed to be contained in at most one conic. The additional points in $\mathcal{A}\setminus\chi(F(n))$ parametrise a certain subset of the set of canonical subschemes $X\subset\Hilb_6(\mathbb{P}V)$ such that
    \begin{enumerate}[label=(\roman*)]
        \item $X$ is contained in a unique conic and
        \item the associated Dynkin diagram $F_X$ to $X$ is the union of $F$ and a point.
    \end{enumerate}
    Note that we have $\mathcal{A}\subset\overline{\chi(F(n))}$ by construction. 
    
    The contraction $\gamma:\Hilb_6(\mathbb{P}V)\dashrightarrow\mathbb{P}(U^\vee\otimes V^\vee\otimes W^\vee)//(\SL(U)\times\SL(W))$ is defined on $\mathcal{A}(F(n))$, so
    \begin{equation}
    \label{eqn:sigma-closure}
        \gamma\left(\mathcal{A}\right)\subset\overline{\gamma\left(\chi(F(n))\right)}=\overline{\pi_{UW}\left(F(n)\right)}.
    \end{equation}
    The subschemes $[X]\in\mathcal{A}\setminus\chi(F(n))$ contain no length $3$ collinear subscheme because no tensor $\varphi\in F(n)$ contains a rank at most $1$ point. This means that each subscheme $X$ must be contained in a smooth conic, so $\gamma$ contracts $\mathcal{A}\setminus\chi(F(n))$ onto $\pi_{UW}(\Sigma_3)$ (see Proposition~\ref{prop:gamma-facts}). Then ${\gamma(\mathcal{A})=\pi_{UW}(F(n))\cup\pi_{UW}(\Sigma_3)}$. Because $F(n)\cup\Sigma_3$ is contained in the $\SL(U)\times\SL(W)$-stable locus then we can lift \eqref{eqn:sigma-closure} to $\mathbb{P}(U^\vee\otimes V^\vee\otimes W^\vee)$ by applying $\pi_{UW}^{-1}$ to both sides, which gives $\Sigma_3\cup F(n)\subset\overline{F(n)}$.

    Conversely, if there is a tensor $\varphi\in F(n)$ whose associated cubic surface $S_\varphi$ contains no rank at most $1$ points (either all points of $F(n)$ satisfy this, or none of them do) then the function $R_{24}$ vanishes on $F(n)$. So the closure $\overline{F(n)}$ is contained in the closed subscheme $\mathcal{R}\subset\mathbb{P}(U^\vee\otimes V^\vee\otimes W^\vee)$ of tensors $\varphi$ whose associated surface $S_\varphi$ contains a rank at most $1$ point. We have $\Sigma_3\cap\mathcal{R}=\emptyset$, so $\Sigma_3\not\subset\overline{F(n)}$.
\end{proof}

\begin{proposition}
\label{prop:strictlysemistable-canonical-closures}
    If $\varphi\in\mathbb{P}(U^\vee\otimes V^\vee\otimes W^\vee)^{G-ss}$ is a $G$-semistable tensor whose associated cubic surface $S_\varphi$ is not $\SL(U)$-stable then $\varphi\in\overline{A_2(a)}\cup\overline{A_2(b)}$.
\end{proposition}
\begin{proof}
    This proof is almost identical to that of Proposition~\ref{prop:unstable-canonical-closures}. The proof amounts to carefully checking all of the closure relations between subschemes of the form $\chi(F(n))\subset\Hilb_6(\mathbb{P}V)$ using the lists of associated subschemes in \cite{ng}, and then passing to $F(n)\subset\mathbb{P}(U^\vee\otimes V^\vee\otimes W^\vee)$ by using Proposition~\ref{prop:closure-rules}. And we have $\Sigma_3\subset\overline{A_2(b)}$ by Proposition~\ref{prop:sigma3-closures}.
\end{proof}

We can now prove the following criterion for $G$-stability:
\begin{theorem}
\label{thm:G-stability}
    A tensor $\varphi\in\mathbb{P}(U^\vee\otimes V^\vee\otimes W^\vee)$ is $G$-stable if and only if the associated cubic surface $S_\varphi$ is $\SL(U)$-stable. 
\end{theorem}
\begin{proof}
    Fix a $G$-semistable tensor $\varphi\in\mathbb{P}(U^\vee\otimes V^\vee\otimes W^\vee)^{G-ss}$. By Proposition~\ref{prop:strictlysemistable-canonical-closures}, either \linebreak$\varphi\in\overline{A_2(a)}\cup\overline{A_2(b)}$ or the surface $S_\varphi$ is $\SL(U)$-stable. In the latter case, $\varphi$ is $G$-stable by Proposition~\ref{prop:sufficient-G-stability}. So it suffices to prove that all tensors of types $A_2(a)$ and $A_2(b)$ are strictly $G$-semistable. We again refer the reader to \cite{ng} for the following descriptions of the points of $A_2(a)$ and $A_2(b)$.
    \begin{itemize}
        \item Every tensor $\varphi\in A_2(a)$ is of the form
        $$\varphi=\left[\begin{array}{ccc|ccc|ccc|ccc}
            0&0&1 & 0&0&0 & 0&-\alpha&0 & 1&1&0\\
            0&0&0 & 0&1&0 & -1&0&0 & 1&0&0 \\
            0&0&0 & 0&0&\lambda & 1&0&0 & 0&0&1
        \end{array}\right]$$
        for some $\alpha\in\mathbb{C}\setminus\{0,1\}$ and $\lambda\in\mathbb{C}\setminus\{-1,0\}$. Viewing the tensor as $\varphi=[\varphi_0\vert \varphi_1\vert \varphi_2\vert \varphi_3]$, we change coordinates to $\varphi\sim[\varphi_0\vert \varphi_1-\lambda \varphi_3\vert \varphi_3\vert \varphi_2]$ and see that $\varphi$ is of block form \ref{ns1}\hspace{-\Tspace}. Then $\varphi$ is not $G$-stable by Proposition~\ref{prop:G-s-hilbertmumford}. A slightly more complicated change of coordinates shows that $\varphi$ is also of block form \ref{ns1T}.
        \item Every tensor $\varphi\in A_2(b)$ is of the form
        $$\varphi=\left[\begin{array}{ccc|ccc|ccc|ccc}
            0&1&0 & 1&0&0 & 0&0&0 & 0&0&0\\
            0&0&-\beta & 0&0&1-\beta & 1&-\lambda&0 & 0&0&1\\
            0&0&0 & 0&0&1 & 0&1-\lambda&0 & 1&0&0
        \end{array}\right]$$
        for some $\lambda\in\mathbb{C}\setminus\{0,1\}$ and $\beta\in\mathbb{C}\setminus\{0,(1-\lambda)^{-1}\}$. Viewing $\varphi=[\varphi_0\vert\varphi_1\vert\varphi_2\vert\varphi_3]$, we change coordinates to $\varphi\sim[\varphi_1\vert\varphi_0\vert\varphi_3\vert\varphi_3]$. We then swap the first two columns and permute the rows by $(w_0,w_1,w_2)\mapsto (w_2,w_0,w_1)$ (recalling the rows are indexed by the basis of $W^\vee$). We then have
        \[\varphi\sim\left[\begin{array}{ccc|ccc|ccc|ccc}
             0&0&1-\beta & 0&0&-\beta & 0&0&1 & -\lambda&1&0 \\
             0&0&1 & 0&0&0 & 0&1&0 & 1-\lambda&0&0 \\
             0&1&0 & 1&0&0 & 0&0&0 & 0&0&0
        \end{array}\right]\]
        and we see that $\varphi$ is of block form \ref{ns2}\hspace{-\Tspace} and is not $G$-stable.
    \end{itemize}
\end{proof}

\subsection{Geometry of the GIT quotient}
\label{subsec:G-polystability}

We conclude with a description of the $G$-polystable locus of $\mathbb{P}(U^\vee\otimes V^\vee\otimes W^\vee)$ and some of the geometry of the GIT quotient $\mathbb{P}(U^\vee\otimes V^\vee\otimes W^\vee)//G$. Inside the $\mathbb{P}(U^\vee\otimes V^\vee\otimes W^\vee)^{G-ss}$ we have

\begin{proposition}
\label{prop:full-semistable-closures}
    The $G$-semistable locus $\mathbb{P}(U^\vee\otimes V^\vee\otimes W^\vee)^{G-ss}$ decomposes as the disjoint union
    $$\mathbb{P}(U^\vee\otimes V^\vee\otimes W^\vee)^{G-ss}=\mathbb{P}(U^\vee\otimes V^\vee\otimes W^\vee)^{G-s}\cup\overline{A_2(a)}^{G-ss}\cup\overline{A_2(b)}^{G-ss}.$$
    The $G$-stable locus $\mathbb{P}(U^\vee\otimes V^\vee\otimes W^\vee)^{G-s}$ admits the following stratification, where we draw an arrow $Z\to Z'$ between strata if and only if $Z'\subset\overline{Z}$ (we omit arrows made redundant by transitivity). We draw a box $\boxed{Z}$ to indicate that $Z$ is closed inside $\mathbb{P}(U^\vee\otimes V^\vee\otimes W^\vee)^{G-s}$.
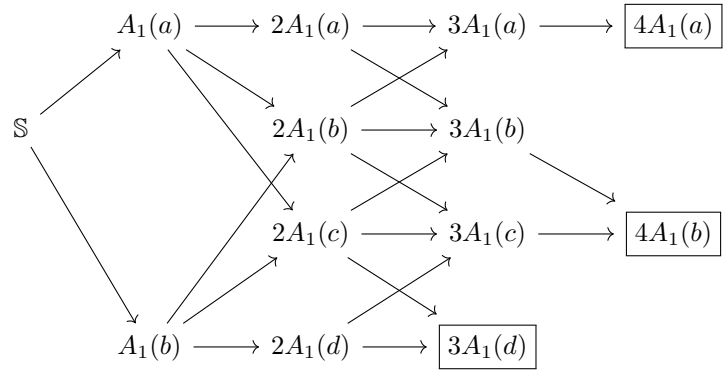
    \[\begin{tikzcd}
    	& {A_1(a)} & {2A_1(a)} & {3A_1(a)} & {\boxed{4A_1(a)}} \\
    	{\mathbb{S}} && {2A_1(b)} & {3A_1(b)} \\
    	&& {2A_1(c)} & {3A_1(c)} & {\boxed{4A_1(b)}} \\
    	& {A_1(b)} & {2A_1(d)} & {\boxed{3A_1(d)}}
    	\arrow[from=1-2, to=1-3]
    	\arrow[from=1-2, to=2-3]
    	\arrow[from=1-2, to=3-3]
    	\arrow[from=1-3, to=1-4]
    	\arrow[from=1-3, to=2-4]
    	\arrow[from=1-4, to=1-5]
    	\arrow[from=2-1, to=1-2]
    	\arrow[from=2-1, to=4-2]
    	\arrow[from=2-3, to=1-4]
    	\arrow[from=2-3, to=2-4]
    	\arrow[from=2-3, to=3-4]
    	\arrow[from=2-4, to=3-5]
    	\arrow[from=3-3, to=2-4]
    	\arrow[from=3-3, to=3-4]
    	\arrow[from=3-3, to=4-4]
    	\arrow[from=3-4, to=3-5]
    	\arrow[from=4-2, to=2-3]
    	\arrow[from=4-2, to=3-3]
    	\arrow[from=4-2, to=4-3]
    	\arrow[from=4-3, to=3-4]
    	\arrow[from=4-3, to=4-4]
    \end{tikzcd}\]
    Similarly, the codimension $2$ subvariety $\overline{A_2(a)}^{G-ss}$ admits the following stratification, where a box $\boxed{Z}$ indicates that $Z$ is closed inside $\overline{A_2(a)}^{G-ss}$:
    
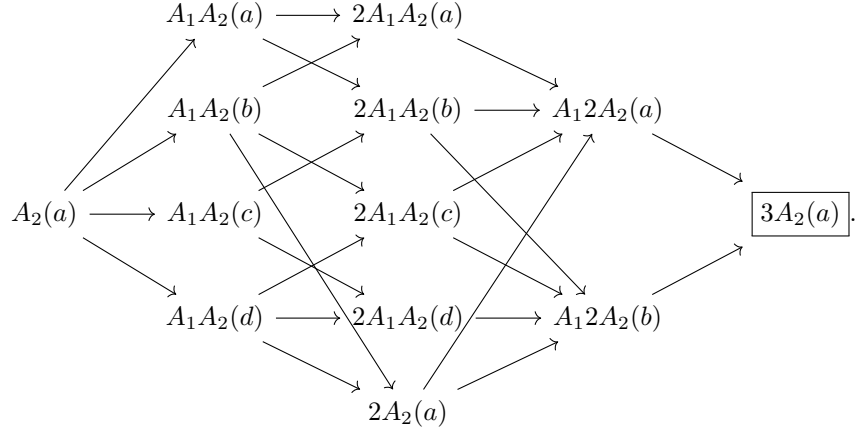
\[\begin{tikzcd}
        & {A_1A_2(a)} & {2A_1A_2(a)} \\
        & {A_1A_2(b)} & {2A_1A_2(b)} & {A_12A_2(a)} \\
        {A_2(a)} & {A_1A_2(c)} & {2A_1A_2(c)} && {\boxed{3A_2(a)}.} \\
        & {A_1A_2(d)} & {2A_1A_2(d)} & {A_12A_2(b)} \\
        && {2A_2(a)}
        \arrow[from=1-2, to=1-3]
        \arrow[from=1-2, to=2-3]
        \arrow[from=1-3, to=2-4]
        \arrow[from=2-2, to=1-3]
        \arrow[from=2-2, to=3-3]
        \arrow[from=2-2, to=5-3]
        \arrow[from=2-3, to=2-4]
        \arrow[from=2-3, to=4-4]
        \arrow[from=2-4, to=3-5]
        \arrow[from=3-1, to=1-2]
        \arrow[from=3-1, to=2-2]
        \arrow[from=3-1, to=3-2]
        \arrow[from=3-1, to=4-2]
        \arrow[from=3-2, to=2-3]
        \arrow[from=3-2, to=4-3]
        \arrow[from=3-3, to=2-4]
        \arrow[from=3-3, to=4-4]
        \arrow[from=4-2, to=3-3]
        \arrow[from=4-2, to=4-3]
        \arrow[from=4-2, to=5-3]
        \arrow[from=4-3, to=4-4]
        \arrow[from=4-4, to=3-5]
        \arrow[from=5-3, to=2-4]
        \arrow[from=5-3, to=4-4]
    \end{tikzcd}\]
    Finally, the codimension $2$ subvariety $\overline{A_2(b)}^{G-ss}$ admits the following stratification, where a box \linebreak $\boxed{Z}$ indicates that $Z$ is closed inside $\overline{A_2(b)}^{G-ss}$:
    \[\begin{tikzcd}
    	& {A_1A_2(e)} & {2A_1A_2(f)} & \\
    	{A_2(b)} & {A_1A_2(f)} & {2A_1A_2(e)} & {\boxed{A_12A_2(c)}} \\
    	& {A_1A_2(g)} & {2A_2(b)} \\
    	& {A_3(d)} & {A_1A_3(f)} & {\boxed{\Sigma_3}.}
    	\arrow[from=1-2, to=1-3]
    	\arrow[from=1-3, to=2-4]
    	\arrow[from=2-1, to=1-2]
    	\arrow[from=2-1, to=2-2]
    	\arrow[from=2-1, to=3-2]
    	\arrow[from=2-1, to=4-2]
    	\arrow[from=2-2, to=2-3]
    	\arrow[from=2-3, to=2-4]
    	\arrow[from=3-2, to=3-3]
    	\arrow[from=3-2, to=4-3]
    	\arrow[from=3-3, to=2-4]
    	\arrow[from=3-3, to=4-4]
    	\arrow[from=4-2, to=4-3]
    	\arrow[from=4-3, to=4-4]
    \end{tikzcd}\]
\end{proposition}
\begin{proof}
    We again apply Propositions~\ref{prop:closure-rules} and \ref{prop:sigma3-closures}, using the list of associated length $6$ subschemes in \cite{ng}. 
\end{proof}

If $\varphi$ is a tensor of type $F(n)$ where $\vert F\vert\geq 4$ and $F\ne2A_2$ then $F(n)=G\cdot\varphi$ is a $G$-orbit. As a result, Proposition~\ref{prop:full-semistable-closures} gives a partial answer to the question of which strictly $G$-semistable tensors $\varphi'\in\mathbb{P}(U^\vee\otimes V^\vee\otimes W^\vee)^{G-sss}$ are contained in the orbit closure $\overline{G\cdot\varphi}$ of another strictly $G$-semistable tensor $\varphi\in\mathbb{P}(U^\vee\otimes V^\vee\otimes W^\vee)^{G-sss}$. To properly answer this question, we first recall some standard facts about $G$-polystability. These can be found in, for example, Sections~4 and 5 of Victoria Hoskins' notes \cite{hoskins-notes}. 

\begin{definition}
    Let $T$ be a vector space and let $H$ be a reductive group acting linearly on $\mathbb{P}T$. A $H$-semistable point $t\in\mathbb{P}T$ is said to be $H$-\emph{polystable} if and only if the orbit $H\cdot t$ is closed inside the $H$-semistable locus $\mathbb{P}T^{H-ss}$. If $t$ is also strictly $G$-semistable then we say $t$ is \emph{strictly} $H$-polystable. 

    We let $\mathbb{P}T^{H-sps}\subset\mathbb{P}T$ be the set of strictly $H$-polystable points.
\end{definition}

Note that all $H$-stable points are also $H$-polystable. The subvariety $P$ is closed inside $\mathbb{P}T^{H-ss}$. For every strictly $H$-semistable point $t\in\mathbb{P}T^{H-sss}$ there is a strictly $H$-polystable point $s\in P$, unique up to $H$-equivalence, such that $s\in\overline{H\cdot t}$. Then, under the GIT quotient $p:\mathbb{P}T\dashrightarrow\mathbb{P}T//H$, we have $p(s)=p(t)$. Conversely, for any point $\overline{s}\in\mathbb{P}T//H$ the preimage $p^{-1}(\overline{s})\subset\mathbb{P}T$ contains a unique closed $H$-orbit $H\cdot s$ and all other points $t\in p^{-1}(\overline{s})$ satisfy $s\in\overline{H\cdot t}$. As a result, the closed embedding
$$\iota:\mathbb{P}T^{G-sps}//H\longrightarrow\mathbb{P}T^{G-ss}//H$$
is a bijective morphism, although $\iota$ may not be an isomorphism of schemes. The $\SL(U)$-polystability of cubic surfaces is a result of \cite[Subsection~1.14]{mumford-proj-var}.
\begin{proposition}
\label{prop:cubic-polystability}
    A strictly $\SL(U)$-semistable cubic surface $S\subset\mathbb{P}U$ is $\SL(U)$-polystable if and only if $S$ has $3A_2$ singularities. The GIT quotient $\vert\mathcal{O}_{\mathbb{P}U}(3)\vert^{\SL(U)-sss}//\SL(U)$ of the strictly \linebreak$\SL(U)$-semistable locus is a singleton.
\end{proposition}

Before stating the corresponding result for $\mathbb{P}(U^\vee\otimes V^\vee\otimes W^\vee)$, we prove a dimension equality. Recall from Definition~\ref{defn:R-R24} that $\mathcal{R}\subset\mathbb{P}(U^\vee\otimes V^\vee\otimes W^\vee)$ is the subvariety parametrising tensors $\varphi\in\mathbb{P}(U^\vee\otimes V^\vee\otimes W^\vee)$ containing a rank at most $1$ point; $\mathcal{R}\subset\mathbb{P}(U^\vee\otimes V^\vee\otimes W^\vee)$ is out by the principal ideal $(R_{24})\subset\mathbb{C}[U\otimes V\otimes W]_\bullet$.

\begin{proposition}
\label{prop:GIT-dimension-equalities}
    Let $F(n)\subset\mathcal{R}^{G-ss}$ be a stratum where $F\subset 3A_2$ or $F\subset 4A_1$. We have
    $$\dim\left(\overline{F(n)}//G\right)=\dim\left(\overline{\mathcal{S}(F)}//\SL(U)\right).$$
\end{proposition}
\begin{proof}    
    We recall the finite, integral extension $A_\bullet\subset\mathbb{C}[U\otimes V\otimes W]_\bullet^G$ from Definition~\ref{defn:almost-inv-ring}. This extension remains finite and integral after quotienting both sides by $(R_{24})$. By \cite{salmon-invariants}, the discriminant $\Delta\in\mathbb{C}[\Sym^3U]_\bullet^G$ is degree $32$ and we have
    \begin{align}
        \Delta&=I_{32}+\textnormal{a polynomial in $I_8,I_{16},I_{24}$}.\nonumber\\
        \Delta\circ\det&=J_{32}+\textnormal{a polynomial in $J_8,J_{16},J_{24}$}.\nonumber
    \end{align}
    If $\varphi\in\mathcal{R}$ then the surface $S_\varphi$ is singular, so the polynomial $R_{24}$ divides the pullback $\Delta\circ\det$ of the discriminant $\Delta$. There are no other relations between $R_{24}$ and the $J_{3d}$s, so we have isomorphisms
    $$A_\bullet/(R_{24})=A_\bullet/(R_{24},\Delta\circ\det)\cong\mathbb{C}[J_{24},J_{48},J_{72},J_{300}]\cong\mathbb{C}[I_8,I_{16},I_{24},I_{100}]=\mathbb{C}[\Sym^3U]_\bullet^{\SL(U)}/(\Delta).$$
    Taking projective spectra, we have 
    $$\mathcal{R}//G\xlongrightarrow{\eta\vert_{\mathcal{R}//G}}\Proj(A_\bullet/(R_{24}))\xlongrightarrow{\sim}\overline{\mathcal{S}(A_1)}//\SL(U).$$
    The map $\eta\vert_{\mathcal{R}//G}$ is finite. Writing $\eta_\mathcal{R}:=\eta\vert_{\mathcal{R}//G}$, the inverse image $\eta^{-1}_\mathcal{R}(\overline{\mathcal{S}(F)}//\SL(U))$ decomposes as the disjoint union of the connected components $\overline{F(n)}//G$ (where $F(n)\subset\mathcal{R}$). The restriction of $\eta_\mathcal{R}$ to each connected component $\overline{F(n)}//G$ is a finite, surjective morphism onto its image $\overline{\mathcal{S}(F)}//\SL(U)$. The desired equality then follows.
\end{proof}

\begin{theorem}
\label{thm:G-polystability}
    A tensor $\varphi\in\mathbb{P}(U^\vee\otimes V^\vee\otimes W^\vee)$ is $G$-polystable if and only if either $\varphi\in 3A_2(a)$ or $\varphi\in\overline{2A_2(b)}^{G-ss}=2A_2(b)\cup A_12A_2(c)\cup\Sigma_3$. 

    The quotient map $\pi_{UVW}$ (set-theoretically) contracts the stratum $\overline{A_2(a)}^{G-ss}$ onto the singleton $3A_2(a)//G$, and $\pi_{UVW}$ (set-theoretically) contracts the stratum $\overline{A_2(b)}^{G-ss}$ onto $\overline{2A_2(b)}//G$. There is a bijective morphism $\mathbb{P}^1\to\overline{A_2(b)}//G$.
\end{theorem}
\begin{proof}
    Recall that $\mathbb{P}(U^\vee\otimes V^\vee\otimes W^\vee)^{G-sss}$ decomposes as the disjoint union of its two connected components $\overline{A_2(a)}^{G-ss}$ and $\overline{A_2(b)}^{G-ss}$. We start with the connected component $\overline{A_2(a)}^{G-ss}$. By Proposition~\ref{prop:full-semistable-closures} the subvariety $3A_2(a)$ is closed inside $\mathbb{P}(U^\vee\otimes V^\vee\otimes W^\vee)^{G-ss}$ and $3A_2(a)$ is a $G$-orbit, so $3A_2(a)\subset \overline{A_2(a)}^{G-sps}$. To prove all other points $\varphi\in\overline{A_2(a)}^{G-ss}$ are not $G$-polystable it suffices to show that $\overline{A_2(a)}//G$ is a singleton. Because $\overline{A_2(a)}^{G-ss}$ is connected then we only need to show that its quotient $\overline{A_2(a)}^{G-ss}$ is zero-dimensional. This is an immediate consequence of Propositions~\ref{prop:cubic-polystability} and \ref{prop:GIT-dimension-equalities}.

    The other connected component $\overline{A_2(b)}^{G-ss}$ is more complicated. By \cite{ng}, a tensor \linebreak ${\varphi\in\mathbb{P}(U^\vee\otimes V^\vee\otimes W^\vee)}$ is contained in $\overline{2A_2(b)}^{G-ss}$ if and only if $\varphi$ is $G$-equivalent to a tensor of the form
    $$\varphi_{[s:t]}=\begin{pmatrix}
        0&u_3&u_2\\u_3&su_1+tu_2&u_0\\u_1&u_0&0
        \end{pmatrix},$$
    where $[s:t]\in\mathbb{P}^1$. If $[s:t]\in\{[1:0],[0:1]\}$ then $\varphi_{[s:t]}\in A_12A_2(c)$. If $\varphi=[1:1]$ then $\varphi\in\Sigma_3$. Otherwise, $\varphi_{[s:t]}\in 2A_2(b)$. Two tensors $\varphi_{[s:t]}$ and $\varphi_{[s':t']}$ are $G$-equivalent if and only if $[s:t]=[s':t']$ or $[s:t]=[t':s']$. This equivalence comes from the involution $\tau\in G$ that swaps $u_1$ and $u_2$, $v_0$ and $v_2$, and $w_0$ and $w_2$, which identifies $\tau\cdot\varphi_{[s:t]}=\varphi_{[t:s]}$. Geometrically, $\tau$ swaps the two $A_2$ singularities of $S_\varphi$ inside $\mathbb{P}U$ and swaps the two length $3$ points of $X_\varphi$ and $Y_\varphi$ inside $\mathbb{P}V$ and $\mathbb{P}W$, respectively.

    We first prove that all tensors $\varphi\in\overline{2A_2(b)}^{G-ss}$ are $G$-polystable, i.e. that $G\cdot\varphi$ is closed inside $\mathbb{P}(U^\vee\otimes V^\vee\otimes W^\vee)^{G-ss}$. By Proposition~\ref{prop:full-semistable-closures}, this is immediate if $\varphi\in A_12A_2(c)\cup\Sigma_3$. So suppose that $\varphi\in 2A_2(b)$ and let $\varphi'\in\overline{2A_2(b)}^{G-ss}$ be any other tensor. There are two possibilities:
    \begin{itemize}
        \item If $\varphi'\in 2A_2(b)\cup A_12A_2(c)$ then, by Proposition~\ref{prop:tensor-dims}, both $\varphi$ and $\varphi'$ have $1$-dimensional $G$-stabilisers and their $G$-orbits are irreducible of codimension $5$. So $G\cdot\varphi'\not\subset \overline{G\cdot\varphi}$ and $\varphi'\notin\overline{G\cdot\varphi}$.
        \item Otherwise $\varphi'\in\Sigma_3$. In this case, recalling the map $\eta$ from Definition~\ref{defn:almost-inv-ring}, we have $\eta(\pi_{UVW}(\varphi'))=[1:0:0:0:0:0]$ and $\eta(\pi_{UVW}(\varphi))=[\lambda:1:8:0:0:0]$ for some $\lambda\in\mathbb{C}$ (the cubic invariants $I_d$ all vanish at $[S_\sigma]$ when $\sigma\in\Sigma_3$, but not at $[S_\varphi]$). Then the points $\varphi$ and $\varphi'$ are separated by $G$-invariant functions so $\varphi'\notin \overline{G\cdot\varphi}$.
    \end{itemize}
    We now need to prove that $\varphi$ is not $G$-polystable whenever $\varphi\in\overline{A_2(b)}^{G-ss}\setminus\overline{2A_2(b)}^{G-ss}$. By Proposition~\ref{prop:tensor-dims}, the $G$-stabiliser of every element $\varphi\in\overline{A_2(b)}^{G-ss}\setminus\overline{2A_2(b)}^{G-ss}$ is finite. If a \mbox{$G$-semistable} point $\varphi\in\mathbb{P}(U^\vee\otimes V^\vee\otimes W^\vee)$ has a finite $G$-stabiliser then $\varphi$ is either $G$-stable (if $G\cdot\varphi$ is closed inside $\mathbb{P}(U^\vee\otimes V^\vee\otimes W^\vee)$) or $\varphi$ is strictly $G$-semistable (otherwise). By assumption, $\varphi$ is not $G$-stable, so $\varphi$ must also not be $G$-polystable. 

    The map $f:\mathbb{P}^1\to\overline{2A_2(b)}^{G-ss}$ sending $[s:t]\mapsto\varphi_{[s:t]}$ descends to a surjective morphism ${(\pi_{UVW}\circ f):\mathbb{P}^1\to\overline{2A_2(b)}//G}$. The map $\pi_{UVW}\circ f$ factors through the quotient of $\mathbb{P}^1$ by the subgroup $\{1,\tau\}\subset\PGL(2)$ that acts as $\tau\cdot[s:t]=[t:s]$. The induced map ${g:\mathbb{P}^1//\{1,\tau\}\to\overline{2A_2(b)}//G}$ is a bijective morphism, and $\mathbb{P}^1//\{1,\tau\}\cong\mathbb{P}(1,2)\cong\mathbb{P}^1$. 
\end{proof}

The three Theorems~\ref{thm:G-semistability} and \ref{thm:G-stability} and \ref{thm:G-polystability} combine to prove Theorem~\ref{thm:showtheorem}. The stratification of $\mathbb{P}(U^\vee\otimes V^\vee\otimes W^\vee)^{G-ss}$ by $\mathbb{S}$ and $\Sigma_3$ and the strata of type $F(n)$ given in Proposition~\ref{prop:full-semistable-closures} descends to a stratification of the quotient $\mathbb{P}(U^\vee\otimes V^\vee\otimes W^\vee)//G$ by $\mathbb{S}//G$ and $\Sigma_3//G$ and the strata of type $F(n)//G$. We have

\begin{proposition}
    The dimensions of the strata of $\mathbb{P}(U^\vee\otimes V^\vee\otimes W^\vee)$ are the following:
    \begin{itemize}
        \item $\dim(\mathbb{S}//G)=4$; 
        \item if $F\subset 4A_1$ then $\dim(F(n)//G)=4-\vert F\vert$; 
        \item if $A_2\subset F\subset 3A_2$ then either
        \begin{itemize}
            \item $F(n)=A_2(b),A_1A_2(g),2A_2(b)$, in which case $\dim (F(n)//G)=1$, or
            \item $\dim(F(n)//G)=0$; or
        \end{itemize}
        \item $\dim(\Sigma_3//G)=0$.
    \end{itemize}
\end{proposition}
\begin{proof}
    The equalities for the $G$-stable strata follow from using orbit-stabiliser and Proposition~\ref{prop:S(F)-dimensions}. All other equalities follow from Theorem~\ref{thm:G-polystability}. By Proposition~\ref{prop:full-semistable-closures}, the strata $A_2(b)$ and $A_1A_2(g)$ are the only strictly $G$-semistable strata of type $F(n)$ whose closures $\overline{F(n)}$ contain $2A_2(b)$. We then have surjective maps $A_2(b)//G\to\overline{2A_2(b)}//G$ and $A_1A_2(g)//G\to\left(\overline{2A_2(b)}\setminus\Sigma_3\right)//G$ which are bijections on $\mathbb{C}$-points, giving the desired equality of dimensions. 
\end{proof}

\begin{remark}
    The GIT quotient $\mathbb{P}(U^\vee\otimes V^\vee\otimes W^\vee)//G$ is essentially the last interesting moduli space of $n\times3\times3$ tensors. Specifically, fix $0<n<9$ and let $\varphi_{V\otimes W}:V\otimes W\to\mathbb{C}^n$ be a surjective map. The kernel map
    $$\varphi^\perp:\mathbb{C}^{9-n}\cong\ker(\varphi_{V\otimes W})\longrightarrow V^\vee\otimes W^\vee$$
    is adjoint to a $(9-n)\times3\times3$ tensor. The duality between injective $n\times3\times3$ tensors and injective $(9-n)\times3\times3$ tensors induces the standard isomorphism
    $$f:\Gr(n,V^\vee\otimes W^\vee)\xlongrightarrow{\sim}\Gr(9-n,V\otimes W)$$
    of Grassmannians. Because $\Gr(n,V^\vee\otimes W^\vee)\cong\mathbb{P}(\mathbb{C}^n\otimes V^\vee\otimes W^\vee)//(\SL(V)\times\SL(W))$, then $f$ induces an isomorphism of GIT quotients
    $$f_{VW}:\mathbb{P}(\mathbb{C}^n\otimes V^\vee\otimes W^\vee)//(\SL(n)\times\SL(V)\times\SL(W))\xlongrightarrow{\sim}\mathbb{P}(\mathbb{C}^{9-n}\otimes V\otimes W)//(\SL(n)\times\SL(V)\times\SL(W)).$$
    When $n=1,2,7,8$ there are finitely many $n\times3\times3$-tensors up to $\SL(n)\times\SL(3)\times\SL(3)$-equivalence. Ng described the $2$-dimensional moduli space for $n=3$ (and hence $n=6$) in \cite{ng-333}. 

    The geometry of the associated $5\times3\times3$ tensor $\varphi^\perp$ to a $4\times3\times3$ tensor was known to Segre in \cite{segre-threefolds}. When the associated cubic surface $S_\varphi$ is smooth, the intersection
    $$Z_\varphi:=\mathbb{P}\ker(\varphi_{V\otimes W})\cap(\mathbb{P}V\times\mathbb{P}W)\subset\mathbb{P}(V\otimes W)$$
    consists of $6$ points $Z_\varphi=\{x_1\otimes y_1,\ldots,x_6\otimes y_6\}$. The projections of $Z_\varphi$ onto $\mathbb{P}V$ and $\mathbb{P}W$ recover the length $6$ subschemes $X_\varphi$ and $Y_\varphi$. The cubic threefold $T_\varphi\subset\mathbb{P}\ker(\varphi_{V\otimes W})$ cut out by $\det(\varphi^\perp)=0$ has $6A_1$ singularities: one at each point $x_i\otimes y_i\in Z_\varphi$. Conversely, every cubic threefold with $6A_1$ singularities has a unique determinantal representation (up to $\GL(V)\times\GL(W)$-equivalence and the involution swapping $V$ and $W$). 
\end{remark}

\appendix   
\section{Hilbert-Mumford block forms}
\label{appendix:block-forms}

All computations were carried out in SageMath version 10.3 \cite{sagemath}.\par\bigskip

\noindent\textbf{$G$-instability:} Our computational implementation of the convex Hilbert-Mumford criterion for $G$-semistability of $\mathbb{P}(U^\vee\otimes V^\vee\otimes W^\vee)$ outputs the following $12$ maximal block forms (and their transposes) for $G$-instability. A tensor $\varphi\in\mathbb{P}(U^\vee\otimes V^\vee\otimes W^\vee)$ is $G$-unstable if and only if it is of one of the following block forms:
\begin{enumerate}[label=$(u\arabic*)$]
    \item associated to the weight vector $((1,1,1,-3),(0,0,0),(0,0,0))$:
    \[\left[\begin{array}{ccc|ccc|ccc|ccc}
     \ast&\ast&\ast&\ast&\ast&\ast&\ast&\ast&\ast&0&0&0  \\
     \ast&\ast&\ast&\ast&\ast&\ast&\ast&\ast&\ast&0&0&0\\
     \ast&\ast&\ast&\ast&\ast&\ast&\ast&\ast&\ast&0&0&0
    \end{array}\right]\]
    \item associated to the weight vector $((9, -3, -3, -3),(0, 0, 0), (4, 4, -8))$:
    
    \[\left[\begin{array}{ccc|ccc|ccc|ccc}
         \ast&\ast&\ast&\ast&\ast&\ast&\ast&\ast&\ast&\ast&\ast&\ast  \\
         \ast&\ast&\ast&\ast&\ast&\ast&\ast&\ast&\ast&\ast&\ast&\ast  \\
         \ast&\ast&\ast&0&0&0&0&0&0&0&0&0\\
    \end{array}\right]\]
    \item associated to the weight vector $((9, -3, -3, -3),(8, -4, -4), (8, -4, -4))$:
    \[\left[\begin{array}{ccc|ccc|ccc|ccc}
         \ast&\ast&\ast&\ast&\ast&\ast&\ast&\ast&\ast&\ast&\ast&\ast  \\
         \ast&\ast&\ast&\ast&0&0&\ast&0&0&\ast&0&0\\
         \ast&\ast&\ast&\ast&0&0&\ast&0&0&\ast&0&0\\
    \end{array}\right]\]
    \item associated to the weight vector $((9, 9, -3, -15),(8, -4, -4), (8, -4, -4))$:
    \[\left[\begin{array}{ccc|ccc|ccc|ccc}
         \ast&\ast&\ast&\ast&\ast&\ast&\ast&\ast&\ast&\ast&0&0  \\
         \ast&\ast&\ast&\ast&\ast&\ast&\ast&0&0&0&0&0\\
         \ast&\ast&\ast&\ast&\ast&\ast&\ast&0&0&0&0&0\\
    \end{array}\right]\]
    \item associated to the weight vector $((9, 9, -3, -15),(12, 0, -12), (4, 4, -8))$:
    \[\left[\begin{array}{ccc|ccc|ccc|ccc}
         \ast&\ast&\ast&\ast&\ast&\ast&\ast&\ast&0&\ast&0&0  \\
         \ast&\ast&\ast&\ast&\ast&\ast&\ast&\ast&0&\ast&0&0  \\
         \ast&\ast&0&\ast&\ast&0&\ast&0&0&0&0&0\\
    \end{array}\right]\]
\end{enumerate}
{\noindent \emph{Contained in form} $(u2)$:}
\begin{enumerate}[resume*]
    \item associated to the weight vector $((1, 1, 1, -3),(4, 0, -4), (4, 0, -4))$:
    \[\left[\begin{array}{ccc|ccc|ccc|ccc}
    \ast&\ast&\ast&\ast&\ast&\ast&\ast&\ast&\ast&\ast&\ast&0\\
    \ast&\ast&0&\ast&\ast&0&\ast&\ast&0&\ast&0&0\\
    \ast&0&0&\ast&0&0&\ast&0&0&0&0&0
    \end{array}\right]\]
    \item associated to the weight vector $((12, 0, -12), (16, 4, -20), (9, 9, -3, -15))$:
    \[\left[\begin{array}{ccc|ccc|ccc|ccc}
    \ast&\ast&\ast&\ast&\ast&\ast&\ast&\ast&\ast&\ast&\ast&0\\
    \ast&\ast&\ast&\ast&\ast&\ast&\ast&\ast&0&\ast&0&0\\
    \ast&0&0&\ast&0&0&0&0&0&0&0&0
    \end{array}\right]\]
    \item associated to the weight vector $((3, 3, -3, -3),(4, -2, -2), (6, 0, -6))$:
    \[\left[\begin{array}{ccc|ccc|ccc|ccc}
    \ast&\ast&\ast&\ast&\ast&\ast&\ast&\ast&\ast&\ast&\ast&\ast\\
    \ast&\ast&\ast&\ast&\ast&\ast&\ast&0&0&\ast&0&0\\
    \ast&0&0&\ast&0&0&0&0&0&0&0&0
    \end{array}\right]\]
    \item associated to the weight vector $((3, 3, 3, -9),(8, -4, -4), (4, 4, -8))$:
    \[\left[\begin{array}{ccc|ccc|ccc|ccc}
    \ast&\ast&\ast&\ast&\ast&\ast&\ast&\ast&\ast&\ast&0&0\\
    \ast&\ast&\ast&\ast&\ast&\ast&\ast&\ast&\ast&\ast&0&0\\
    \ast&0&0&\ast&0&0&\ast&0&0&0&0&0
    \end{array}\right]\]
\end{enumerate}
\bigskip\bigskip\bigskip 
{\noindent\emph{Contained in form} $(u4)$:}
\begin{enumerate}[resume*]
    \item associated to the weight vector $((3, 3, -3, -3),(0, 0, 0), (4, -2, -2))$:
    \[\left[\begin{array}{ccc|ccc|ccc|ccc}
         \ast&\ast&\ast&\ast&\ast&\ast&\ast&\ast&\ast&\ast&\ast&\ast  \\
         \ast&\ast&\ast&\ast&\ast&\ast&0&0&0&0&0&0\\
         \ast&\ast&\ast&\ast&\ast&\ast&0&0&0&0&0&0\\
    \end{array}\right]\]
    \item associated to the weight vector $((5, 1, 1, -7),(4, 0, -4), (4, 0, -4))$:
    \[\left[\begin{array}{ccc|ccc|ccc|ccc}
         \ast&\ast&\ast&\ast&\ast&\ast&\ast&\ast&\ast&\ast&0&0  \\
         \ast&\ast&\ast&\ast&\ast&0&\ast&\ast&0&0&0&0  \\
         \ast&\ast&0&\ast&0&0&\ast&0&0&0&0&0\\
    \end{array}\right].\]
\end{enumerate}
\noindent\emph{Contained in form} $(u5)$:
\begin{enumerate}[resume*]
    \item associated to the weight vector $(3, 3, -3, -3),(2, 2, -4), (2, 2, -4))$:
    \[\left[\begin{array}{ccc|ccc|ccc|ccc}
         \ast&\ast&\ast&\ast&\ast&\ast&\ast&\ast&0&\ast&\ast&0  \\
         \ast&\ast&\ast&\ast&\ast&\ast&\ast&\ast&0&\ast&\ast&0  \\
         \ast&\ast&0&\ast&\ast&0&0&0&0&0&0&0\\
    \end{array}\right]\]
\end{enumerate}

\noindent\textbf{Strict $G$-semistability:} Our computational implementation of the convex Hilbert-Mumford criterion for $G$-(semi)stability of $\mathbb{P}(U^\vee\otimes V^\vee\otimes W^\vee)$ outputs the following 7 maximal block forms (and their transposes) for strict $G$-semistability, in addition to the block forms for $G$-unstability. A $G$-semistable tensor $\varphi\in\mathbb{P}(U^\vee\otimes V^\vee\otimes W^\vee)$ is strictly $G$-semistable if and only if it is of one of the following block forms:
\begin{enumerate}[label=$(ns\arabic*)$]
    \item associated to the weight $((3, 0, 0, -3), (2, -1, -1), (1, 1, -2))$:
    \[\left[\begin{array}{ccc|ccc|ccc|ccc}
    \ast&\ast&\ast & \ast&\ast&\ast & \ast&\ast&\ast & \ast&0&0\\
    \ast&\ast&\ast & \ast&\ast&\ast & \ast&\ast&\ast & \ast&0&0\\
    \ast&\ast&\ast & \ast&0&0 & \ast&0&0 & 0&0&0
    \end{array}\right]\]
    \item associated to the weight $((1, 0, 0, -1), (1, 0, -1), (1, 0, -1))$:
    \[\left[\begin{array}{ccc|ccc|ccc|ccc}
    \ast&\ast&\ast & \ast&\ast&\ast & \ast&\ast&\ast & \ast&\ast&0\\
    \ast&\ast&\ast & \ast&\ast&0 & \ast&\ast&0 & \ast&0&0\\
    \ast&\ast&0 & \ast&0&0 & \ast&0&0 & 0&0&0
    \end{array}\right]\]
\end{enumerate}
\emph{Contained in form $(ns1)$:}
\begin{enumerate}[resume*]
    \item associated to the weight $((1, 1, 0, -2), (1, 0, -1), (1, 0, -1))$:
    \[\left[\begin{array}{ccc|ccc|ccc|ccc}
    \ast&\ast&\ast & \ast&\ast&\ast & \ast&\ast&\ast & \ast&0&0\\
    \ast&\ast&\ast & \ast&\ast&\ast & \ast&\ast&0 & 0&0&0\\
    \ast&\ast&0 & \ast&\ast&0 & \ast&0&0 & 0&0&0
    \end{array}\right]\]
    \item associated to the weight $((2, 0, -1, -1), (1, 0, -1), (1, 0, -1))$:
    \[\left[\begin{array}{ccc|ccc|ccc|ccc}
    \ast&\ast&\ast & \ast&\ast&\ast & \ast&\ast&0 & \ast&\ast&0\\
    \ast&\ast&\ast & \ast&\ast&0 & \ast&0&0 & \ast&0&0\\
    \ast&\ast&\ast & \ast&0&0 & 0&0&0 & 0&0&0
    \end{array}\right]\]
\end{enumerate}
\emph{Contained in unstable form $(u2)$:}
\begin{enumerate}[resume*]
    \item associated to the weight $((0, 0, 0, 0), (2, -1, -1), (1, 1, -2))$:
    \[\left[\begin{array}{ccc|ccc|ccc|ccc}
    \ast&\ast&\ast & \ast&\ast&\ast & \ast&\ast&\ast & \ast&\ast&\ast \\
    \ast&\ast&\ast & \ast&\ast&\ast & \ast&\ast&\ast & \ast&\ast&\ast \\
    \ast&0&0 & \ast&0&0 & \ast&0&0 & \ast&0&0
    \end{array}\right]\]
\end{enumerate}

\printbibliography[title={References}] 

\end{document}